\documentclass[12pt,reqno]{amsart}
\usepackage[utf8]{inputenc}
\usepackage[T1]{fontenc}
\usepackage[usenames, dvipsnames]{color}
\usepackage{ulem}

\usepackage{dsfont, amsfonts, amsmath, amssymb,amscd, stmaryrd, latexsym, amsthm, dsfont}
\usepackage[frenchb,english]{babel}
\usepackage{enumerate}
\usepackage{longtable}
\usepackage{geometry}
\usepackage{float}
\usepackage{tikz}
\usetikzlibrary{shapes,arrows}
\usepackage{float}
\usepackage{tikz}
\usepackage{xypic}
\usetikzlibrary{shapes,arrows}

\usepackage{pifont}
\usepackage{float}
\usepackage{tikz}
\usepackage{xypic}
\usetikzlibrary{shapes,arrows}

	\definecolor{darkcerulean}{rgb}{0.03, 0.27, 0.49}
	\definecolor{firebrick}{rgb}{0.7, 0.13, 0.13}
		\definecolor{forestgreen(traditional)}{rgb}{0.0, 0.27, 0.13}
		\definecolor{hanpurple}{rgb}{0.32, 0.09, 0.98}
	\definecolor{forestgreen(web)}{rgb}{0.13, 0.55, 0.13}
\newtheorem{theorem}{Theorem}[section]
\newtheorem{lemma}[theorem]{Lemma}
\newtheorem{proposition}[theorem]{Proposition}
\newtheorem{corollary}[theorem]{Corollary}

\theoremstyle{remark}
\newtheorem{remark}[theorem]{\bf Remark}

\newtheorem{definition}[theorem]{\bf Definition}

\usepackage[pagebackref]{hyperref}
\renewcommand*{\backref}[1]{}\renewcommand*{\backrefalt}[4]{\ifcase #1 (\tt not cited)\or (\tt cited on page~#2)\else (\tt cited on pages~#2)\fi}

\usepackage{hyperref}
\hypersetup{
	colorlinks=true,
	urlcolor=blue,
	citecolor=blue}
\def\NN{\mathbb{N}}
\def\RR{\mathds{R}}
\def\HH{I\!\! H}
\def\QQ{\mathbb{Q}}
\def\CC{\mathds{C}}
\def\ZZ{\mathbb{Z}}
\def\DD{\mathds{D}}
\def\OO{\mathcal{O}}
\def\kk{\mathds{k}}
\def\KK{\mathbb{K}}
\def\ho{\mathcal{H}_0^{\frac{h(d)}{2}}}
\def\LL{\mathbb{L}}
\def\L{\mathds{k}_2^{(2)}}
\def\M{\mathds{k}_2^{(1)}}
\def\k{\mathds{k}^{(*)}}
\def\l{\mathds{L}}

\def\kk{\mathds{k}}
\begin{document}
	
	\def\NN{\mathbb{N}}
	\def\RR{\mathds{R}}
	\def\HH{I\!\! H}
	\def\QQ{\mathbb{Q}}
	\def\CC{\mathds{C}}
	
	\def\FF{\mathbb{F}}
	\def\KK{\mathbb{K}}
	
	\def\ZZ{\mathbb{Z}}
	\def\DD{\mathds{D}}
	\def\OO{\mathcal{O}}
	\def\kk{\mathds{k}}
	\def\KK{\mathbb{K}}
	\def\ho{\mathcal{H}_0^{\frac{h(d)}{2}}}
	\def\LL{\mathbb{L}}
	\def\L{\mathds{k}_2^{(2)}}
	\def\M{\mathds{k}_2^{(1)}}
	\def\k{\mathds{k}^{(*)}}
	\def\l{\mathds{L}}
	\def\2r{\mathrm{rank_2}}
		\def\rg{\mathrm{rank}}
	
	\def\C{\mathrm{C}}
	\def\Y{\mathbf{Y}}
	\def\V{\mathbf{V}}
	\def\vep{\varepsilon}

	\selectlanguage{english}

		\title[Greenberg's conjecture and    Iwasawa module  I]{Greenberg's conjecture and    Iwasawa module  of Real biquadratic fields I}

	\author[M. M. Chems-Eddin]{Mohamed Mahmoud Chems-Eddin}
	\address{Mohamed Mahmoud CHEMS-EDDIN: Department of Mathematics, Faculty of Sciences Dhar El Mahraz,
		Sidi Mohamed Ben Abdellah University, Fez,  Morocco}
	\email{2m.chemseddin@gmail.com}



\begin{abstract}
	The main aim of this paper is to investigate   Greenberg's conjecture for real biquadratic fields. More precisely, we propose the following problem: 
	
	\begin{center}
		What are real biquadratic number fields $k$ such that   $\rg(A(k_\infty))=\rg(A(k_1))$ ?
	\end{center} 
	where   $A(k_\infty)$  is     the $2$-Iwasawa module of $k$ and $A(k_1)$ is the $2$-class group of $k_1$ the first layer of the cyclotomic $\ZZ_2$-extension of $k$. Moreover, we give several families of real biquadratic fields  $k$ such that $A(k_\infty)$ is trivial or isomorphic to $\ZZ/2^{n} \ZZ$ or   $\ZZ/2\ZZ \times\ZZ/2^n \ZZ$, where $n$  
	 is a   given positive integer.  The reader can also find   some results concerning the $2$-rank of the class group of certain real  triquadratic fields.

\end{abstract}
	
	\selectlanguage{english}
	
	\maketitle

	\section{Introduction}

 Let  $k$ be a   number field and $\ell$ a prime number. Let $A_\ell(k)$  or simply $A(k)$ when $\ell=2$  (resp.   $E_k$) denote the $\ell$-class group (resp.   the unit group) of $k$.
A   $\ZZ_\ell$-extension of $k$ is an infinite extension of $k$ denoted by  $k_\infty$ such that $\mathrm{Gal}(k_\infty/k)\simeq \ZZ_\ell$, where $\ZZ_\ell$ is the ring of $\ell$-adic numbers. For each $n\geq 1$, the extension $k_\infty/k$ contains a unique field denoted by $k_n$ and called the $n$th layer of the cyclotomic $\ZZ_\ell$-extension of $k$   of degree $\ell^n$. Furthermore, we have:
$$k=k_0\subset k_1 \subset k_2 \subset\cdots \subset k_n \subset \cdots  \subset k_\infty=\bigcup_{n\geq 0} k_n.$$
 In particular, for an odd prime number $\ell$, let $ \QQ_{\ell, n}$   be the unique real subfield  of the cyclotomic field  $\QQ(\zeta_{\ell^{n+1}}) $ of degree $\ell^n$ over $\QQ$,  
 and for $\ell=2$,   let $  \QQ_{2, n}$ be the field $\QQ(2\cos( {2\pi}/{2^{n+2}}))$, for all $n\geq 1$.
 Then $k_\infty= \bigcup k_n$, where $k_n= k\QQ_{\ell, n}$,  is called the   cyclotomic $\ZZ_\ell$-extension of $k$.  
  The inverse limit $A(k_\infty):=\varprojlim A_\ell(k_n)$
with respect to the norm maps is called the Iwasawa module for $k_\infty/k$. A spectacular result due to Iwasawa, affirms that  there exist integers $\lambda_k$,  $\mu_k\geq 0$ and  $\nu_k$, all independent of $n$, and  an integer $n_0$ such that:
\begin{eqnarray}\label{iwasawa}h_\ell(k_n)=\ell^{\lambda_k n+\mu_k \ell^n+\nu_k},\end{eqnarray}
for all $n\geq n_0$.  Here $h_\ell(k)$ denote the $\ell$-class number of a number field $k$. The integers $\lambda_k$,  $\mu_k$ and  $\nu_k$ are called the Iwasawa invariants of $k_\infty/k$
 (cf. \cite{iwasawa59} and \cite{washington1997introduction} for more details).
 
  In 1976,
  Greenberg conjectured that the invariants $\mu$ and $\lambda$ must be equal to $0$ for cyclotomic $\mathbb Z_\ell$-extension of totally real number
  fields (cf. \cite{Greenberg}).
  
   It was further proved by Ferrero and Washington    (cf. \cite{FerreroWashington}) that the $\mu$-invariant always vanishes
for the cyclotomic $\mathbb Z_\ell$-extension when the number field is abelian over the field $\mathbb{Q}$ of rational numbers. In other words, this  conjecture 
is equivalent to $h_\ell(k_n)$ being uniformly bounded.
   Greenberg's conjecture is still open, with partial progress made by considering particular values of $\ell$ and specific families of number fields (especially the real quadratic fields), for example, we refer the reader to   \cite{7,chemskatharina,fukuda,1FukuKom,kumakawa2,8,PLorenzo,mizusawa2,mizusawa3,mizusawa4,mizusawa5,mouhib,mouhib-mova}.
  Moreover, very recently some authors have taken interest  in the investigation of Greenberg's conjecture  for some particular families of real biquadratic fields. In fact, 
\cite{chems24,Elmahi1,Elmahi2,Elmahi3}, the   authors Chems-Eddin, El Mahi and Ziane investigated this conjecture for  biquadratic fields of the form   $ \QQ(\sqrt{pq_1}, \sqrt{q_1q_2})$
where $p\equiv 5\pmod 8$ and $q_1\equiv q_2\equiv3\pmod 4$  are three prime numbers such that $\ \left( \frac{p}{q_1} \right)=\ \left( \frac{p}{q_2} \right)$.
 Furthermore, in   \cite{Laxmi}
Laxmi and      Saikia investigated this conjecture for biquadratic fields of the form   $ \QQ(\sqrt{p}, \sqrt{r})$ with $p \equiv 9 \mod{16}, \ r \equiv 3 \mod 4$ are two prime numbers such that $\ \left( \frac{p}{r} \right) = -1$, and $\left(\frac{2}{p}\right)_{4} = -1$. Consider the following definition.
	\begin{definition}
	A number field $k$  is said QO-field if it is a quadratic extension of certain number field $k'$ whose class number is odd. We shall call $k'$ a base field of the QO-field $k$ and that $k/k'$ is a QO-extension. An extension of  QO-fields is an extension of number fields $L/M$ such that $L$ and $M$ are QO-fields.	\end{definition}

In this paper, we  propose and investigate  the following problems.

	\begin{center}
		 {\bf  Problems: }
	\end{center}
Let $K$ be a real  biquadratic number field such that $K_1/K$ is a ramified  extension of QO-fields, where $K_1=K(\sqrt{2})$ is the first layer of the cyclotomic $\ZZ_2$-extension of $K$. Consider the following problems:

	 	 \begin{center}
	 	\noindent  {  Problem 1}: 	What are the real biquadratic fields $K$  such that  
	 	 	$$\rg(A(K_\infty))\leq 2   \text{ and   }    \rg( A(K_\infty))=\rg(A(K))?$$
	 	 \end{center}
 \begin{center}
 	\noindent  { Problem 2}: 	What is the structure of $A(K_\infty)$?$\qquad\qquad\qquad\quad$
 
 \end{center}
\bigskip
Here 	we provide the first part of the answer to these problems, and a second part is  prepared in a another paper. We note that the analogue of Problem 1  for real  quadratic fields have been partially investigated, without being directly formulated, in the some of the previously cited references, particularly in  \cite{7}.
 	 \bigskip
  	
		Let $K$ be a real biquadratic field. To answer the first problem   we will implement the following strategy: 
	First we eliminate all fields $K$  such that  $K$ and $K_1$ are not QO-fields or $K_1/K$ is unramified over $2$. In the remaining list $($see the forms $A) -F)$ below$)$ we shall determine 
	all fields $K$ such that $\rg(A(K))\leq 2$ and finally to get our result we  investigate the equality $\rg(A(K))= \rg(A(K_1))$.\\
	
	We have the following  lemma which is due to Conner and Hurrelbrink (cf. \cite[Corollary 18.4]{connor88}).

	\begin{lemma}[\cite{{connor88}}, Corollary 18.4]\label{realQuad} Let $F$ be a real quadratic field. The class number of $F$ is odd   if and only if it takes one of the following forms:
		\begin{enumerate}[$    1)$]
			\item $F=\QQ(\sqrt{2})$, $\QQ(\sqrt{p})$ where $p$  is a   prime number congruent to $1\pmod 4$. 
			\item $F=\QQ(\sqrt{q})$, $\QQ(\sqrt{2q})$ or $\QQ(\sqrt{q_1q_2})$  where  $q$, $q_1$, $q_2$ are prime  numbers congruent to $3\pmod 4$.
		\end{enumerate}
	\end{lemma}
	
	 Furthermore, we have the following lemma that combines results by Conner and Hurrelbrink (cf. \cite[Corollaries 21.2, 21.4 and Proposition 21.5]{connor88}) and Ku{\v{c}}era 
	  (cf. \cite[Theorem 1]{kuvcera1995parity}).

	\begin{lemma}[\cite{{connor88}}, Corollaries 21.2, 21.4 and Proposition 21.5]\label{realBiquad} Let $K$ be a real biquadratic number field. The class number of $K$ is odd if and only if   it takes one of the following forms:
		\begin{enumerate}[$    1)$]
			\item $K=\QQ(\sqrt{2},\sqrt{q})$ or  $\QQ(\sqrt{q_1},\sqrt{q_2})$, 
			
			\item $K=\QQ(\sqrt{q},\sqrt{p})$ or $\QQ(\sqrt{2q},\sqrt{p})$ with $\left(\frac{p}{q}\right)=-1$ or $p\equiv5 \pmod 8$,
			\item	$K=\QQ(\sqrt{2},\sqrt{q_1q_2})$ with   $q_1\equiv3 \pmod 8$ or $q_2\equiv3 \pmod 8$, 
			\item $K=\QQ(\sqrt{p},\sqrt{q_1q_2})$  with $\left(\frac{q_1}{p}\right)=-1$ or $\left(\frac{q_2}{p}\right)=-1$,
			\item $K=\QQ(\sqrt{p_1},\sqrt{p_2})$ with $\left(\frac{p_1}{p_2}\right)=-1$ or $[\left(\frac{p_1}{p_2}\right)=1$ and $\left(\frac{p_1}{p_2}\right)_4\not=\left(\frac{p_2}{p_1}\right)_4]$.
		\end{enumerate}
		Here, $p$, $p_1$, $p_2$  denote  primes numbers congruent to $1\pmod 4$ and $q$, $q_1$, $q_2$ denote  primes numbers congruent to $3\pmod 4$.
	\end{lemma}

	\bigskip
	Let $K$ be a real  biquadratic number field such that $K_1/K$ is an extension of   QO-fields that is ramified over $2$.  Then, based on Lemmas \ref{realQuad} and \ref{realBiquad}, we check that $K$ takes one of the following six forms:
	\begin{enumerate}[$    A)$]
		\item   $K=\QQ(\sqrt{q},\sqrt{d})$ where $d > 1$ is an odd  positive  square-free integer that is not divisible by $q$.

			\item  $K=\QQ(\sqrt{2q},\sqrt{d})$ where    $q\equiv3 \pmod 4$   and $d\equiv 1 \pmod 4$ is a  positive square-free integer.

		\item  $K=\QQ(\sqrt{q_1q_2},\sqrt{d})$ where  $q_1\equiv3 \pmod 4$,  $q_2\equiv3 \pmod 8$ are two prime numbers and $d\equiv 1\pmod 4$ is a positive    square-free integer that is not divisible by   $q_1q_2$.

		\item  $K=\QQ(\sqrt{q_1q_2},\sqrt{d})$ or $\QQ(\sqrt{q_1q_2},\sqrt{2d})$,  where  $q_1\equiv7 \pmod 8$,  $q_2\equiv3 \pmod 8$ are two prime numbers and $d\equiv 3\pmod 4$ is a positive        square-free integer that is not divisible by   $q_1q_2$.

\item  $K=\QQ(\sqrt{q_1q_2},\sqrt{d})$ or $\QQ(\sqrt{q_1q_2},\sqrt{2d})$, where $q_1\equiv3 \pmod 8$,  $q_2\equiv3 \pmod 8$ are two prime numbers  and    $d\equiv 3 \pmod 4$ is a  positive square-free integer  that is not divisible by   $q_1q_2$.

		\item  $K=\QQ(\sqrt{p},\sqrt{d})$ where $p $ and $d$  satisfy one of the following conditions:
		\begin{enumerate}[$    a)$]
			
			\item  $p\equiv1 \pmod 4$   and $d=q_1q_2$ for two prime numbers $q_1\equiv q_2\equiv 3\pmod 4$  with   $\left(\frac{q_2}{p}\right)=-1$,
			\item  $p\equiv1 \pmod 4$   and $d=q_1q_2$ for two prime numbers $q_1\equiv 3\pmod 4$ and   $q_2\equiv 3 \pmod 8$,
			\item $p\equiv1 \pmod 4$   and $d\equiv 3 \pmod 4$ are two prime numbers such that $\left(\frac{p}{d}\right)=-1$ or $p\equiv5 \pmod 8$,
			\item $p\equiv1 \pmod 4$   and $d\equiv 1 \pmod 4$ are two prime numbers such that$\left(\frac{p}{d}\right)=-1$ or $[\left(\frac{p}{d}\right)=1$ and $\left(\frac{p}{d}\right)_4\not=\left(\frac{d}{p}\right)_4]$.
		\end{enumerate}

	\end{enumerate}
	Here, $p$ denotes a prime  number  congruent to $1\pmod 4$ and $q$, $q_1$, $q_2$ denote three primes numbers congruent to $3\pmod 4$. Let $r$ and $s$ be two prime numbers. Let us name $L$ any   real biquadratic field   of the form  

\begin{center}
	 {$ \displaystyle L:=\QQ(\sqrt{\delta_0},\sqrt{r})\text{ or }\QQ(\sqrt{\delta_0},\sqrt{rs}) \text{ where }   r\equiv 1\pmod8   \text{ and }  \left(\frac{\delta_0}{r}\right)=1$}
\end{center}
 
 where $\delta_0\in\{q, q_1q_2\}	$.  
We keep the above notations for the rest of the paper. The main theorem of this paper is stated as below.

	\begin{theorem}[{\bf  The Main Theorem}]\label{maintheorem}	Let $K$ be a real biquadratic number field that is of the form $A)$, $B)$ or $C)$ with $K\not=L$.  
		Then  $\rg(A(K_\infty))\leq 2$  and  $ \rg(A(K_\infty))=\rg( A(K))$ if and only if $K$ takes one of the following forms:
		\begin{enumerate}[$1)$]
			\item $K=\QQ(\sqrt{q},\sqrt{r})$,  where  $q\equiv 3\pmod 4$ and $r$ are two prime numbers  such that we have one of the following congruence conditions:
			\begin{enumerate}[$\C1:$]
				\item   $r\equiv 3$ or $5\pmod 8$,
				   
				\item   $r\equiv 7\pmod 8$ and  $q\equiv 3\pmod 8$.
	\end{enumerate}
	$\textbf{In this case, we have: } A(K_\infty)  =0.$
		
			\item $K=\QQ(\sqrt{q},\sqrt{rs})$, where  $q\equiv 3\pmod 4$, $r$ and $s$ are prime numbers such that $ \left(\frac{q}{s}\right)=\left(\frac{q}{r}\right)=-1$ and we have one of the following congruence conditions:
			\begin{enumerate}[$\C1:$]
				\item $r\equiv     5\pmod 8$ and  $s\equiv     3\pmod 8$, 
				
				\item $r\equiv     3$ or $5\pmod 8$,  $s\equiv    7\pmod 8$  and $q\equiv 3\pmod 8$.
			\end{enumerate}
	 	$\textbf{In this case, we have: } \rg(A(K_\infty))  =1.$

	\item $K=\QQ(\sqrt{q},\sqrt{rs})$, where  $q\equiv 3\pmod 4$, $r$ and $s$ are prime numbers such that $ \left(\frac{q}{r}\right)=-\left(\frac{q}{s}\right)=-1$ and we have one of the following congruence conditions:
	\begin{enumerate}[$\C1:$]
		\item $r\equiv    5 \pmod 8$ and  $s\equiv   5$ or $3\pmod 8$,
		\item   $r\equiv     3 \pmod 8$,  $s\equiv    3\pmod 8$  and $q\equiv 7\pmod 8$, 
		\item $r\equiv    3 \pmod 8$ and  $s\equiv   5\pmod 8$,
			\item    $r\equiv   7\pmod 8$,  $s\equiv  3$ or $5\pmod 8$ and   $q\equiv 3\pmod 8$,
		
		\item      $r\equiv  3$ or $5\pmod 8$, $s\equiv   7\pmod 8$ and $q\equiv 3\pmod 8$.	
		
	\end{enumerate}
	$\textbf{In this case, we have: } \rg(A(K_\infty))  =1.$

		\item  $K=\QQ(\sqrt{q},\sqrt{rs})$, where  $q\equiv 3\pmod 4$, $r$ and $s$ are prime numbers such that $ \left(\frac{q}{r}\right)=\left(\frac{q}{s}\right)=1$  and we have one of the following congruence conditions:
	\begin{enumerate}[$\C1:$]
		\item   $r\equiv    7\pmod 8$, $s\equiv     3$ or $5\pmod 8$,  and $q\equiv 3\pmod 8$,
	
	\item $r\equiv     3 \pmod 8$ and  $s\equiv    5\pmod 8$.
	\end{enumerate}
	$\textbf{In this case, we have: } \rg(A(K_\infty))  =1.$	
	 
			\item  $K=\QQ(\sqrt{q},\sqrt{rs})$, where  $q\equiv 3\pmod 4$, $r$ and $s$ are prime numbers such that $ \left(\frac{q}{r}\right)=\left(\frac{q}{s}\right)=1$  and $r\equiv s\equiv 3$ or $5\pmod 8$. 
			
			\noindent$\textbf{In this case, we have: } \rg(A(K_\infty))  =2.$

		\item $K=\QQ(\sqrt{q},\sqrt{rst})$, where  $q\equiv 3\pmod 4$, $r\equiv 3\pmod 8$, $s\equiv 5\pmod 8$ and $t$ are prime numbers such that $ \left(\frac{q}{r}\right)=\left(\frac{q}{s}\right)=-\left(\frac{q}{t}\right)=1$ and we have one of the following conditions:
			\begin{enumerate}[$\C1:$]
			\item   $t\equiv          5\pmod 8$,    
			
			\item  $t\equiv          3\pmod 8$ and  $q\equiv          7\pmod 8$,  
			
			\item $t\equiv          7\pmod 8$ and  $q\equiv          3\pmod 8$.
		\end{enumerate}
	 	
	\noindent$\textbf{In this case, we have: } \rg(A(K_\infty))  =2.$

	\item $K=\QQ(\sqrt{q},\sqrt{rst})$, where  $q\equiv 3\pmod 4$, $r\equiv 5\pmod 8$, $s$ and $t$ are prime numbers such that $ \left(\frac{q}{r}\right)=-\left(\frac{q}{s}\right)=-\left(\frac{q}{t}\right)=1$ and we have one of the following conditions:
	\begin{enumerate}[$\C1:$]
		\item   $s\equiv 3\pmod 8$ and     $t\equiv5\pmod 8$, 
		
		\item  $s\equiv 7\pmod 8$, $t\equiv          3\pmod 8$ and  $q\equiv          3\pmod 8$.
	\end{enumerate}	 
	 	\noindent$\textbf{In this case, we have: } \rg(A(K_\infty))  =2.$
	 	
	 	\item $K=\QQ(\sqrt{q},\sqrt{rst})$, where  $q\equiv 3\pmod 4$, $r\equiv 3\pmod 8$, $s$ and $t$ are prime numbers such that $ \left(\frac{q}{r}\right)=-\left(\frac{q}{s}\right)=-\left(\frac{q}{t}\right)=1$ and we have one of the following conditions:
	 \begin{enumerate}[$\C1:$]
	 	\item   $s\equiv 3\pmod 8$ and     $t\equiv5\pmod 8$ and  $q\equiv          7\pmod 8$, 
	 	
	 	\item  $s\equiv 5\pmod 8$, $t\equiv          7\pmod 8$ and  $q\equiv          3\pmod 8$.
	 \end{enumerate}	 
	 \noindent$\textbf{In this case, we have: } \rg(A(K_\infty))  =2.$

	 \item  $K=\QQ(\sqrt{q},\sqrt{rst})$, where $q\equiv 3\pmod 8$, $r\equiv 3\pmod 8$, $s\equiv 5\pmod 8$ and  $t\equiv 7\pmod 8$   
	 are prime numbers such that $ \left(\frac{q}{r}\right)=\left(\frac{q}{s}\right)= \left(\frac{q}{t}\right)=-1$.

	 \noindent$\textbf{In this case, we have: } \rg(A(K_\infty))  =2.$

	 \item $K=\QQ(\sqrt{2q},\sqrt{r})$, where $q\equiv 3\pmod 4$ and $r\equiv     5 \pmod 8$     are prime numbers.
	 
	 \noindent$\textbf{In this case, we have: } \rg(A(K_\infty))  =0.$

	 	\item $K=\QQ(\sqrt{2q},\sqrt{rs})$, where $q\equiv 3\pmod 8$, $r\equiv     3 \pmod 8$ and  $s\equiv    7\pmod 8$   are prime numbers such that $ \left(\frac{q}{r}\right)=\left(\frac{q}{s}\right)=-1$.
	 		 
	 	  \noindent$\textbf{In this case, we have: } \rg(A(K_\infty))  =1.$ 
	 		 
	 		\item $K=\QQ(\sqrt{2q},\sqrt{rs})$,	 where $q\equiv 3\pmod 4$, $r$ and $s$    are prime numbers such that $ \left(\frac{q}{r}\right)=-\left(\frac{q}{s}\right)=-1$ and satisfy one of the following conditions:
	 		\begin{enumerate}[$\C1:$]
	 			\item   $r\equiv      s\equiv    5\pmod 8$,   
	 			
	 			\item $r\equiv s \equiv    3 \pmod 8$ and   $q\equiv 7\pmod 8$,

	 				\item  $r\equiv7 \pmod 8$, $s \equiv    3 \pmod 8$ and   $q\equiv 3\pmod 8$,
	 				
	 				\item  $r\equiv3 \pmod 8$, $s \equiv    7 \pmod 8$ and   $q\equiv 3\pmod 8$.
	 		\end{enumerate}
	 		
	 		  \noindent$\textbf{In this case, we have: } \rg(A(K_\infty))  =1.$

	 	\item $K=\QQ(\sqrt{2q},\sqrt{rs})$, where $q\equiv 3\pmod 8$, $r\equiv     7 \pmod 8$ and  $s\equiv    3\pmod 8$   are prime numbers such that $ \left(\frac{q}{r}\right)=\left(\frac{q}{s}\right)=1$.

	 	\noindent$\textbf{In this case, we have: } \rg(A(K_\infty))  =1.$

	 		\item $K=\QQ(\sqrt{2q},\sqrt{rs})$, where $q\equiv 3\pmod 4$, $r\equiv s \equiv   3$ or $5 \pmod 8$     are prime numbers such that $ \left(\frac{q}{r}\right)=\left(\frac{q}{s}\right)=-1$.

	 	\noindent$\textbf{In this case, we have: } \rg(A(K_\infty))  =2.$

	 	\item   $K=\QQ(\sqrt{2q},\sqrt{rst})$, where $q\equiv 3\pmod 4$, $r\equiv 3 \pmod 8 $, $s\equiv5 \pmod 8$, $t$     are prime numbers satisfying one of the following conditions:
	  \begin{enumerate}[$\C1:$]
	 	\item   $t\equiv   3\pmod 8$, $q\equiv 7\pmod 8$,  $ -\left(\frac{q}{r}\right)=-\left(\frac{q}{s}\right)=\left(\frac{q}{t}\right)=1$,
	 	
	 	\item  $t\equiv   7\pmod 8$, $q\equiv 3\pmod 8$,  $ \left(\frac{q}{r}\right)=\left(\frac{q}{s}\right)=\left(\frac{q}{t}\right)=-1$.
	 \end{enumerate}

	 \noindent$\textbf{In this case, we have: } \rg(A(K_\infty))  =2.$

	  \item   $K=\QQ(\sqrt{2q},\sqrt{rst})$, where $q\equiv 3\pmod 4$, $r$, $s$, $t$     are prime numbers satisfying one of the following conditions:
	  \begin{enumerate}[$\C1:$]
	  	\item  $r\equiv   3\pmod 8$, $s\equiv   3\pmod 8$,  $t\equiv  5\pmod 8$, $q\equiv 7\pmod 8$,  $ -\left(\frac{q}{r}\right)=\left(\frac{q}{s}\right)=\left(\frac{q}{t}\right)=1$,
	  	
	  	\item   $r\equiv   3\pmod 8$, $s\equiv   7\pmod 8$,  $t\equiv   5\pmod 8$, $q\equiv 3\pmod 8$,  $ -\left(\frac{q}{r}\right)=-\left(\frac{q}{s}\right)=\left(\frac{q}{t}\right)=1$,
	  	
		  	\item   $r\equiv   5\pmod 8$, $s\equiv   7\pmod 8$,  $t\equiv   3\pmod 8$, $q\equiv 3\pmod 8$,  $ -\left(\frac{q}{r}\right)=-\left(\frac{q}{s}\right)=\left(\frac{q}{t}\right)=1$, 
		  	
		  		\item   $r\equiv   3\pmod 8$, $s\equiv   7\pmod 8$,  $t\equiv   5\pmod 8$, $q\equiv 3\pmod 8$,  $ \left(\frac{q}{r}\right)=-\left(\frac{q}{s}\right)=\left(\frac{q}{t}\right)=1$. 	
	  	
	  \end{enumerate}

	  \noindent$\textbf{In this case, we have: } \rg(A(K_\infty))  =2.$

	 	\item $K=\QQ(\sqrt{2q},\sqrt{qr})$, where $q\equiv 3\pmod 4$ and $r$     are prime numbers  such that we have one of the following congruence conditions:
	 	\begin{enumerate}[$\C1:$]
	 		\item   $r\equiv   3\pmod 8$,  
	 		
	 		\item $r\equiv     7 \pmod 8$ and   $q\equiv 3\pmod 8$.
	 	\end{enumerate}

	 	\noindent$\textbf{In this case, we have: } \rg(A(K_\infty))  =0.$

	 \item $K=\QQ(\sqrt{2q},\sqrt{qrs})$, where  $q\equiv 3\pmod 4$, $r$ and $s$ are prime numbers such that $ \left(\frac{q}{r}\right)=\left(\frac{q}{s}\right)=-1$  and we have one of the following congruence conditions:
	 \begin{enumerate}[$\C1:$]
	 	\item   $r\equiv    5\pmod 8$, $s\equiv     3 \pmod 8$,  
	 	
	 	\item $r\equiv     5 \pmod 8$ and  $s\equiv    7\pmod 8$ and $q\equiv 3\pmod 8$.
	 \end{enumerate}
	
	\noindent $\textbf{In this case, we have: } \rg(A(K_\infty))  =1.$ 
	 			
	 	 \item  $K=\QQ(\sqrt{2q},\sqrt{qrs})$, where  $q\equiv 3\pmod 4$, $r$ and $s$ are prime numbers such that $ \left(\frac{q}{r}\right)=-\left(\frac{q}{s}\right)=-1$  and we have one of the following congruence conditions:
	 	 \begin{enumerate}[$\C1:$]
	 	 	\item   $r\equiv    5\pmod 8$, $s\equiv     3 \pmod 8$,  
	 	 	
	 	 	\item $r\equiv     3 \pmod 8$ and  $s\equiv    5\pmod 8$,
	 	 	
	 	 		\item   $r\equiv     7 \pmod 8$,  $s\equiv    5\pmod 8$ and $q\equiv 3\pmod 8$,
	 	 	\item  $r\equiv     5\pmod 8$,  $s\equiv    7\pmod 8$ and $q\equiv 3\pmod 8$.
	 	 	
	 	 \end{enumerate}
	 	
	 		\noindent $\textbf{In this case, we have: } \rg(A(K_\infty))  =1.$ 
	 	
	 	\item  $K=\QQ(\sqrt{2q},\sqrt{qrs})$, where  $q\equiv 3\pmod 4$, $r$ and $s$ are prime numbers such that $ \left(\frac{q}{r}\right)=\left(\frac{q}{s}\right)=1$  and we have one of the following congruence conditions:
	 	\begin{enumerate}[$\C1:$]
	 		\item   $r\equiv    7\pmod 8$, $s\equiv     5 \pmod 8$ and $q\equiv 3\pmod 8$,  
	 		
	 		\item $r\equiv     3 \pmod 8$ and  $s\equiv    5\pmod 8$.
	 	\end{enumerate}
	 	
	 	\noindent $\textbf{In this case, we have: } \rg(A(K_\infty))  =1.$

	   	\item $K=\QQ(\sqrt{2q},\sqrt{qrst})$, where $q\equiv 3\pmod 4$, $r\equiv 3 \pmod 8 $, $s\equiv5 \pmod 8$, $t\equiv5 \pmod 8$     are prime numbers such that 
	   	$ -\left(\frac{q}{r}\right)=-\left(\frac{q}{s}\right)=\left(\frac{q}{t}\right)=1$.

	   	\noindent$\textbf{In this case, we have: } \rg(A(K_\infty))  =2.$

  	\item  $K=\QQ(\sqrt{2q},\sqrt{qrst})$, where $q\equiv 3\pmod 4$, $r\equiv 5 \pmod 8 $, $s\equiv3 \pmod 8$, $t\equiv5 \pmod 8$     are prime numbers such that 
  $ -\left(\frac{q}{r}\right)=\left(\frac{q}{s}\right)=\left(\frac{q}{t}\right)=1$.

  \noindent$\textbf{In this case, we have: } \rg(A(K_\infty))  =2.$

			\item   $K=\QQ(\sqrt{q_1q_2},\sqrt{\delta r})$ where   $q_1\equiv3 \pmod 4$,  $q_2\equiv3 \pmod 8$ and $r$ are prime numbers, and   $\delta\in  \{1,q_1,q_2\}$ with $\delta r\equiv 1\pmod 4$ and such that we have one of the following conditions:
			\begin{enumerate}[$\C1:$]
				\item $r\equiv 3\pmod 8$,
				\item $r\equiv 5\pmod 8$, $ q_1\equiv 3\pmod 8$ and  $\left(\frac{q_1q_2}{ r}\right)=-1$,
				\item $r\equiv 5\pmod 8$, $ q_1\equiv 7\pmod 8$ and  $\left(\frac{q_1}{ r}\right)=-1$,

				\item  $r\equiv  7\pmod 8$,   $q_1\equiv 3\pmod 8$.
				
			\end{enumerate}
							$\textbf{In this case, we have: } \rg(A(K_\infty))  =0.$

		 \item   $K=\QQ(\sqrt{q_1q_2},\sqrt{ r})$ where   $q_1\equiv3 \pmod 4$,  $q_2\equiv3 \pmod 8$ and   
			$r\equiv 5\pmod 8$ are prime numbers such that   $\left(\frac{q_1q_2}{ r}\right)=\left(\frac{q_1}{ r}\right)=1$.  

		 	\noindent$\textbf{In this case, we have: } \rg(A(K_\infty))  =1.$

	\item   $K=\QQ(\sqrt{q_1q_2},\sqrt{\delta rs})$ where     $r$ and $s$ are prime numbers such that $ \left(\frac{q_1q_2}{s}\right)=\left(\frac{q_1q_2}{r}\right)=-1$, and   $\delta\in  \{1,q_1,q_2\}$ with $\delta rs\equiv 1\pmod 4$ and such that we have one of the following conditions:
	\begin{enumerate}[$\C1:$]
			\item  $r\equiv 5\pmod 8$ and $s\equiv 3\pmod 8$ with $q_1\equiv 3\pmod 8$ or $[q_1\equiv 7\pmod 8$ and $\left(\frac{q_1}{r}\right)=-1]$,
		
		\item $r\equiv 3\pmod 8$ and $s\equiv 3\pmod 8$ with  $q_1\equiv 7\pmod 8$ and $\left(\frac{q_1}{r}\right)\not=\left(\frac{q_1}{s}\right)$, 
		
		\item $r\equiv 3\pmod 8$ and $s\equiv 7\pmod 8$ with  $q_1\equiv 3\pmod 8$.
	\end{enumerate}
	$\textbf{In this case, we have: } \rg(A(K_\infty))  =1.$

		  \item   $K=\QQ(\sqrt{q_1q_2},\sqrt{\delta rs})$ where     $r$ and $s$ are prime numbers such that $ \left(\frac{q_1q_2}{r}\right)=-\left(\frac{q_1q_2}{s}\right)=-1$, and   $\delta\in  \{1,q_1,q_2\}$ with $\delta rs\equiv 1\pmod 4$ and such that we have one of the following conditions:
		 \begin{enumerate}[$\C1:$]
		 	\item  $ r\equiv         s\equiv 3\pmod 8$ with $q_1\equiv 7\pmod 8$ and $\left(\frac{q_1}{r}\right) =-\left(\frac{q_1}{s}\right)=1$,
		 	\item $ r\equiv     3\pmod 8$ and  $    s\equiv 5\pmod 8$  with $q_1\equiv 7\pmod 8$ and $\left(\frac{q_1}{s}\right) =-1$,  
		 	
		 	\item   $ r\equiv     5\pmod 8$ and  $    s\equiv 3\pmod 8$ with   $q_1\equiv 3\pmod 8$ or    $[q_1\equiv 7\pmod 8$ and $\left(\frac{q_1}{r}\right) =-1]$.
		 \end{enumerate}
		 $\textbf{In this case, we have: } \rg(A(K_\infty))  =1.$

		 \item  $K=\QQ(\sqrt{q_1q_2},\sqrt{\delta rs})$ where     $r$ and $s$ are prime numbers such that $ \left(\frac{q_1q_2}{r}\right)=\left(\frac{q_1q_2}{s}\right)=1$, and   $\delta\in  \{1,q_1,q_2\}$ with $\delta rs\equiv 1\pmod 4$ and such that we have one of the following conditions:
		 \begin{enumerate}[$\C1:$]
		 	\item  $r\equiv 3\pmod 8$ and $s\equiv 3\pmod 8$  with       $ q_1\equiv 7\pmod 8$ and $\left(\frac{q_1}{r}\right) \not=\left(\frac{q_1}{s}\right) $,
		 	
		 	\item   $r\equiv 5\pmod 8$ and $s\equiv 3\pmod 8$  with       $ q_1\equiv 7\pmod 8$ and $\left(\frac{q_1}{r}\right) =-1$.
		 \end{enumerate}
		 $\textbf{In this case, we have: } \rg(A(K_\infty))  =1.$

		  \item   $K=\QQ(\sqrt{q_1q_2},\sqrt{\delta rs})$ where     $r$ and $s$ are prime numbers such that $ \left(\frac{q_1q_2}{r}\right)=-\left(\frac{q_1q_2}{s}\right)=-1$, and   $\delta\in  \{1,q_1,q_2\}$ with $\delta rs\equiv 1\pmod 4$ and such that we have one of the following conditions:
		 \begin{enumerate}[$\C1:$]

		 \item  $r\equiv   s\equiv 5\pmod 8$    with   
		 $[q_1\equiv 3\pmod 8$ and $\left(\frac{q_1}{s}\right)=1]$ or   $[q_1\equiv 7\pmod 8$ and $-\left(\frac{q_1}{r}\right) =\left(\frac{q_1}{s}\right)=1]$.

		 \item $r\equiv     3\pmod 8$ and  $   s\equiv 5\pmod 8$ with$\left(\frac{q_1}{s}\right)=1 $.
		 \end{enumerate}
		 $\textbf{In this case, we have: } \rg(A(K_\infty))  =2.$

		 \item   $K=\QQ(\sqrt{q_1q_2},\sqrt{\delta rs})$ where     $r$ and $s$ are prime numbers such that $ \left(\frac{q_1q_2}{r}\right)=\left(\frac{q_1q_2}{s}\right)=1$, and   $\delta\in  \{1,q_1,q_2\}$ with $\delta rs\equiv 1\pmod 4$ and such that we have one of the following conditions:
		 \begin{enumerate}[$\C1:$]
		 	\item   $r\equiv 7\pmod 8$ and $s\equiv 5 \pmod 8$ with   $q_1\equiv 3\pmod 8$ and     
		 	$\left(\frac{q_1}{s}\right)=1$,

		 	\item $r\equiv 7\pmod 8$ and $s\equiv 3 \pmod 8$ with   $q_1\equiv 3\pmod 8$ and $\left(\frac{q_1}{r}\right) =\left(\frac{q_1}{s}\right)$,
		  	\item $r\equiv 3\pmod 8$ and $s\equiv 3\pmod 8$ with   $q_1\equiv 3\pmod 4$ and $\left(\frac{q_1}{r}\right) =\left(\frac{q_1}{s}\right)$,
		  	\item $r\equiv 5\pmod 8$ and $s\equiv 5\pmod 8$ with       $ q_1\equiv 7\pmod 8$ and $[\left(\frac{q_1}{r}\right)=-1$ or $  \left(\frac{q_1}{s}\right)=-1 ]$,
		 	 
		 	\item $r\equiv 5\pmod 8$ and $s\equiv 3\pmod 8$ with    $q_1\equiv 3\pmod 4$  and $\left(\frac{q_1}{r}\right)=1$.
		 	
		 	
		 \end{enumerate}
		 $\textbf{In this case, we have: } \rg(A(K_\infty))  =2.$

		\end{enumerate}	
	\end{theorem}

	We note that a suitable permutation of  the primes $r$, $s$ and $t$ is taken into account in the above theorem.
	
\begin{remark}
	Note that the above theorem gives a partial answer for the problem   proposed  in the abstract that we  stated as follows:
	
		\begin{center}
		What are real biquadratic number fields $k$ such that,   $\rg(A(k_\infty))=\rg(A(k_1))$ ?
	\end{center}
	
The determination of these fields may have several interesting applications on Greenberg's conjecture and class field theory, as we can control the size of $A(k_\infty)$ through the class number of $k_1$ which is a triquadratic number field (see the last section of examples of such applications). 	Notice that the  investigation of this problem for the fields $K$ of the forms appearing in $A)$ to $F)$  also  gives    some families of real biquadratic fields $K'$ such that $K'_1/K'$ is unramified or  $K'_1/K'$ is not an extension of QO-fields. 
Following is the list of such $K'$:  
\begin{enumerate}[$\bullet$]
   	\item   $K'=\QQ(\sqrt{2q},\sqrt{d})$ or $\QQ(\sqrt{q},\sqrt{2d})$, where $d > 1$ is an odd  positive  square-free integer that is not divisible by $q$.
	
 	\item  $K'=\QQ(\sqrt{2q},\sqrt{2d})$, where    $q\equiv3 \pmod 4$   and $d\equiv 1 \pmod 4$ is a  positive square-free integer that is not divisible by   $q$.
	
 	\item  $K'=\QQ(\sqrt{2q_1q_2},\sqrt{d})$ or $\QQ(\sqrt{q_1q_2},\sqrt{2d})$, where  $q_1\equiv3 \pmod 4$,  $q_2\equiv3 \pmod 8$ are two prime numbers and $d\equiv 1\pmod 4$ is a positive    square-free integer that is not divisible by   $q_1q_2$.
\end{enumerate}	
In fact, for the fields $K'$ such that $K'_1=K_1$, where  $K$ is a real  biquadratic field that is taking one of the forms appearing in the main theorem, we have $\rg(K'_\infty)=\rg(K'_1)=\rg(K)$.
\end{remark}

	\bigskip
	The plan of this paper is the following. In Section \ref{sec2}, we collect some useful preliminary results. In Section \ref{sec3}, we present the proof of the main theorem. Section 
	\ref{sec4} is dedicated to the investigation of  the structure of the $2$-Iwasawa module  of  certain real biquadratic fields.

	\section{\bf Preliminary results}\label{sec2}
	
In this section, we collect some preliminary results concerning $2$-class group,  $2$-class number, unit group and Hilbert $2$-class field tower of a number field.

\subsection{Class number formula and the $2$-rank of the class group}

	\begin{lemma}[\cite{Ku-50}]\label{wada's f.}
		Let $k$ be a multiquadratic number field of degree $2^n$, $n\geq 2$,  and $k_i$ be the $s=2^n-1$ quadratic subfields of $k$. Then
		$$h(k)=\frac{1}{2^v}q(k)\prod_{i=1}^{s}h(k_i),$$
		where  $ q(k)=[E_k: \prod_{i=1}^{s}E_{k_i}]$ and   $$     v=\left\{ \begin{array}{cl}
			n(2^{n-1}-1); &\text{ if } k \text{ is real, }\\
			(n-1)(2^{n-2}-1)+2^{n-1}-1 & \text{ if } k \text{ is imaginary.}
		\end{array}\right.$$
	\end{lemma}
	\bigskip
	To compute that unit index  $q(k)$ appearing in the above lemma, it is useful  to
	  recall the following  method given in    \cite{wada}, that describes a fundamental system  of units of a real  multiquadratic field $k_0$. Let  $\sigma_1$ and 
	$\sigma_2$ be two distinct elements of order $2$ of the Galois group of $k_0/\mathbb{Q}$. Let $k_1$, $k_2$ and $k_3$ be the three subextensions of $k_0$ invariant by  $\sigma_1$,
	$\sigma_2$ and $\sigma_3= \sigma_1\sigma_2$, respectively. Let $\varepsilon$ denote a unit of $k_0$. Then \label{algo wada}
	$$\varepsilon^2=\varepsilon\varepsilon^{\sigma_1}  \varepsilon\varepsilon^{\sigma_2}(\varepsilon^{\sigma_1}\varepsilon^{\sigma_2})^{-1},$$
	and we have, $\varepsilon\varepsilon^{\sigma_1}\in E_{k_1}$, $\varepsilon\varepsilon^{\sigma_2}\in E_{k_2}$  and $\varepsilon^{\sigma_1}\varepsilon^{\sigma_2}\in E_{k_3}$.
	It follows that the unit group of $k_0$  
	is generated by the elements of  $E_{k_1}$, $E_{k_2}$ and $E_{k_3}$, and the square roots of elements of   $E_{k_1}E_{k_2}E_{k_3}$ which are perfect squares in $k_0$.
	\bigskip

	 \begin{lemma}[\cite{Qinred}, Lemma 2.4]\label{AmbiguousClassNumberFormula} Let $k/k'$  be a QO-extension of number fields. Then the rank of the $2$-class group of $k$ is given by
		$$\rg({A(k)})=t_{k/k'}-1-e_{k/k'},$$
		where  $t_{k/k'}$ is the number of  ramified primes (finite or infinite) in the extension  $k/k'$ and $e_{k/k'}$ is  defined by   $2^{e_{k/k'}}=[E_{k'}:E_{k'} \cap N_{k/k'}(k^*)]$.
	\end{lemma}

	The following lemma is a particular case of Fukuda's Theorem  \cite{fukuda}.

	\begin{lemma}[\cite{fukuda}]\label{lm fukuda}
		Let $k_\infty/k$ be a $\mathbb{Z}_2$-extension and $n_0$  an integer such that any prime of $k_\infty$ which is ramified in $k_\infty/k$ is totally ramified in $k_\infty/k_{n_0}$.
		\begin{enumerate}[\rm $1)$]
			\item If there exists an integer $n\geq n_0$ such that   $h_2(k_n)=h_2(k_{n+1})$, then $h_2(k_n)=h_2(k_{m})$ for all $m\geq n$.
			\item If there exists an integer $n\geq n_0$ such that $\rg( A(k_n))= \rg(A(k_{n+1}))$, then
			$\rg(A(k_{m}))= \rg(A(k_{n}))$ for all $m\geq n$.
		\end{enumerate}
	\end{lemma}

	Let us state the following useful remark. For more about the context of this remark we refer the reader to \cite[p. 160]{lemmermeyer2013reciprocity}.
	Let $\varepsilon_d$ denote the fundamental unit of a real quadratic fields $\QQ(\sqrt{d})$.
	
	\begin{remark}\label{scholsrec}
		Let $p=2$ or $p\equiv 1\pmod 4$ be a prime number and let  $r$ be  any odd prime  number  such that $ \left(\frac{p}{ r}\right)=1$. Let $\mathcal R_{\mathbb{Q}(\sqrt{p})}$ and $\mathcal R_{\mathbb{Q}(\sqrt{p})}'$ be the two primes of $\mathbb{Q}(\sqrt{p})$ above $r$.
		We choose $\mathcal R_{\mathbb{Q}(\sqrt{p})}$ and $\mathcal R_{\mathbb{Q}(\sqrt{p})}'$ such that :
		\begin{enumerate}[$1)$]
			\item  $ \left(\frac{\varepsilon_p}{\mathcal R_{\mathbb{Q}(\sqrt{p})}}\right)=\left(\frac{\varepsilon_p}{\mathcal R_{\mathbb{Q}(\sqrt{p})}'}\right)$ if ($p\equiv 1\pmod 4$ and $r\equiv 1\pmod 4$) or ($p=2$ and $r\equiv 1\pmod 8$). In this case  $ \left(\frac{\varepsilon_p}{\mathcal R_{\mathbb{Q}(\sqrt{p})}}\right) = \left(\frac{p}{r}\right)_4\left(\frac{r}{p}\right)_4$ and this is called Scholz reciprocity law (cf. \cite[p. 160]{lemmermeyer2013reciprocity}).

			\item 	$ \left(\frac{\varepsilon_p}{\mathcal R_{\mathbb{Q}(\sqrt{p})}}\right)=-\left(\frac{\varepsilon_p}{\mathcal R_{\mathbb{Q}(\sqrt{p})}'}\right)=1$  if ($p\equiv 1\pmod 4$ and $r\equiv 3\pmod 4$) or ($p=2$ and $r\equiv 7\pmod 8$).  
		\end{enumerate}
	\end{remark}

\subsection{Hilbert $2$-class field tower of  fields $k$ such that  $A(k)\simeq\ZZ/2 \ZZ\times\ZZ/2^n \ZZ$ }	
In this subsection, we review some results from class field theory that will be very useful for what follows.  Let $k$ be  a number field.    
 Let  $k^{(1)}$  be the Hilbert $2$-class field of $k$, that is  the maximal unramified  abelian field extension
of $k$ whose degree over $k$ is a $2$-power. Put $k^{(0)} = k$ and let $k^{(i)}$ denote the Hilbert $2$-class field of $k^{(i-1)}$
for any integer $i\geq 1$. Then the sequence of fields
$$k=k^{(0)} \subset k^{(1)} \subset  k^{(2)}  \subset \cdots\subset k^{(i)} \cdots \subset  \bigcup_{i\geq 0} k^{(i)}=\mathcal{L}(k)$$
is called   the $2$-class field tower of $k$. If for all $i\geq1$,  $k^{(i)}\neq k^{(i-1)}$, the tower is said to be infinite, otherwise the tower is said to be  finite, and the minimal integer $i$ satisfying the condition $k^{(i)}= k^{(i-1)}$ is called the length of the tower. The field $\mathcal{L}(k)$ is called the maximal unramified pro-$2$-extension  of $k$ and for $k_\infty $, the cyclotomic $\ZZ_2$-extension of $k$, the group $G_{k_\infty}=\mathrm{Gal}(\mathcal{L}(k_\infty)/k_\infty)$ is isomorphic to the inverse limit $\varprojlim \mathrm{Gal}(\mathcal{L}(k_n)/k_n)$ with respect to the restriction map.

\bigskip

Now let $k$ be a number field such that $A(k)\simeq\ZZ/2 \ZZ\times\ZZ/2^n \ZZ$, where $n\geq2$ is a natural number.
Put $G_k=\mathrm{Gal}(\mathcal{L}(k)/k)$. So by class field theory, we have 
  $G_k/G_k'\simeq\ZZ/2 \ZZ\times\ZZ/2^n \ZZ$ and  $G_k=\langle a,
	b\rangle$ such that $a^2\equiv b^{2^n}\equiv 1\mod G_k'$ and $A(k)=\langle \mathfrak{c}, \mathfrak{d}\rangle\simeq \langle aG_k',
	bG'\rangle$ where $(\mathfrak{c}, k^{(2)}/k)=aG_k'$
	and $(\mathfrak{d}, k^{(2)}/k)=bG_k'$ with $(\ .\ ,
	k^{(2)}/k)$ is the Artin symbol in the extension
	$k^{(2)}/k$. Therefore, there exist three normal subgroup of $G_k$ of index $2$ denoted  $H_{1, 2}$, $H_{2, 2}$ and
	$H_{3, 2}$ such that
	\begin{center}
		$H_{1, 2}=\langle b, G_k'\rangle$, $H_{2, 2}=\langle ab, G_k'\rangle$ and
		$H_{3, 2}=\langle a, b^2, G_k'\rangle.$
	\end{center}
 Furthermore, there exist three normal subgroups of $G_k$ of index
	$4$  denoted  $H_{1, 4}$,
	$H_{2, 4}$ and $H_{3, 4}$ such that
 	$$H_{1, 4}=\langle a,b^4, G_k'\rangle, \quad H_{2, 4}=\langle ab^2, G_k'\rangle \text{ and } H_{3, 4}=\langle b^2, G_k'\rangle.$$
  Each subgroup $H_{i,j}$ of $A(k)$
corresponds to 	an unramified extension    ${K}_{i,j}$ of $k$ contained in
	$k^{(1)}$ such that $A(k)/H_{i,j}\simeq
	\operatorname{Gal}(K_{i,j}/k)$ and
	$H_{i,j}=\mathcal{N}_{{K}_{i,j}/k}(A(K_{i,j}) )$. The Hilbert $2$-class field tower of $k$ can be schematized as follows:
	\begin{figure}[H]
		$$
		\xymatrix{
			& k^{(2)} \ar@{<-}[d] & \\
			& k^{(1)}\ar@{<-}[ld]\ar@{<-}[d]\ar@{<-}[rd]\\
			{K_{1,  4}}\ar@{<-}[rd]&\ar@{<-}[ld] {K_{3,  4}}\ar@{<-}[d]\ar@{<-}[rd]  & {K_{2,  4}}\ar@{<-}[ld]\\
			{K_{1,  2}}\ar@{<-}[rd]& {K_{3,  2}} \ar@{<-}[d]  & {K_{2,  2}}\ar@{<-}[ld]\\
			&k
		}
		$$
		\caption{\label{Fig2}}
	\end{figure}
	
For more details, we refer the reader to \cite{aaboune}.	
With the above hypothesis and notations we have:

		\begin{theorem}[\cite{aaboune}, Theorem 4.10]\label{AabounePrzekhini}
		The following assertions are equivalent:
		\begin{enumerate}[\rm $1)$]
			\item $G_k$ is abelian,
			\item The  $2$-class group of ${K_{3,  2}}$ is isomorphic to  $\ZZ/2 \ZZ\times\ZZ/2^{n-1} \ZZ$,
			\item The  $2$-class number of ${K_{3,  2}}$ equals $2^n$,
			\item The $2$-class group of  ${K_{i, 2}}$ is cyclic of order $2^n$ with $i=1$ or $2$,
			\item The  $2$-class number of ${K_{i,  2}}$ equals $2^n$ with $i=1$ or $2$,
			\item The Hilbert $2$-class field tower of $k$ stops at $k^{(1)}$.
		\end{enumerate}
	\end{theorem}
	
Furthermore, we have:
	
	 \begin{proposition}[\cite{BLS98}, Proposition 7]\label{LemBenjShn}
		Let $k$ be a number field such that $A(k)\simeq (2^m, 2^n)$ for some positive integers $m$ and $n$. If there is an unramified
		quadratic extension of $k$  with $2$-class number $2^{m+n-1}$,  then all three unramified quadratic extensions of $k$ have $2$-class number
		$2^{m+n-1}$,  and the $2$-class field tower of $k$ terminates at $k^{(1)}$.
		
		Conversely, if the $2$-class field tower of $k$ terminates at $k^{(1)}$, then all three
		unramified quadratic extensions of $k$ have $2$-class number $2^{m+n-1}$.
	\end{proposition}

	\section{\bf The proof of the main theorem}\label{sec3}

 Let $\mathbb P$ denote the set of prime numbers and put $\mathbb P_{m}:=\{p\in \mathbb P: p\equiv 1\pmod 8\text{ and } \left(\frac{m}{p}\right)=1\}$. Let us start with the following lemmas concerning the rank of the $2$-class group of real biquadratic fields.
 \begin{lemma}\label{lemmaqA} Let $K=\QQ(\sqrt{q},\sqrt{d})$ where $q\equiv 3\pmod 4$  is a prime number and $d > 1$ is an odd    square-free integer that is not divisible by $q$. Then $\rg(A(K))\leq 2$ if and only if $K$ takes one of the following forms:
 	\begin{enumerate}[$1)$]
 		\item $K=\QQ(\sqrt{q},\sqrt{d})$, where   $d=r$ is a prime number. {\bf In this case,} $\rg(A(K))\in \{0,1\}$, more precisely, $\rg(A(K))=1$ if and only if 
 		$ \left(\frac{q}{r}\right)=1$	and $r\equiv 1\pmod 8$.
 		
 		\item   $K=\QQ(\sqrt{q},\sqrt{d})$, where    $d=rs $ with $r$ and $s$ are two primes number such that $(r,s)\not\in   \mathbb P_{q}^2$. {\bf In this case,} $\rg(A(K))\in \{1,2\}$. More precisely, $\rg(A(K))=1$ if and only if, after a suitable permutation, $r$ and $s$ satisfy one of the following conditions:
 		\begin{enumerate}[\indent$	 \C1: $]
 			\item   $ \left(\frac{q}{r}\right)= \left(\frac{q}{s}\right)=-1$,
 			\item  $ \left(\frac{q}{r}\right)= -\left(\frac{q}{s}\right)=-1$ and $s\not\equiv 1\pmod 8$,
 			\item  $ \left(\frac{q}{r}\right)=  \left(\frac{q}{s}\right)= 1$ and  $[( r\equiv  7\pmod 8$ and $s\equiv 3$ or $5\pmod 8 )$ or $( r\equiv  3\pmod 8$ and $s\equiv 5\pmod 8 )]$.
 		\end{enumerate}
 		\item   $K=\QQ(\sqrt{q},\sqrt{d})$, where    $d=rst $ with $r$, $s$ and $t$ are   primes numbers satisfying one of the following conditions:
 		\begin{enumerate}[\indent$	 \C1: $]

 			\item   $\left(\frac{q}{r}\right)=\left(\frac{q}{s}\right)=-\left(\frac{q}{t}\right)=1$,   $\left(\frac{-1}{r}\right)=\left(\frac{2}{r}\right)=-1$ and 
 			$\left(\frac{-1}{s}\right)\not=\left(\frac{2}{s}\right) $, 
 			\item    $\left(\frac{q}{r}\right)=-\left(\frac{q}{s}\right)=-\left(\frac{q}{t}\right)=1$ and  $[  \left(\frac{-1}{r}\right)=\left(\frac{2}{r}\right)=-1 $ or   $ \left(\frac{-1}{r}\right)\not=\left(\frac{2}{r}\right) ]$,
 			
 		 \item $\left(\frac{q}{r}\right)=\left(\frac{q}{s}\right)=\left(\frac{q}{t}\right)=-1$.   
 		\end{enumerate}
 	{\bf In this case,} $\rg(A(K))=2$.
 	 	\end{enumerate}	
 \end{lemma}
 \begin{proof}
 	Assume that 	$K=\QQ(\sqrt{q},\sqrt{d})$ where $d > 1$ is an odd  positive  square-free integer that is not divisible by $q$.
 	Let $t$ be the number of prime ideals of $ F=\QQ(\sqrt{q})$ that are ramified in $K$. Thus,  $ \rg(A(K))=t_{K/F}-1-e_{K/F}$. 
 	As  $e_{K/F}\in\{0,1,2\}$, then the inequality $\rg(A(K))\leq 2$ implies that $t_{K/F}\leq 5$. The last inequality is equivalent to the fact that  $d$ takes one of the following five forms :
 	\begin{eqnarray*}
 		&r,&\quad rs, \quad rst  \text{ with }  \left(\frac{q}{t}\right)=-1,\quad
 		rstk \text{ with }  \left(\frac{q}{s}\right)=\left(\frac{q}{t}\right)=\left(\frac{q}{k}\right)=-1,\\
 		& rstke&  \text{ with }    \left(\frac{q}{r}\right)=\left(\frac{q}{s}\right)=\left(\frac{q}{t}\right)=\left(\frac{q}{k}\right)=\left(\frac{q}{e}\right)=-1,
 	\end{eqnarray*}
 	where $r$, $s$, $t$,  $k$ and $e$ are prime numbers all different of $q$. 
 	\begin{enumerate}[$\blacktriangleright$]
 		\item   Let  $d=r$ be a prime number. If $ \left(\frac{q}{r}\right)=-1$, then   $ \rg(A(K))=1-1-e_{K/F}$ and so necessary  $   \rg(A(K))=0$.
 	 If $ \left(\frac{q}{r}\right)=1$, then $ \rg(A(K))=1- e_{K/F}$. Thus, by  \cite[Theorem 3.3]{Azmouh2-rank},    $ \rg(A(K))=1$ if and only if   $r\equiv 1\pmod 8$, otherwise  $ \rg(A(K))=0$.
 		
 		\item  Let  $d=rs$.  
 		We have  : 
 		\begin{enumerate}[$\star$]
 			\item  If   $ \left(\frac{q}{r}\right)= \left(\frac{q}{s}\right)=-1$, then   $t_{K/F}=2$  and so we have  $ \rg(A(K))=1$ (cf.  \cite[Theorem 3.3]{Azmouh2-rank}).
 			
 			\item If    $ \left(\frac{q}{r}\right)= -\left(\frac{q}{s}\right)=-1$,  then   $t_{K/F}=3$ and by   \cite[Theorem 3.3]{Azmouh2-rank}, we have $e_{K/F}\in \{0,1\}$. More precisely,  $e_{K/F}=0$ if and only if  $\left(\frac{-1}{ s}\right)=\left(\frac{2}{ s}\right)=1$.
 			Therefore, $ \rg(A(K))\in \{1,2\}$. More precisely, $ \rg(A(K))=2$ if and only if $\left(\frac{-1}{ s}\right)=\left(\frac{2}{ s}\right)=1$, which is equivalent to $s\equiv 1\pmod 8$.


 			\item  If    $ \left(\frac{q}{r}\right)=  \left(\frac{q}{s}\right)= 1$, then   $t_{K/F}=4$  and so  by  \cite[Theorem 3.3]{Azmouh2-rank}, we have
 			  $ \rg(A(K))=1$, $2$ or $3$. More precisely, $ \rg(A(K))=3  $ if and only if $\left(\frac{-1}{ r}\right)=\left(\frac{2}{ r}\right)=   \left(\frac{-1}{ s}\right)=\left(\frac{2}{ s}\right)=1$, which is equivalent to $r\equiv s\equiv 1\pmod 8$.   Moreover, $ \rg(A(K))=2  $ if and only if $(r\equiv 1\pmod 8$ and $s\equiv 3$ or $5\pmod 8)$, 
 		   $(r\equiv s\equiv 3\pmod 8 )$, $(r\equiv s\equiv 5\pmod 8 )$ or $(r\equiv s\equiv 7\pmod 8 )$.
 	\end{enumerate}

 			\item  Let  $d=rst$    with $  \left(\frac{q}{t}\right)=-1$. 	We have  : 
 			\begin{enumerate}[$\star$]
 			
 				\item
 				If $\left(\frac{q}{r}\right)=\left(\frac{q}{s}\right)=1$, then $t_{K/F}=5$ and by \cite[Theorem 3.3]{Azmouh2-rank}, we have $  \rg(A(K))\in\{4,3,2\}$. More precisely, $ \rg(A(K))=4$ if and only if $r\equiv s\equiv 1\pmod 4$ and 
 		$\left(\frac{2}{r}\right)=\left(\frac{2}{s}\right)=1$, moreover, 
 		$ \rg(A(K))=2$ if and only if $\left(\frac{-1}{r}\right)=\left(\frac{2}{r}\right)=-1$ and 
 		$\left(\frac{-1}{s}\right)\not=\left(\frac{2}{s}\right) $.
 		 \item $\left(\frac{q}{r}\right)=-\left(\frac{q}{s}\right)=1$, then $t_{K/F}=4$ and by \cite[Theorem 3.3]{Azmouh2-rank}, we have $  \rg(A(K))\in\{3,2\}$. More precisely,
 	 $ \rg(A(K))=3$ if and only if  
 	 $\left(\frac{-1}{r}\right)=\left(\frac{2}{r}\right)=1$.
 	 
 	 	\item
 	 If $\left(\frac{q}{r}\right)=\left(\frac{q}{s}\right)=-1$, then $t_{K/F}=3$ and so   $\rg(A(K))=2$ (cf. \cite[Theorem 3.3]{Azmouh2-rank}).
  	\end{enumerate}
  Therefore, $ \rg(A(K))=2$ if and only if 
    $[$$\left(\frac{q}{r}\right)=\left(\frac{q}{s}\right)=1$,   $\left(\frac{-1}{r}\right)=\left(\frac{2}{r}\right)=-1$ and 
  $\left(\frac{-1}{s}\right)\not=\left(\frac{2}{s}\right) $$]$ 
   or $[$ $\left(\frac{q}{r}\right)=-\left(\frac{q}{s}\right)=1$ and  $\left(\left(\frac{-1}{r}\right)=\left(\frac{2}{r}\right)=-1\right.$ or   $\left.\left(\frac{-1}{r}\right)\not=\left(\frac{2}{r}\right)\right) ]$.

 	\item  Let $d=	rstk$  with $ \left(\frac{q}{s}\right)=\left(\frac{q}{t}\right)=\left(\frac{q}{k}\right)=-1$.
 		\begin{enumerate}[$\star$]
 		\item If $\left(\frac{q}{r}\right)=-1$, then $t_{K/F}=4$ and so   $\rg(A(K))=3$ (cf. \cite[Theorem 3.3]{Azmouh2-rank}).
 		
 		\item If $\left(\frac{q}{r}\right)=1$, then $t_{K/F}=5$ and so    $ \rg(A(K))=4$ or $3$. More precisely, $ \rg(A(K))=4$ if and only if $r\equiv1\pmod 8$ (cf. \cite[Theorem 3.3]{Azmouh2-rank}).
  	\end{enumerate}

 		\item  Let $d=	rstke$  with    $ \left(\frac{q}{r}\right)=\left(\frac{q}{s}\right)=\left(\frac{q}{t}\right)=\left(\frac{q}{k}\right)=\left(\frac{q}{e}\right)=-1$.
 	Then  $t_{K/F}=5$ and so    $ \rg(A(K))=4$  (cf. \cite[Theorem 3.3]{Azmouh2-rank}).
 	
  	\end{enumerate}	
  This gives the lemma.	
  \end{proof}

The following lemma concerns the fields $K$ that are of the form $C)$.

 \begin{lemma}\label{lemmaq1q2Bd=1mod4} Let $K=\QQ(\sqrt{q_1q_2},\sqrt{d})$ where $q_1\equiv3 \pmod 4$,  $q_2\equiv3 \pmod 8$ are two prime numbers  and  ${\color{blue} d \equiv 1\pmod 4}$ is a  positive  square-free integer that is {\bf not} divisible by $q_1q_2$. Let $\delta\in \{1,q_1,q_2\}$. Then $\rg(A(K))\leq 2$ if and only if $K$ takes one of the following forms:
 	\begin{enumerate}[$1)$]
 		\item  $K=\QQ(\sqrt{q_1q_2},\sqrt{d})$, where   $d=\delta r\equiv 1\pmod 4 $ with $r$ is a prime  number. {\bf In this case,} $\rg(A(K))\in \{0,1\}$, more precisely, $\rg(A(K))=1$ if and only if   $\left(\frac{q_1q_2 }{ r}\right)=\left(\frac{-1 }{ r}\right)=\left(\frac{q_1 }{ r}\right)=1$.

 		\item   $K=\QQ(\sqrt{q_1q_2},\sqrt{d})$, where   $d=\delta rs\equiv 1\pmod 4 $ with $r$ and $s$ are two prime  numbers that {\bf do not} satisfy $ \left(\frac{q_1q_2}{r}\right)=  \left(\frac{q_1q_2}{s}\right)=  \left(\frac{q_1 }{ r}\right)=\left(\frac{-1}{ r}\right)=\left(\frac{q_1 }{ s}\right)=\left(\frac{-1}{s}\right)=1$. {\bf In this case,} $\rg(A(K))\in \{ 1,2\}$, more precisely, $\rg(A(K))=1$ if and only if, after a suitable permutation of $r$ and $s$, we have one of the following conditions:
 		\begin{enumerate}[$\bullet$]
 			\item  $ \left(\frac{q_1q_2}{r}\right)= \left(\frac{q_1q_2}{s}\right)=-1$,
 			\item $ \left(\frac{q_1q_2}{r}\right)= -\left(\frac{q_1q_2}{s}\right)=-1$ and 
 			$[\left(\frac{q_1}{s}\right)\not=\left(\frac{-1}{s}\right)$ or $\left(\frac{q_1}{ s}\right)=\left(\frac{-1}{s}\right)=-1]$,

 			\item $ \left(\frac{q_1q_2}{r}\right)=  \left(\frac{q_1q_2}{s}\right)= 1$, $[ \left(\frac{q_1 }{ r}\right)=-1$ or $\left(\frac{q_1 }{ s}\right)=-1  ]$, $[\left(\frac{-1}{ r}\right)=-1$ or $ \left(\frac{-1}{s}\right)=-1]$
 			and $[ \left(\frac{q_1 }{ r}\right) \left(\frac{-1}{ r}\right)=-1  \text{ or  }   \left(\frac{q_1 }{ s}\right) \left(\frac{-1}{s}\right)=-1   ]$.
 			
 		\end{enumerate}

 		\item    $K=\QQ(\sqrt{q_1q_2},\sqrt{d})$, where   $d=\delta rst\equiv 1\pmod 4 $ with $r$, $s$ and $t$ are   prime  numbers satisfying one of the following conditions:
 		\begin{enumerate}[$\bullet$]
 			\item $\left(\frac{q_1q_2}{r}\right)=\left(\frac{q_1q_2}{s}\right)=\left(\frac{q_1q_2}{t}\right)=-1$, 	
 			\item  $\left(\frac{q_1q_2}{r}\right)=-\left(\frac{q_1q_2}{s}\right)=\left(\frac{q_1q_2}{t}\right)=-1$ and $[\left(\frac{q_1}{s}\right)\not=\left(\frac{-1}{s}\right)$ or $\left(\frac{q_1}{ s}\right)=\left(\frac{-1}{s}\right)=-1  ]$.
 		\end{enumerate}
 		{\bf In this case,} $\rg(A(K))=2$.  
 	\end{enumerate}	
 \end{lemma}
 \begin{proof}	 	 Assume that $K=\QQ(\sqrt{q_1q_2},\sqrt{d})$ where   $q_1\equiv3 \pmod 4$,  $q_2\equiv3 \pmod 8$ are two prime numbers  and {\color{blue}$d\equiv 1\pmod 4$} is any       positive   square-free integer that is not divisible by   $q_1q_2$.
 	Let $t_{K/F}$ be the number of prime ideals of $ F $ that are ramified in $K$. Thus,  $ \rg(A(K))=t_{K/F}-1-e_{K/F}$, where $F=\QQ(\sqrt{q_1q_2})$. 
 	As  $e_{K/F}\in\{0,1,2\}$,   the inequality $\rg(A(K))\leq 2$ implies that $t_{K/F}\leq 5$. The last inequality is equivalent to the fact that  $d$ takes one of the following five forms :
 	\begin{eqnarray*}
 		&\delta r, &\ \delta rs, \quad \delta  rst  \text{ with }  \left(\frac{q_1q_2}{t}\right)=-1,\\
 		&\delta rstk& \text{ with }  \left(\frac{q_1q_2}{s}\right)=\left(\frac{q_1q_2}{t}\right)=\left(\frac{q_1q_2}{k}\right)=-1,\\
 		&\delta  rstke &  \text{ with }    \left(\frac{q_1q_2}{r}\right)=\left(\frac{q_1q_2}{s}\right)=\left(\frac{q_1q_2}{t}\right)=\left(\frac{q_1q_2}{k}\right)=\left(\frac{q_1q_2}{e}\right)=-1,
 	\end{eqnarray*}
 	where $\delta\in \{1,q_1,q_2\}$ and $r$, $s$, $t$,  $k$ and $e$ are prime numbers all different of $q_1$ and $q_2$.

 	\begin{enumerate}[$\blacktriangleright $]
 		\item 	 Let  $d=\delta r$. If $ \left(\frac{q_1q_2}{r}\right)=-1$, then   $ \rg(A(K))=1-1-e_{K/F}$ and so necessary  $   \rg(A(K))=0$.  If $ \left(\frac{q_1q_2}{r}\right)=1$, then  $ \rg(A(K))=2-1-e_{K/F}$ and $e_{K/F}=0$ if and only if   $\left(\frac{-1 }{ r}\right)=\left(\frac{q_1 }{ r}\right)=1$ (cf. \cite[Théorèmes 3.3 et 3.4]{Azmouh2-rank}).  
 		Hence  $ \rg(A(K))=1$ if and only if  $\left(\frac{-1 }{ r}\right)=\left(\frac{q_1 }{ r}\right)=1$, otherwise  $ \rg(A(K))=0$.

 		\item Let  $d=\delta rs$. We have  $ \rg(A(K))=t_{K/F}-1-e_{K/F}$ and by \cite[Théorèmes 3.3 et 3.4]{Azmouh2-rank}, we have:
 		\begin{enumerate}[$\star$]
 			\item  If   $ \left(\frac{q_1q_2}{r}\right)= \left(\frac{q_1q_2}{s}\right)=-1$,  then     $ \rg(A(K))= 2-1-0=1$.

 			\item If    $ \left(\frac{q_1q_2}{r}\right)= -\left(\frac{q_1q_2}{s}\right)=-1$, then  $ \rg(A(K)) =2-e_{K/F}\in\{1,2\}$,   and  $  \rg(A(K))=2$ if and only if $\left(\frac{q_1 }{s}\right)=\left(\frac{-1}{s}\right)=1$.

 			\item  If    $ \left(\frac{q_1q_2}{r}\right)=  \left(\frac{q_1q_2}{s}\right)= 1$, then we have: 
 			$ \rg(A(K)) =3-e_{K/F}\in\{1,2,3\}$,   more precisely,   $  \rg(A(K))=3$ if and only if $\left(\frac{q_1 }{ r}\right)=\left(\frac{-1}{ r}\right)=\left(\frac{q_1 }{ s}\right)=\left(\frac{-1}{s}\right)=1$
 			and  $  \rg(A(K))=1$ if and only if $[ \left(\frac{q_1 }{ r}\right)=-1$ or $\left(\frac{q_1 }{ s}\right)=-1  ]$, $[\left(\frac{-1}{ r}\right)=-1$ or $ \left(\frac{-1}{s}\right)=-1]$
 			and $[ \left(\frac{q_1 }{ r}\right) \left(\frac{-1}{ r}\right)=-1  \text{ or  }   \left(\frac{q_1 }{ s}\right) \left(\frac{-1}{s}\right)=-1   ]$.
 		\end{enumerate}	
 		
 		\item Let  $d=\delta  rst  \text{ with }  \left(\frac{q_1q_2}{t}\right)=-1$. Notice that if 
 		$\left(\frac{q_1q_2}{r}\right)=\left(\frac{q_1q_2}{s}\right)=1$,  $t_{K/F}=5$ and so $  \rg(A(K))\geq 3$. Furthermore, if $\left(\frac{q_1q_2}{r}\right)=\left(\frac{q_1q_2}{s}\right)=-1$, then    $\rg(A(K))=3-1-0=2$. Now assume that $\left(\frac{q_1q_2}{r}\right)=-\left(\frac{q_1q_2}{s}\right)=-1$. 
 		Thus, we have $\rg(A(K))=4-1-e_{K/F}\in\{2,3\} $ and $\rg(A(K))=2$ if and only if $[\left(\frac{q_1}{s}\right)\not=\left(\frac{-1}{s}\right)$ or $\left(\frac{q_1}{ s}\right)=\left(\frac{-1}{s}\right)=-1]$.

 		\item Let  $d= \delta rstk$    with    $\left(\frac{q_1q_2}{s}\right)=\left(\frac{q_1q_2}{t}\right)=\left(\frac{q_1q_2}{k}\right)=-1$.
 		We have	$ \rg(A(K))=t_{K/F}-1-e_{K/F}\in \{3,4\}$, and $ \rg(A(K))=3$ if and only if $\left(\frac{q_1q_2}{r}\right) =-1$ or $[\left(\frac{q_1q_2}{r}\right) =1$ and  $\left(\frac{q_1}{r}\right)=\left(\frac{-1}{r}\right)=1]$.
 		\item Let  $d= \delta rstke$    with $\left(\frac{q_1q_2}{r}\right)=\left(\frac{q_1q_2}{s}\right)=\left(\frac{q_1q_2}{t}\right)=\left(\frac{q_1q_2}{k}\right)=\left(\frac{q_1q_2}{e}\right)=-1$. It is clear that 	$	  \rg(A(K))=t_{K/F}-1-e_{K/F}=4$.
 		
 	\end{enumerate}	This completes the proof of the lemma.
  \end{proof}

 \begin{lemma}\label{lemma2qB} Let $K=\QQ(\sqrt{2q},\sqrt{d})$ where $q\equiv 3\pmod 4$  is a prime number and $d\equiv1\pmod4$ is a  positive   square-free integer. Then $\rg(A(K))\leq 2$ if and only if $K$ takes one of the following forms:
 	\begin{enumerate}[$1)$]
 		\item $K=\QQ(\sqrt{2q},\sqrt{d})$, where   $d=r\equiv1\pmod4$ or  $d=qr\equiv1\pmod4$ with $r$ is a    prime number. {\bf In this case,} $\rg(A(K))\in \{0,1\}$, more precisely, $\rg(A(K))=1$ if and only if 
 		$ \left(\frac{q}{r}\right)=1$	and $r\equiv 1\pmod 8$.
 		
 		\item   $K=\QQ(\sqrt{2q},\sqrt{d})$, where    $d=\delta rs\equiv1\pmod4 $ with $\delta\in \{1,q\}$, $r$ and $s$ are two primes number such that $(r,s)\not\in   \mathbb P_{q}^2$. {\bf In this case,} $\rg(A(K))\in \{1,2\}$, more precisely, $\rg(A(K))=1$ if and only if, after a suitable permutation, $r$ and $s$ satisfy one of the following condition:
 		\begin{enumerate}[\indent$	 \C1: $]
 			\item   $ \left(\frac{2q}{r}\right)= \left(\frac{2q}{s}\right)=-1$,
 			\item  $ \left(\frac{2q}{r}\right)= -\left(\frac{2q}{s}\right)=-1$ and $s\not\equiv 1\pmod 8$,
 			\item  $ \left(\frac{2q}{r}\right)=  \left(\frac{2q}{s}\right)= 1$,  with  $ r\equiv  7\pmod 8$ and $s\equiv 3 \pmod 8 $.
 		\end{enumerate}
 		
 		\item   $K=\QQ(\sqrt{2q},\sqrt{d})$, where    $d=\delta rst\equiv1\pmod4 $
 		with $\delta\in \{1,q\}$ and  $r$, $s$ and $t$ are   primes numbers satisfying one of the following conditions:
 		\begin{enumerate}[\indent$	 \C1: $]
 		
 			\item   $\left(\frac{2q}{r}\right)=\left(\frac{2q}{s}\right)=-\left(\frac{2q}{t}\right)=1$,   $\left(\frac{-1}{r}\right)=\left(\frac{2}{r}\right)=-1$ and 
 			$\left(\frac{-1}{s}\right)\not=\left(\frac{2}{s}\right) $, 
 			\item    $\left(\frac{2q}{r}\right)=-\left(\frac{2q}{s}\right)=-\left(\frac{2q}{t}\right)=1$ and  $[  \left(\frac{-1}{r}\right)=\left(\frac{2}{r}\right)=-1 $ or   $ \left(\frac{-1}{r}\right)\not=\left(\frac{2}{r}\right) ]$,
 			
 		 \item $\left(\frac{2q}{r}\right)=\left(\frac{2q}{s}\right)=\left(\frac{2q}{t}\right)=-1$.   
 		\end{enumerate}
 		{\bf In this case,} $\rg(A(K))=2$.
 	\end{enumerate}	
 \end{lemma}
 \begin{proof}
 Let $t$ be the number of prime ideals of $ F=\QQ(\sqrt{2q})$ that are ramified in $K$. Thus,  $ \rg(A(K))=t_{K/F}-1-e_{K/F}$. 
 As  $e_{K/F}\in\{0,1,2\}$,   the inequality $\rg(A(K))\leq 2$ implies that $t_{K/F}\leq 5$. The last inequality is equivalent to the fact that  $d$ takes one of the following five forms :
 \begin{eqnarray*}
 	&\delta r,&\quad \delta rs, \quad \delta rst  \text{ with }  \left(\frac{2q}{t}\right)=-1,\\
 &	\delta rstk& \text{ with }  \left(\frac{2q}{s}\right)=\left(\frac{2q}{t}\right)=\left(\frac{2q}{k}\right)=-1,\\
 	& \delta rstke&  \text{ with }    \left(\frac{2q}{r}\right)=\left(\frac{2q}{s}\right)=\left(\frac{2q}{t}\right)=\left(\frac{2q}{k}\right)=\left(\frac{2q}{e}\right)=-1,
 \end{eqnarray*}
 where $\delta\in \{1,q\}$ and $r$, $s$, $t$,  $k$ and $e$ are odd prime numbers all different of $q$. 
 \begin{enumerate}[$\blacktriangleright$]
 	\item   Let  $d=\delta r$. Assume that $ \left(\frac{2q}{r}\right)=-1$, then  by $ \rg(A(K))=1-1-e_{K/F}=0$ (cf. \cite[Théorème 3.3]{Azmouh2-rank}).
 	   Now assume that $ \left(\frac{2q}{r}\right)=1$, then  $ \rg(A(K))=2-1-e_{K/F}$ and $e_{K/F}=0$ if and only if   $\left(\frac{-1 }{ r}\right)=\left(\frac{2 }{ r}\right)=1$ (cf. \cite[Théorème 3.3]{Azmouh2-rank}).  
 Thus,  $ \rg(A(K))=1$ if and only if  $\left(\frac{-1 }{ r}\right)=\left(\frac{2}{ r}\right)=1$, otherwise  $ \rg(A(K))=0$.
  
  	\item 	We proceed similarly, as in the proof of Lemma \ref{lemmaqA}, for the remain cases to complete the proof.
  	\end{enumerate}
 \end{proof}

 \bigskip
 
 Now we investigate the $2$-rank of the class group of $K_1$. 
 Let us state the following remark that will be useful latter.  For more details concerning the properties of norms residue symbols, we refer the reader to 
    \cite[Chapter II, Theorem 3.1.3]{grasbook} or 
 \cite[Chapter X]{Fuller}  and he can find them summarized in \cite{ACZrendi}.

 \begin{remark}\label{remkresiduesymbolsVYfroms}Let $\mathcal{P}$ be a prime ideal of a real biquadratic field $K=\QQ(\sqrt{d_1},\sqrt{d_2})$ lying above an odd rational prime $p$ that is not ramified and not totally decomposed in $K$. Then for  a quadratic subfield $k=\QQ(\sqrt{d})$ of $K$, the decomposition of $p$ takes one of the following cases:
 	
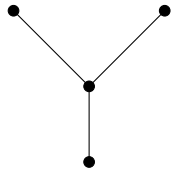
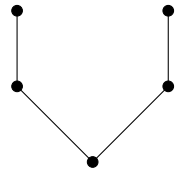
\begin{center}

 		\begin{tikzpicture}
 			
 			\node at (-3.5,3) (nodeA) {$K$};
 			\node at (-3.5,2)(nodeB) {$k$};
 			\node at (-3.5,1) (nodeQ) {$\QQ$};
 			
 			\draw  (nodeA) -- (nodeB);
 			\draw  (nodeB) -- (nodeQ);
 			
 			\filldraw (-1,3) circle (2pt);
 			\filldraw (1,3) circle (2pt);
 			
 			\filldraw (0,2) circle (2pt);
 			\draw (-1,3) -- (0,2) -- (1,3);
 			\filldraw (0,1) circle (2pt);
 			\draw  (0,2) -- (0,1);
 			
 			\node at (0,0) (nodeA) {\text{In this case: $p$  is inert in $k$}};
 			\filldraw (5,3) circle (2pt);
 			\filldraw (7,3) circle (2pt);
 			\filldraw (5,2) circle (2pt);
 			\filldraw (7,2) circle (2pt);
 			\filldraw (6,1) circle (2pt);
 			\draw (5,2) -- (6,1) -- (7,2);
 			\draw (5,2) -- (5,3);
 			\draw (7,2) -- (7,3);
 			\node at (7,0) (nodeA) {\text{In this case: $p$  decomposes in $k$}};
 		\end{tikzpicture}
 	\end{center}
 	 Let us say that the {\bf decomposition of $p$ in $K$ takes the $\Y$  form for $k$} if we are in the left side case and let us say that {\bf the decomposition of $p$ in $K$ takes the $\V$  form for $k$} if we are in the right side case. Let $u$ be unit of $K$. Then by the properties of the norm residue symbol    we have:
 	\begin{enumerate}[$1)$]
 		\item  If $u\in k$ and the decomposition $p$ in $K$ takes the  $\Y$ form for $k$ then we have :
 		$$\left(\frac{u,\,p}{ \mathcal P}\right) =\left(\frac{N_{k/\QQ}(u) }{ p}\right).$$
 		
 		In fact, since the prime ideal $ \mathcal P_{k}:=\mathcal P\cap k$ of $k$ above $p$ is totally decomposed in $K$ and  $u\in k$, this implies that 
 		$\left(\frac{u,\,p}{ \mathcal P}\right)=\left(\frac{u}{\mathcal P_{k}}\right)$. As $\mathcal P_{k}$ the unique prime ideal of $k$ laying above $p$,  
 		$\left(\frac{u,\,p}{\mathcal P_{k}}\right)=\left(\frac{N_{k/\QQ}(u) }{ p}\right)$, which gives the above equality. Similarly, by using the properties of norm residue symbols,  we have: 		
 		\item  If $u\in k$ and the decomposition $p$ in $K$ takes the  $\V$ form for $k$ then we have :
 		$$\left(\frac{u,\,p}{ \mathcal P}\right) =1.$$
 		
 		\item   If $u\not\in k$ and the decomposition $p$ in $K$ takes the  $\V$ form for $k$ then we have :
 		$$\left(\frac{u,\,p}{ \mathcal P}\right) =\left(\frac{N_{K/k}(u),\,p}{ \mathcal P_{k}}\right),$$
 		where   $ \mathcal P_{k}=\mathcal P\cap k$. If furthermore  $N_{K/k}(u)\in \QQ$, then $\left(\frac{u,\,p}{ \mathcal P}\right) =\left(\frac{N_{K/k}(u)}{ p}\right)$.
 		\item If $u\in \QQ$ then 
 		$$\left(\frac{u,\,p}{ \mathcal P}\right)= 1.$$
 	\end{enumerate}

 \end{remark}

 Let $L$ be the biquadratic field named in the introduction. We have:
 
 \begin{lemma}\label{realTruquad2-rankq} Let $K= \QQ(\sqrt{q},\sqrt{d})\not=L$ and  $K_1= \QQ(\sqrt{q},\sqrt{d}, \sqrt{2})$  where  $q\equiv 3\pmod 4$ is a prime number and $d > 2$  is square free integer relatively prime to  $q$.

 	\noindent $\mathbf{1)}$	Assume that $K_1= \QQ(\sqrt{q},\sqrt{d}, \sqrt{2})$, with $d=r$ is an odd prime number.  
 	We have  $$\rg(A(K_1))\in\{0,1\}.$$ More precisely:
 	$	\rg(A(K_1)) =0$ if and only if  we have      one of the following conditions :  
 	\begin{enumerate}[\indent$\bullet$]
 		\item $r\equiv 3$ or $5\pmod 8$, 
 		\item $r\equiv 7\pmod 8$ and  $q\equiv 3\pmod 8$.
 	\end{enumerate}

 	\noindent $\mathbf{2)}$	Assume that $K_1= \QQ(\sqrt{q},\sqrt{d}, \sqrt{2})$ with $d=rs$, where $r$ and $s$ are two prime numbers.
 	We have:  $$\rg(A(K_1))\in\{1,2,3,4\}.$$
 	More precisely,  we have:
 	
 	\noindent $\rg(A(K_1))=1$ if and only if, after a suitable permutation of $r$ and $s$, we have one of the following conditions:
 	
 	\begin{enumerate}[$\C1:$]
 		\item $ \left(\frac{q}{s}\right)=\left(\frac{q}{r}\right)=-1$ and we have one of the following congruence conditions:
 		\begin{enumerate}[\indent$\bullet$]
 			\item $r\equiv     5\pmod 8$ and  $s\equiv     3\pmod 8$,  
 			
 			\item $r\equiv     3$ or $5\pmod 8$,  $s\equiv    7\pmod 8$  and $q\equiv 3\pmod 8$.
 		\end{enumerate}

 		\item 	$ \left(\frac{q}{r}\right)=-\left(\frac{q}{s}\right)=-1$   and we have one of the following congruence conditions:
 		\begin{enumerate}[\indent$\bullet$]
 			\item $r\equiv    5 \pmod 8$ and  $s\equiv   5$ or $3\pmod 8$,
 			\item   $r\equiv     3 \pmod 8$,  $s\equiv    3\pmod 8$  and $q\equiv 7\pmod 8$,  
 			\item $r\equiv    3 \pmod 8$ and  $s\equiv   5\pmod 8$,

 		\item    $r\equiv   7\pmod 8$,  $s\equiv  3$ or $5\pmod 8$ and   $q\equiv 3\pmod 8$,
 		
 		\item     $r\equiv  3$ or $5\pmod 8$, $s\equiv   7\pmod 8$ and $q\equiv 3\pmod 8$.	
 			
 		\end{enumerate}
 		
 		\item  $ \left(\frac{q}{r}\right)=\left(\frac{q}{s}\right)=1$  and we have one of the following congruence conditions:
 		\begin{enumerate}[\indent$\bullet$]
 			\item   $r\equiv    7\pmod 8$, $s\equiv     3$ or $5\pmod 8$,  and $q\equiv 3\pmod 8$,
 			
 			\item $r\equiv     3 \pmod 8$ and  $s\equiv    5\pmod 8$.
 		\end{enumerate}
 	\end{enumerate}

 	\noindent $\rg(A(K_1))=2$ if and only if, after a suitable permutation of $r$ and $s$, we have one of the following conditions:
 	
 	\begin{enumerate}[$\C1:$]
 		\item $ \left(\frac{q}{s}\right)=\left(\frac{q}{r}\right)=-1$ and we have one of the following congruence conditions:
 		\begin{enumerate}[\indent$\bullet$]
 			\item $r\equiv     3$ or $5\pmod 8$,  $s\equiv    7\pmod 8$  and $q\equiv 7\pmod 8$,
 			
 			\item $r\equiv s\equiv 3$ or $5\pmod 8$,
 			
 			\item $r\equiv     3$ or $5\pmod 8$  and  $s\equiv    1\pmod 8$,
 			\item   $r\equiv    7\pmod 8$,  $s\equiv     7$ or $1\pmod 8$ and $q\equiv 3\pmod 8$.
 			
 		\end{enumerate}

 		\item 	$ \left(\frac{q}{r}\right)=-\left(\frac{q}{s}\right)=-1$  and we have one of the following congruence conditions:
 		\begin{enumerate}[\indent$\bullet$]
 			\item $r\equiv     7\pmod 8$  and  $s\equiv    7\pmod 8$  and $q\equiv 3\pmod 8$,      
 			\item $r\equiv     3\pmod 8$  and  $s\equiv    3\pmod 8$  and $q\equiv 3\pmod 8$,
 			\item   $r\equiv   1\pmod 8$,  $s\equiv   7\pmod 8$ and $q\equiv 3\pmod 8$,
 			
 			\item   $r\equiv   1\pmod 8$ and  $s\equiv   3$ or $5\pmod 8$,
 			
 			\item    $r\equiv   7\pmod 8$,  $s\equiv  3$ or $5\pmod 8$ and $q\equiv 7\pmod 8$,
 			
 			\item    $r\equiv  3$ or $5\pmod 8$, $s\equiv   7\pmod 8$ and $q\equiv 7\pmod 8$.

 		\end{enumerate}
 		
 		\item  $ \left(\frac{q}{r}\right)=\left(\frac{q}{s}\right)=1$  and we have one of the following congruence conditions:
 		\begin{enumerate}[\indent$\bullet$]
 			\item $r\equiv s\equiv 3$ or $5\pmod 8$,
 			
 			\item  	$r\equiv     7\pmod 8$,  $s\equiv    3$ or $5\pmod 8$  and $q\equiv 7\pmod 8$. 
 		\end{enumerate}
 	\end{enumerate}
 \end{lemma}
 \begin{proof}
 	As the class number of $K_1'=\QQ(\sqrt{q},\sqrt{2}) $ is odd, we get  $\rg(A(K_1))=t_{K_1/K_1'}-1-e_{K_1/K_1'}$. 
 	By combining \cite[Lemma 2.5]{Chemsarith} and  \cite[Lemma 3.2]{CEH2024}, we have
 	$\sqrt{\varepsilon_{q}}=\frac{\sqrt{2}}{2}(x+y\sqrt{q})   $ (resp. $\sqrt{\varepsilon_{2q}}=\frac{\sqrt{2}}{2}(\alpha+\beta\sqrt{2q}))$ where $x$ and $y$ $($resp. $\alpha$ and $\beta )$ are two integers such that $(-1)^{\gamma}2=x^2-y^2q$ $($resp.  $(-1)^{\gamma}2=\alpha^2-\beta^22q)$, where $\gamma=0$ or $1$ according to whether 
 	$q\equiv 7\pmod 8$ or $q\equiv 3\pmod 8$.
 	  Notice that $E_{K_1'}=\langle-1, \vep_{2},\sqrt{\vep_{q}},\sqrt{\vep_{2q}} \rangle $.
 	\begin{enumerate}[$1)$]
 		\item Let $r$ be an odd prime number and $\mathcal R$   a prime ideal of $K_1'$   lying above $r$. By the previous remark, if $r$ is not totally decomposed in $K_1'$, then  $\left(\dfrac{-1,\,r}{ \mathcal R}\right)=1 $,   else we have $\left(\dfrac{-1,\,r}{ \mathcal R}\right)=\left(\dfrac{-1 }{ r}\right) $. Furthermore, we have: 
 		\begin{enumerate}[$\bigstar$]
 			\item Assume that $r\equiv 3$ or $5\pmod 8$. Then the decomposition of $r$ in $K_1'$ takes the $\Y$ form for $\QQ(\sqrt{2})$ and so
 			\begin{eqnarray}
 				\left(\frac{\vep_2,\,r}{ \mathcal R}\right)&=&\left(\frac{-1}{ r}\right)=(-1)^{\frac{r-1}{2}} .   
 			\end{eqnarray}
 			For the other symbols we distinguish the following two subcases:\\
 			\noindent$\bullet$ If $ \left(\frac{q}{r}\right)=-1$, then the decomposition of $r$ in $K_1'$ takes the $\V$ form for $\QQ(\sqrt{2q})$. As 
 			$N_{K_1'/\QQ(\sqrt{2q})}(\sqrt{\vep_{ q}})= -\frac{1}{{2}}(x^2-y^2q)         =(-1)^{\gamma+1}\in \QQ$ and  $N_{K_1'/\QQ(\sqrt{2q})}(\sqrt{\vep_{ 2q}})=-\vep_{ 2q}$, then the previous remark gives:
 			\begin{eqnarray}
 				\left(\frac{\sqrt{\vep_{ q}},\,r}{ \mathcal R}\right)&=& \left(\frac{(-1)^{\gamma+1} }{r}\right)   ,\\
 				& &\nonumber\\
 				\left(\frac{\sqrt{\vep_{2q}},\,r}{ \mathcal R}\right)&=&\left(\frac{ -\vep_{ 2q},\,r}{ \mathcal R_{\QQ(\sqrt{2q})}}\right)=\left(\frac{-1}{ r}\right)\left(\frac{ 2}{ r}\right)  .
 			\end{eqnarray}
 			The last equality follows from the fact that $2\vep_{ 2q}=\nu^2$ for some $\nu  \in \QQ(\sqrt{2q})$ $($in fact: $\sqrt{\varepsilon_{2q}}=\frac{\sqrt{2}}{2}(\alpha+\beta\sqrt{2q})))$.

 			\noindent $\bullet$ Now if  $ \left(\frac{q}{r}\right)=1$,	then the decomposition of $r$ in $K_1'$ takes the $\V$ form for $\QQ(\sqrt{q})$. As 
 			$N_{K_1'/\QQ(\sqrt{q})}(\sqrt{\vep_{ q}})=  -\vep_{ q} $ and  $N_{K_1'/\QQ(\sqrt{q})}(\sqrt{\vep_{ 2q}})= (-1)^{\gamma+1}\in \QQ$, then as in the previous case and using the above remark we get: 
 			
 			\begin{eqnarray}
 				\left(\frac{\sqrt{\vep_{q}},\,r}{ \mathcal R}\right)&=& \left(\frac{ -\vep_{ q},\,r}{ \mathcal R_{\QQ(\sqrt{2q})}}\right)=\left(\frac{-1 }{ r}\right)\left(\frac{ 2}{ r}\right)  	\\
 				& &\nonumber\\
 				\left(\frac{\sqrt{\vep_{ 2q}},\,r}{ \mathcal R}\right)&=& \left(\frac{(-1)^{\gamma+1} }{r}\right)   .
 			\end{eqnarray}

 			Now one can deduces easily that  $\rg(A(K_1))=2-1-e_{K_1/K_1'}=1-e_{K_1/K_1'}=0$.

 			\item Assume that $r\equiv 7$ or $1\pmod 8$ with $ \left(\frac{q}{r}\right)=-1$.  Then the decomposition of    $r$ in $K_1'$ takes the $\V$ form for $\QQ(\sqrt{2})$ and so the above remark gives:
 			
 			\begin{eqnarray}
 				\left(\frac{\vep_2,\,r}{ \mathcal R}\right)&=&1.
 			\end{eqnarray}
 			As 
 			$N_{K_1'/\QQ(\sqrt{2})}(\sqrt{\vep_{q}})=  \frac{1}{{2}}(x^2-y^2q)         =(-1)^{\gamma}\in \QQ$ and  $N_{K_1'/\QQ(\sqrt{2})}(\sqrt{\vep_{2q}})=  \frac{1}{{2}}(\alpha^2-\beta^22q)         =(-1)^{\gamma}\in \QQ$,   the previous remark gives:
 			\begin{eqnarray}
 				\left(\frac{\sqrt{\vep_{ q}},\,r}{ \mathcal R}\right)&=& \left(\frac{(-1)^{\gamma} }{r}\right)   ,\\
 				& &\nonumber\\
 				\left(\frac{\sqrt{\vep_{ 2q}},\,r}{ \mathcal R}\right)&=& \left(\frac{(-1)^{\gamma} }{r}\right).
 			\end{eqnarray}
 			Therefore, in this subcase, we have  $\rg(A(K_1))=2-1-e_{K_1/K_1'}=0$ or $1$ according to whether $(r\equiv 7\pmod 8$ and  $q\equiv 3\pmod 8$) or not.
 			
 				\item Assume that $r\equiv 7\pmod 8$ with  $ \left(\frac{q}{r}\right)=1$. Then $r$ is totally decomposed in $K_1'$. 
 			Therefore, in this subcase, we have  $\rg(A(K_1))=3-e_{K_1/K_1'}$. 		Let $\mathcal R$ (resp. $\mathcal R'$) be an ideal  of $K_1'$ above the ideal    $\mathcal R_{\QQ(\sqrt{2})}$ (resp. $\mathcal R'_{\QQ(\sqrt{2})}$)  of $\QQ(\sqrt{2})$ as in Remark \ref{scholsrec}-2), for $p=2$.

 			Thus, we have  
 				
 				\begin{eqnarray}
 					\left(\dfrac{-1,\,r}{ \mathcal R}\right)=\left(\dfrac{-1,\,r}{ \mathcal R'}\right)=-1, \quad	 \left(\frac{\vep_2,\,r}{ \mathcal R}\right) = 1  \text{ and }  \left(\frac{\vep_2,\,r}{ \mathcal R'}\right) = -1.
 				\end{eqnarray}
 				So  $e_{K_1/K_1'}\geq 2$ and $\rg(A(K_1))=3-e_{K_1/K_1'}\leq 1$. Put $M=\QQ(\sqrt{2r},\sqrt{q})$. $K_1/M$ is an unramified quadratic extension and 
 				$  \rg(A(M))=1-e_{M/\QQ(\sqrt{2r})}\leq 1$. So $A(M)$ is cyclic and this implies that $A(K_1)$ is cyclic. By class number formula and Lemma \ref{realQuad}, we have 
 				$h_2(M)=\frac14q(M)h_2(q)h_2(2r)h_2(2qr)=\frac14q(M)h_2(2qr)$. Note that $q(M)=4$  (cf. \cite[Corollaries  2]{CAZbol} and we check similarly that $q(M)=4$ for $q\equiv7\pmod8$).  
 				According to \cite[p. 345]{kaplan76} and \cite[Corollaries 19.7 and 18.4]{connor88}, $h_2(2qr)=2$ if and only if $q\equiv 3\pmod 8$, else $h_2(2qr) $ is divisible by $4$. Therefore $\rg(A(K_1))=0$ if and only if  $q\equiv 3\pmod 8$, else $\rg(A(K_1))=1$.  			It follows that $e_{K_1/K_1'}\in \{2,3\}$, more precisely, $e_{K_1/K_1'}=2$ if and only if  $q\equiv 7\pmod 8$. 
 				
 				\begin{remark}\label{remqym7mod8} Let $u\in\{\sqrt{\vep_{ q}},\sqrt{\vep_{ 2q}} \}$. In the case when  $r\equiv 7\pmod 8$   with $ \left(\frac{q}{r}\right)=1$, it seems that the computation of 
 					$\left(\frac{u,\,r}{ \mathcal R}\right)$ and $\left(\frac{u,\,r}{ \mathcal R'}\right)$ is  difficult. But from the above paragraph 
 					we deduce that 
 					\begin{enumerate}[$a)$]
 						\item If   $q\equiv 3\pmod 8$, then there exist at least  $u\in\{\sqrt{\vep_{ q}},\sqrt{\vep_{ 2q}} \}$ such that 	$\left(\frac{u,\,r}{ \mathcal R}\right)=-1$ or $\left(\frac{u,\,r}{ \mathcal R'}\right)=-1$, and $\overline{u}\not\in \{\overline{-1}, \overline{\vep_{2}}\}$ in the quotient  $E_{K_1'}/(E_{K_1'}\cap N_{K_1/K_1'}(K_1))$. In fact, we have    $e_{K_1/K_1'}=3$ and by the above signs of norm residue symbols  we have $-1$ and $\vep_2$ are not norms and 
 						$\overline{-1}\not= \overline{\vep_{2}}$ in the quotient  $E_{K_1'}/(E_{K_1'}\cap N_{K_1/K_1'}(K_1))$. 
 						Thus, by necessity, there must be at least one other non-trivial class in   $E_{K_1'}/(E_{K_1'}\cap N_{K_1/K_1'}(K_1))$ that is represented by $u\in\{\sqrt{\vep_{ q}},\sqrt{\vep_{ 2q}} \}$. Similarly, we have:

 						\item 	If   $q\equiv 7\pmod 8$, then for all   $u\in\{\sqrt{\vep_{ q}},\sqrt{\vep_{ 2q}} \}$, we have $\overline{u} \in \{\overline{1},\overline{-1}, \overline{\vep_{2}}\}$ in the quotient  $E_{K_1'}/(E_{K_1'}\cap N_{K_1/K_1'}(K_1))$.
 					\end{enumerate}
 				\end{remark}
 			\end{enumerate}
 		\item We shall use the above signs of norm residue symbols.
 		\begin{enumerate}[\rm$\star$]
 			\item Assume that $ \left(\frac{q}{s}\right)=\left(\frac{q}{r}\right)=-1$. We have $\rg(A(K_1))= 4-1-e=3-e_{K_1/K_1'} $ and  
 			\begin{enumerate}[\indent$a)$]
 				
 				\item    If $r\equiv 5\pmod 8$ and $s\equiv 3\pmod 8$, then $\left(\frac{\sqrt{\vep_{2q}},\,r}{ \mathcal R}\right)=\left(\frac{-1}{ r}\right)\left(\frac{ 2}{ r}\right)=-1$, 
 				$\left(\frac{\vep_2,\,s}{ \mathcal S}\right)=(-1)^{\frac{s-1}{2}}=-1$ and 
 				$\left(\frac{\vep_2\sqrt{\vep_{2q}},\,r}{ \mathcal R}\right)= \left(\frac{\vep_2 ,\,r}{ \mathcal R}\right)\left(\frac{ \sqrt{\vep_{2q}},\,r}{ \mathcal R}\right)= -1$. Thus, $e_{K_1/K_1'}\geq 2$ and so $\rg(A(K_1))=1$.

 				\item   If $r\equiv 3\pmod 8$ and $s\equiv 7\pmod 8$, then we have
 				$- \left(\frac{-1,\,r}{ \mathcal R}\right)=	\left(\frac{\vep_2,\,r}{ \mathcal R}\right)=-1$, 
 				$	\left(\frac{\sqrt{\vep_{ q}},\,r}{ \mathcal R}\right)= \left(\frac{(-1)^{\gamma+1} }{r}\right)   $ and 
 				$	\left(\frac{\sqrt{\vep_{2q}},\,r}{ \mathcal R}\right)=\left(\frac{-1}{ r}\right)\left(\frac{ 2}{ r}\right) =1$. 
 				Moreover,		
 				$ \left(\frac{-1,\,s}{ \mathcal S}\right)=	\left(\frac{\vep_2,\,s}{ \mathcal S}\right)=1$,
 				$	\left(\frac{\sqrt{\vep_{ q}},\,s}{ \mathcal S}\right)= \left(\frac{(-1)^{\gamma} }{s}\right)$  
 				and
 				$\left(\frac{\sqrt{\vep_{ 2q}},\,s}{ \mathcal S}\right)= \left(\frac{(-1)^{\gamma} }{s}\right)$. 
 				Therefore,	$e_{K_1/K_1'}\in \{1,2\}$  and $e_{K_1/K_1'}=2$ if and only if  	$q\equiv 3 \pmod 8$.
 				Hence,  $\rg(A(K_1))\in \{1,2\}$  and $\rg(A(K_1))=1$ if and only if  	$q\equiv 3 \pmod 8$.		 	 Similarly, we have:

 				\item   If $r\equiv 5\pmod 8$ and $s\equiv 7\pmod 8$, then we have $\rg(A(K_1))\in \{1,2\}$  and $\rg(A(K_1))=1$ if and only if  	$q\equiv 3 \pmod 8$.
 				
 				\item   If $r\equiv s\equiv 3$ or $5\pmod 8$, then 		 $\rg(A(K_1))=2$.	 		
 				
 				\item   If $r\equiv 3\pmod 8$ and $s\equiv 1\pmod 8$, then we have	$\rg(A(K_1))=2$.
 				
 				\item   If $r\equiv 5\pmod 8$ and $s\equiv 1\pmod 8$, then we have	$\rg(A(K_1))=2$.
 				
 				\item   If $r\equiv 7\pmod 8$ and $s\equiv 7\pmod 8$, then we have $\rg(A(K_1))\in \{2,3\}$  and $\rg(A(K_1))=2$ if and only if  	$q\equiv 3 \pmod 8$.
 				
 				\item   If $r\equiv 7\pmod 8$ and $s\equiv 1\pmod 8$, then we have $\rg(A(K_1))\in \{2,3\}$  and $\rg(A(K_1))=2$ if and only if  	$q\equiv 3 \pmod 8$.
 				
 				\item   If $r\equiv 1\pmod 8$ and $s\equiv 1\pmod 8$, then we have $\rg(A(K_1))=3$.
 			\end{enumerate}

 			\item   Assume that  $ \left(\frac{q}{r}\right)=-\left(\frac{q}{s}\right)=-1$  and $s\not \equiv 1\pmod 8$. Notice that $\rg(A(K_1))=5-e_{K_1/K_1'}$ or $3-e_{K_1/K_1'}$ according to whether $s\equiv 7\pmod 8$ or not. We have:
 			\begin{enumerate}[\indent$a)$]

 				\item  If $r\equiv   s\equiv 7\pmod 8$, then, by using the above symbols, we get $e_{K_1/K_1'}=3$ or $2$, and $e_{K_1/K_1'}=3$ if and only if $q\equiv 3\pmod 8$. Thus, 
 				$\rg(A(K_1))\in\{2,3\}$ and $\rg(A(K_1))=2$ if and only if $q\equiv 3\pmod 8$.

 				\item If  $r\equiv   s\equiv 3\pmod 8$, then $\rg(A(K_1))\in\{1,2\}$ and $\rg(A(K_1))=1$
 				if and only if $q\equiv 7\pmod 8$.
 				
 				\item If  $r\equiv   s\equiv 5\pmod 8$, then  $\rg(A(K_1))=1$.
 				
 				 \item  If $r\equiv   1\pmod 8$ and  $s\equiv   7\pmod 8$, then  
 				$\rg(A(K_1))\in\{2,3\}$ and $\rg(A(K_1))=2$ if and only if $q\equiv 3\pmod 8$.
 				
 				 \item  If $r\equiv   1\pmod 8$ and  $s\equiv   3$ or $5\pmod 8$, then  
 				 $\rg(A(K_1))=2$.  
 				 
 				 \item  If $r\equiv   7\pmod 8$ and  $s\equiv  3$ or $5\pmod 8$, then  
 				 $\rg(A(K_1))\in\{1,2\}$ and $\rg(A(K_1))=1$ if and only if $q\equiv 3\pmod 8$.
 				 
 				  \item  If     $r\equiv  3$ or $5\pmod 8$ and $s\equiv   7\pmod 8$, then  
 				 $\rg(A(K_1))\in\{1,2\}$ and $\rg(A(K_1))=1$ if and only if $q\equiv 3\pmod 8$.
 				
 				\item   If  $(r\equiv     3\pmod 8$ and  $   s\equiv 5\pmod 8)$ or $(r\equiv     5\pmod 8$ and  $    s\equiv 3\pmod 8)$, then   $\rg(A(K_1))=1$.
 				
 			\end{enumerate}

 			\item   Assume that  $ \left(\frac{q}{r}\right)=\left(\frac{q}{s}\right)=1$. 	
 			\begin{enumerate}[\indent$a)$]
 				\item   If $r\equiv 7\pmod 8$ and $s\equiv 3$ or $5\pmod 8$, then $\rg(A(K_1))\in\{1,2\}$ and $\rg(A(K_1))=1$
 				if and only if $q\equiv 3\pmod 8$.

 				\item   If $r\equiv 3\pmod 8$ and $s\equiv 5\pmod 8$, then  $\rg(A(K_1))=1$.
 				\item    If $r\equiv  s\equiv 7\pmod 8$, then    by Remark \ref{remqym7mod8},  we have   $\rg(A(K_1))\in\{3,4\}$. 
 			\end{enumerate}
 			
 		\end{enumerate}
 		
 	\end{enumerate}	
 	This completes the proof.	
 \end{proof}

 {
 \begin{lemma}\label{lmthreeprimes}   Let  $q\equiv 3\pmod4$ be a prime number and   $d=rst$ be a squarefree integer  such that $r$, $s$ and $t$ are three different prime numbers. Let $K= \QQ(\sqrt{q},\sqrt{d})$  be such that 
 	$\rg(K)=2$ $($cf. Lemma \ref{lemmaqA}-(3)$)$. We have, one of the following cases:  
 	 \begin{enumerate}[$1)$]
 	 	\item $d=rst$, with $r\equiv 3  \pmod 8$ and $s\equiv  5\pmod 8$  and $\left(\frac{q}{r}\right)=\left(\frac{q}{s}\right)=-\left(\frac{q}{t}\right)=1$. In this case, $\rg(A(K_1))\in\{2,3\} $, more precisely, 
 	 	$\rg(A(K_1))=2$ if and only if we have one of the following conditions:
 	 	\begin{enumerate}[$\star$]
 	 		\item $t\equiv 5  \pmod 8$,
 	 		
 	 		\item $t\equiv 3  \pmod 8$ and $q\equiv 7\pmod 8$,
 	 		
 	 		\item $t\equiv 7  \pmod 8$ and $q\equiv 3\pmod 8$.

 	 	\end{enumerate}
 	 	\item  	$d=rst$, with  $r\equiv 3  \pmod 8$, $s\equiv  7\pmod 8$ and $\left(\frac{q}{r}\right)=\left(\frac{q}{s}\right)=-\left(\frac{q}{t}\right)=1$. In this case, $\rg(A(K_1))\in\{3,4\} $, more precisely, $\rg(A(K_1))=4$ if and only if we have one of the following conditions:
 	 	\begin{enumerate}[$\star$]
 	 		\item  $t\equiv 1  \pmod 8$ and $q\equiv 7\pmod 8$,
 	 		
 	 		\item  $t\equiv 7  \pmod 8$ and $q\equiv 7\pmod 8$.
 	 		
 	 	\end{enumerate}

 	 	\item  $d=rst$, with  $r\equiv 7  \pmod 8$  and $-\left(\frac{q}{r}\right)=\left(\frac{q}{s}\right)=\left(\frac{q}{t}\right)=-1$. In this case, $\rg(A(K_1))\in\{3,4,5\} $, more precisely, $\rg(A(K_1))=5$  if and only if, after a suitable permutation of $s$ and $t$, we have one of the following conditions:
 	 	\begin{enumerate}[$\star$]
 	 		\item  $s\equiv 1  \pmod 8$, $t\equiv 1  \pmod 8$ and $q\equiv 7\pmod 8$,
 	 		
 	 		\item  $s\equiv7  \pmod 8$, $t\equiv7  \pmod 8$ and $q\equiv 7\pmod 8$.
 	  	\end{enumerate}
 	 	{\bf Moreover,}  	$\rg(A(K_1))=3$  if and only if, after a suitable permutation of $s$ and $t$, we have one of the following conditions:
 	 	\begin{enumerate}[$\star$]
 	 		\item $s\equiv 3  \pmod 8$, $t\equiv 3  \pmod 8$ and $q\equiv 3\pmod 8$,
 	 		
 	 		\item $s\equiv 5  \pmod 8$, $t\equiv 5  \pmod 8$ and $q\equiv 3\pmod 8$,
 	 		\item $s\equiv 3  \pmod 8$, $t\equiv 5  \pmod 8$ and $q\equiv 3\pmod 8$,
 	 		\item $s\equiv 7  \pmod 8$, $t\equiv 5  \pmod 8$ and $q\equiv 3\pmod 8$,
 	 		\item $s\equiv 1  \pmod 8$, $t\equiv 5  \pmod 8$ and $q\equiv 3\pmod 8$,
 	 		\item $s\equiv 7  \pmod 8$, $t\equiv 3  \pmod 8$ and $q\equiv 3\pmod 8$,
 	 		\item $s\equiv 1  \pmod 8$, $t\equiv 3  \pmod 8$ and $q\equiv 3\pmod 8$.
 	 	\end{enumerate}

 	 	\item  $d=rst$, with  $r\equiv 5  \pmod 8$  and $-\left(\frac{q}{r}\right)=\left(\frac{q}{s}\right)=\left(\frac{q}{t}\right)=-1$. In this case, $\rg(A(K_1))\in\{2,3,4\} $, more precisely, $\rg(A(K_1))=2$ if and only if, after a suitable permutation of $s$ and $t$, we have one of the following conditions:  
 	 	\begin{enumerate}[$\star$]
 	 		\item  $s\equiv 3  \pmod 8$ and  $t\equiv 5  \pmod 8$,
 	 		
 	 		\item   $s\equiv 7  \pmod 8$,   $t\equiv 3  \pmod 8$ and $q\equiv 3\pmod 8$.
 	 	\end{enumerate}
 	 	{\bf Moreover,} 	$\rg(A(K_1))=4$ if and only if, after a suitable permutation of $s$ and $t$, we have one of the following conditions:
 	 	\begin{enumerate}[$\star$]
 	 		\item $s\equiv 1  \pmod 8$ and  $t\equiv 1  \pmod 8$,
 	 		
 	 		\item  $s\equiv 7  \pmod 8$, $t\equiv 1  \pmod 8$ and $q\equiv 7\pmod 8$,
 	 		
 	 		\item  $s\equiv 7  \pmod 8$, $t\equiv 7  \pmod 8$ and $q\equiv 7\pmod 8$.
 	 	\end{enumerate}

 	 	\item  $d=rst$, with  $r\equiv 3  \pmod 8$  and $-\left(\frac{q}{r}\right)=\left(\frac{q}{s}\right)=\left(\frac{q}{t}\right)=-1$. In this case, $\rg(A(K_1))\in\{2,3,4\} $, more precisely, 
 	 	$\rg(A(K_1))=2$ if and only if, after a suitable permutation of $s$ and $t$, we have one of the following conditions:  
 	 	\begin{enumerate}[$\star$]
 	 		\item  $s\equiv 3  \pmod 8$, $t\equiv 5  \pmod 8$ and $q\equiv 7\pmod 8$,
 	 		
 	 		\item    $s\equiv 5  \pmod 8$, $t\equiv 7  \pmod 8$ and $q\equiv 3\pmod 8$.
 	 	\end{enumerate}
 	 	{\bf Moreover,} 	$\rg(A(K_1))=4$  if and only if, after a suitable permutation of $s$ and $t$, we have one of the following conditions:
 	 	\begin{enumerate}[$\star$]
 	 		\item   $s\equiv 1  \pmod 8$  and $t\equiv 1\pmod 8$,
 	 		
 	 		\item  $s\equiv 7  \pmod 8$, $t\equiv 7  \pmod 8$ and $q\equiv 7\pmod 8$,
 	 		
 	 		\item  $s\equiv 7  \pmod 8$, $t\equiv 1  \pmod 8$ and $q\equiv 7\pmod 8$,
 	 		
 	 		\item   $s\equiv 3  \pmod 8$, $t\equiv 3  \pmod 8$ and $q\equiv 3\pmod 8$.
 	 	\end{enumerate}
 	 		\item $d=rst$, with    $\left(\frac{q}{r}\right)=\left(\frac{q}{s}\right)= \left(\frac{q}{t}\right)=-1$. In this case, $\rg(A(K_1))\in\{2,3,4,5\} $ and  $\rg(A(K_1))=2$ if and only if, after a suitable permutation of $r$, $s$ and $t$, we have $r\equiv 3\pmod 8$, $s\equiv 5\pmod 8$, $t\equiv 7\pmod 8$ and $q\equiv 3\pmod 8$.
 	 \end{enumerate}
 \end{lemma}
 \begin{proof}
Let  $K_1'=\QQ(\sqrt{q},\sqrt{2}) $. Let us keep the same notations as in the proof of the previous lemma. We have  $\rg(A(K_1))=t_{K_1/K_1'}-1-e_{K_1/K_1'}=5-e_{K_1/K_1'}$.
 \begin{enumerate}[$\bullet$]
	\item Assume that $r\equiv 3  \pmod 8$, $s\equiv  5\pmod 8$  and $\left(\frac{q}{r}\right)=\left(\frac{q}{s}\right)=-\left(\frac{q}{t}\right)=1$.
	By the signs of the norm residue symbols    in the proof of   the previous lemma and the product formula, we have:
 	{\footnotesize$$\begin{array}{l|l|l|lll}
 			& &t\equiv 3\text{ or }5 \pmod 8 &t\equiv 1\text{ or }7 \pmod 8\\
		\left(\frac{{-1},\,d}{ \mathcal R}\right)=1&	\left(\frac{{-1},\,d}{ \mathcal S}\right)=1&  	\left(\frac{{-1},\,d}{ \mathcal T}\right)=1 &	\left(\frac{{-1},\,d}{ \mathcal T}\right)=1 \\
		\left(\frac{{\vep_{2}},\,d}{ \mathcal R}\right)= -1   &	\left(\frac{{\vep_{2}},\,d}{ \mathcal S}\right)= 1 & \left(\frac{{\vep_{2}},\,d}{ \mathcal T}\right)= \left(\frac{-1 }{t}\right) & \left(\frac{{\vep_{2}},\,d}{ \mathcal T}\right)=1 \\
		\left(\frac{\sqrt{\vep_{ q}},\,d}{ \mathcal R}\right)= 1   &	\left(\frac{\sqrt{\vep_{ q}},\,d}{ \mathcal S}\right)= -1  &  \left(\frac{\sqrt{\vep_{ q}},\,d}{ \mathcal T}\right)= \left(\frac{(-1)^{\gamma+1} }{t}\right) &\left(\frac{\sqrt{\vep_{ q}},\,d}{ \mathcal T}\right)= \left(\frac{(-1)^{\gamma} }{t}\right) \\
		
		\left(\frac{\sqrt{\vep_{2q}},\,d}{ \mathcal R}\right)= \left(\frac{(-1)^{\gamma+1} }{r}\right)   &	\left(\frac{\sqrt{\vep_{2q}},\,d}{ \mathcal S}\right)= 1  &  \left(\frac{\sqrt{\vep_{2q}},\,d}{ \mathcal T}\right)= \left(\frac{-1 }{t}\right)\left(\frac{2}{t}\right)&\left(\frac{\sqrt{\vep_{2q}},\,d}{ \mathcal T}\right)= \left(\frac{(-1)^{\gamma} }{t}\right) \\
	\end{array}$$}
 	where $\mathcal R$ (resp. $\mathcal S$, $\mathcal T$) is a prime ideal of $K_1'$ above $r$ (resp. $s$, $t$).
	\begin{enumerate}[$\star$]
	\item If $t\equiv 5\pmod 8$, then using the above norm residue symbols we check that 
	$E_{K_1'}/(E_{K_1'}\cap N_{K_1/K_1'}(K_1))
	=\{\overline{1},\overline{\vep_{2}},\overline{\sqrt{\vep_{q}}},\overline{\sqrt{\vep_{2q}}}, \overline{\sqrt{\vep_{q}\vep_{2q}}},\overline{\vep_{2}\sqrt{\vep_{q}}}, \overline{\vep_{2}\sqrt{\vep_{2q}}}, \overline{\vep_{2}\sqrt{\vep_{q}\vep_{2q}}}\}$ and $e_{K_1/K_1'}=3$. Thus, $ \rg(A(K_1))=2$.

\item If $t\equiv 3  \pmod 8$, then  $ \left(\frac{\sqrt{\vep_{2q}},\,d}{ \mathcal T}\right)=1 $ and 
 $ \left(\frac{\sqrt{\vep_{2q}},\,d}{ \mathcal R}\right)=-1$ or $1$ according to whether $q\equiv 7\pmod 8$ or not. Thus, $\sqrt{\vep_{2q}}$ is not norm if and only if $q\equiv 7\pmod 8$ and  we have
 $E_{K_1'}/(E_{K_1'}\cap N_{K_1/K_1'}(K_1))
 =\{\overline{1},\overline{\vep_{2}},\overline{\sqrt{\vep_{q}}},\overline{\sqrt{\vep_{2q}}}, \overline{\sqrt{\vep_{q}\vep_{2q}}},$ $\overline{\vep_{2}\sqrt{\vep_{q}}}, \overline{\vep_{2}\sqrt{\vep_{2q}}}, \overline{\vep_{2}\sqrt{\vep_{q}\vep_{2q}}}\}$ 
 or $\{\overline{1},\overline{\vep_{2}},\overline{\sqrt{\vep_{q}}},   \overline{\vep_{2}\sqrt{\vep_{q}}}     \}$ according to whether $q\equiv 7\pmod 8$ or not.
  Therefore, $e_{K_1/K_1'}=3$ or $2$ according to whether $q\equiv 7\pmod 8$ or not.
  Hence, $ \rg(A(K_1))\in \{2,3\}$ and $  \rg(A(K_1))=2$ if and only if $q\equiv 7\pmod 8$. Similarly, we check that  we have:

\item If  $t\equiv 7  \pmod 8$, then $ \rg(A(K_1))\in \{2,3\}$ and $  \rg(A(K_1))=2$ if and only if $q\equiv 3\pmod 8$. 
\item If  $t\equiv 1 \pmod 8$, then $ \rg(A(K_1))=3$.
	 	 \end{enumerate}
 \item	 Proceeding in the same way, we check the remaining cases. As, under the conditions of the last item,   we have many possibilities for the congruence of $r$, $s$ and $t$ modulo $8$, we limit ourselves to check that $\rg(A(K_1))\in\{2,3,4,5\} $ and  $\rg(A(K_1))=2$ if and only if, after a suitable permutation of $r$, $s$ and $t$, we have $r\equiv 3\pmod 8$, $s\equiv 5\pmod 8$, $t\equiv 7\pmod 8$ and $q\equiv 3\pmod 8$. 
  \end{enumerate}
 \end{proof}

\begin{remark} Let $K$ be a real biquadratic field of the form $A)$ with  $K\not= L$.
	By combining  Lemma   \ref{realTruquad2-rankq}, Lemma \ref{lmthreeprimes}, Lemma \ref{lemmaqA}  and  Lemma \ref{lm fukuda}, we get that $ \rg(A(K_\infty))\leq 2$ and $ \rg(A(K_\infty))=\rg( A(K))$ if and only if $K$ is taking one of the forms in  the items $1)$, $2)$,  $3)$,  $4)$, $5)$, $6)$, $7)$ and   $9)$ of the main theorem. 
\end{remark}

\begin{lemma}\label{lemarst2q}
Let  $q\equiv3\pmod4$ be a prime number and   $d=\delta  rst  \equiv 1 \pmod 4$   be a squarefree integer, with $r$, $s$ and $t$ are prime numbers and $\delta\in\{1,q\}$. Let   $K:= \QQ(\sqrt{2q},\sqrt{d})$   be such that 
$\rg(K)=2$ $($cf. Lemma \ref{lemma2qB}-(3)$)$. Then $\rg(K_1)=2$ if and only if   $d$ is in one of the following forms: 
 \begin{enumerate}[$1)$]
	\item   $d=\delta  rst$, with  $r\equiv 3\pmod 8$, $s\equiv 5\pmod 8$  and we have one of the following conditions:
  	\begin{enumerate}[$a)$]
	
	 	\item  $t\equiv 3 \pmod 8$, $q\equiv 7\pmod 8$ and $-\left(\frac{q}{r}\right)=-\left(\frac{q}{s}\right)=\left(\frac{q}{t}\right)=1$.
	 
	 	 	\item  $t\equiv  5\pmod 8$ and $-\left(\frac{q}{r}\right)=-\left(\frac{q}{s}\right)=\left(\frac{q}{t}\right)=1$.
	 
	 	\item  $t\equiv 7\pmod 8$, $q\equiv 3\pmod 8$ and $\left(\frac{q}{r}\right)=\left(\frac{q}{s}\right)=\left(\frac{q}{t}\right)=-1$.
	
 	\end{enumerate}
Notice that in this case, $d$ satisfies the condition $\C1$ of Lemma \ref{lemma2qB}-$(3)$.

	\item $d=\delta  rst$, with $q$, $r$, $s$, $t$ satisfy one of the following conditions:
 \begin{enumerate}[$a)$]
	
 \item   $r\equiv 3 \pmod 8$, $s\equiv 3\pmod 8$, $t\equiv 5\pmod 8$, $q\equiv 7\pmod 8$ and $-\left(\frac{q}{r}\right)=\left(\frac{q}{s}\right)=\left(\frac{q}{t}\right)=1$,

\item  $r\equiv 3 \pmod 8$,  $s\equiv 7\pmod 8$, $t\equiv 5 \pmod 8$, $q\equiv 3\pmod 8$ and $-\left(\frac{q}{r}\right)=-\left(\frac{q}{s}\right)=\left(\frac{q}{t}\right)=1$,

 \item  $r\equiv 5 \pmod 8$,   $s\equiv 3\pmod 8$, $t\equiv 5\pmod 8$ and $-\left(\frac{q}{r}\right)=\left(\frac{q}{s}\right)=\left(\frac{q}{t}\right)=1$,

 \item  $r\equiv 5 \pmod 8$,   $s\equiv 7\pmod 8$, $t\equiv  3\pmod 8$, $q\equiv  3\pmod 8$ and $-\left(\frac{q}{r}\right)=-\left(\frac{q}{s}\right)=\left(\frac{q}{t}\right)=1$,

\end{enumerate}
Notice that in this case, $d$ satisfies the condition $\C2$ of Lemma \ref{lemma2qB}-$(3)$.

	\item 	$d=\delta  rst$, with    $r\equiv 3 \pmod 8$, $s\equiv 7\pmod 8$, $t\equiv 5 \pmod 8$, $q\equiv  3\pmod 8$ and $-\left(\frac{q}{r}\right)= \left(\frac{q}{s}\right)=-\left(\frac{q}{t}\right)=-1$.

\noindent Notice that in this case, $d$ satisfies the condition $\C3$ of Lemma \ref{lemma2qB}-$(3)$.





	 \end{enumerate} 
\end{lemma}
\begin{proof}We proceed similarly as in the proof of the previous lemma.
 \begin{enumerate}[$1)$]
	\item Assume that $r$, $s$, $t$ and $q$ satisfy the condition  $\C1$ of  Lemma \ref{lemma2qB}-$(3)$, i.e.,   
	 $\left(\frac{2q}{r}\right)=\left(\frac{2q}{s}\right)=-\left(\frac{2q}{t}\right)=1$,   $\left(\frac{-1}{r}\right)=\left(\frac{2}{r}\right)=-1$ and 
	$\left(\frac{-1}{s}\right)\not=\left(\frac{2}{s}\right) $. Notice that last two conditions are equivalent to the fact that $r\equiv 3\pmod 8$ and $s\equiv 5\pmod8$. Thus, $\C1$ is equivalent to the fact that we have one of the following conditions:
		\begin{enumerate}[$a)$]
	 	\item  $t\equiv 3$ or $5\pmod 8$  and $-\left(\frac{q}{r}\right)=-\left(\frac{q}{s}\right)=\left(\frac{q}{t}\right)=1$.
		
		\item  $t\equiv 7$ ot $1\pmod 8$  and $\left(\frac{q}{r}\right)=\left(\frac{q}{s}\right)=\left(\frac{q}{t}\right)=-1$.
	\end{enumerate}
	So according to items $4)$, $5)$ (here, a suitable permutation $r$ and $t$ is taken into account) and $6)$ of Lemma \ref{lmthreeprimes}, $\rg(K_1)=2$ if and only if  $q$, $r$, $s$ and $t$ satisfy one of the conditions in the first item.

	\item Assume that $r$, $s$, $t$ and $q$ satisfy the condition  $\C2$ of  Lemma \ref{lemma2qB}-$(3)$, i.e.,  $\left(\frac{2q}{r}\right)=-\left(\frac{2q}{s}\right)=-\left(\frac{2q}{t}\right)=1$ and  $[  \left(\frac{-1}{r}\right)=\left(\frac{2}{r}\right)=-1 $ or   $ \left(\frac{-1}{r}\right)\not=\left(\frac{2}{r}\right) ]$. Notice that last   condition is equivalent to the fact that $r\equiv 3$, $5$ or $7\pmod 8$.

	\begin{enumerate}[$\bullet$]
	\item Let $r\equiv 3$ or  $5 \pmod 8$. Under these conditions, $\C2$ is equivalent to the fact that $s$ and $t$ satisfy one of the following conditions:
	 \begin{enumerate}[$a)$]
		\item   $s\equiv 3\pmod 8$, $t\equiv 5\pmod 8$ and $-\left(\frac{q}{r}\right)=\left(\frac{q}{s}\right)=\left(\frac{q}{t}\right)=1$,
		
		\item   $s\equiv 1\pmod 8$, $t\equiv 1\pmod 8$ and $-\left(\frac{q}{r}\right)=-\left(\frac{q}{s}\right)=-\left(\frac{q}{t}\right)=1$,
		
		\item   $s\equiv 7\pmod 8$, $t\equiv 1\pmod 8$ and $-\left(\frac{q}{r}\right)=-\left(\frac{q}{s}\right)=-\left(\frac{q}{t}\right)=1$,

		 \item   $s\equiv 7\pmod 8$, $t\equiv 7\pmod 8$ and $-\left(\frac{q}{r}\right)=-\left(\frac{q}{s}\right)=-\left(\frac{q}{t}\right)=1$,
		 
		 \item   $s\equiv 7\pmod 8$, $t\equiv 5$ or $3\pmod 8$ and $-\left(\frac{q}{r}\right)=-\left(\frac{q}{s}\right)=\left(\frac{q}{t}\right)=1$,
		
	 	\item   $s\equiv 1\pmod 8$, $t\equiv 5$ or $3\pmod 8$ and $-\left(\frac{q}{r}\right)=-\left(\frac{q}{s}\right)=\left(\frac{q}{t}\right)=1$,
		
		 \item   $s\equiv t\equiv 3 \pmod 8$ and $-\left(\frac{q}{r}\right)=\left(\frac{q}{s}\right)=\left(\frac{q}{t}\right)=1$,
		 
		 \item   $s\equiv t\equiv 5 \pmod 8$ and $-\left(\frac{q}{r}\right)=\left(\frac{q}{s}\right)=\left(\frac{q}{t}\right)=1$.

	\end{enumerate}
Assume that $r\equiv 3\pmod 8$. Using Lemma   \ref{lmthreeprimes}, we check that if $s$ and $t$ satisfy one of the conditions $a)$, ..., $c)$, 
then  $\rg(K_1)=2$ if and only if $s$, $t$ and $q$ satisfy one of the conditions $a)$ and $b)$ in the second item of the lemma. Moreover, proceeding as in the proof of Lemma   \ref{lmthreeprimes},  we check that if  $s$ and $t$ satisfy one of the conditions $g)$ and $h)$, then $\rg(A(K))	\geq 3$.

Now assume that $r\equiv 5\pmod 8$. Proceeding similarly as in the previous case, we get  $\rg(K_1)=2$ if and only if $s$, $t$ and $q$ satisfy one of the conditions $a)$ and $b)$ in the third item.

	\item	Now let $r\equiv 7  \pmod 8$. Under this condition, $\C2$ is equivalent to the fact that $s$ and $t$ satisfy one of the following conditions:
	\begin{enumerate}[$a)$]
		\item   $s\equiv 3\pmod 8$, $t\equiv 5\pmod 8$ and $\left(\frac{q}{r}\right)=\left(\frac{q}{s}\right)=\left(\frac{q}{t}\right)=1$,
		
		\item   $s\equiv 1\pmod 8$, $t\equiv 1\pmod 8$ and $\left(\frac{q}{r}\right)=-\left(\frac{q}{s}\right)=-\left(\frac{q}{t}\right)=1$,
		
		\item   $s\equiv 7\pmod 8$, $t\equiv 1\pmod 8$ and $\left(\frac{q}{r}\right)=-\left(\frac{q}{s}\right)=-\left(\frac{q}{t}\right)=1$,

		\item   $s\equiv 7\pmod 8$, $t\equiv 7\pmod 8$ and $\left(\frac{q}{r}\right)=-\left(\frac{q}{s}\right)=-\left(\frac{q}{t}\right)=1$,
		
		\item   $s\equiv 7\pmod 8$, $t\equiv 5$ or $3\pmod 8$ and $\left(\frac{q}{r}\right)=-\left(\frac{q}{s}\right)=\left(\frac{q}{t}\right)=1$,
		
		\item   $s\equiv 1\pmod 8$, $t\equiv 5$ or $3\pmod 8$ and $\left(\frac{q}{r}\right)=-\left(\frac{q}{s}\right)=\left(\frac{q}{t}\right)=1$,
		
		\item   $s\equiv t\equiv 3 \pmod 8$ and $\left(\frac{q}{r}\right)=\left(\frac{q}{s}\right)=\left(\frac{q}{t}\right)=1$,
		
		\item   $s\equiv t\equiv 5 \pmod 8$ and $\left(\frac{q}{r}\right)=\left(\frac{q}{s}\right)=\left(\frac{q}{t}\right)=1$.
	 	\end{enumerate}
We check that under each of these conditions, we have $\rg(A(K))	\geq 3$.
\end{enumerate}
	
\noindent Hence, $d$ satisfies the condition $\C2$ of  Lemma \ref{lemma2qB}-$(3)$ and 	 $\rg(A(K))=2$ if and only if, after a suitable permutation of $r$, $s$, $t$, we have one of the conditions $a)$,..., $d)$ in the second item.	
	\item  Now assume that $r$, $s$, $t$ and $q$ satisfy the condition  $\C3$ of  Lemma \ref{lemma2qB}-$(3)$, i.e., 
	$\left(\frac{2q}{r}\right)=\left(\frac{2q}{s}\right)=\left(\frac{2q}{t}\right)=-1$.
	
		\begin{enumerate}[$\bullet$]
		\item Let $r\equiv 3  \pmod 8$. Under this condition, $\C3$ is equivalent to the fact that $s$ and $t$ satisfy one of the following conditions:
		\begin{enumerate}[$a)$]
			\item   $s\equiv 3\pmod 8$, $t\equiv 5\pmod 8$ and $-\left(\frac{q}{r}\right)=-\left(\frac{q}{s}\right)=-\left(\frac{q}{t}\right)=-1$,
			
			\item   $s\equiv 1\pmod 8$, $t\equiv 1\pmod 8$ and $-\left(\frac{q}{r}\right)= \left(\frac{q}{s}\right)= \left(\frac{q}{t}\right)=-1$,
			
			\item   $s\equiv 7\pmod 8$, $t\equiv 1\pmod 8$ and $-\left(\frac{q}{r}\right)= \left(\frac{q}{s}\right)= \left(\frac{q}{t}\right)=-1$,

			\item   $s\equiv 7\pmod 8$, $t\equiv 7\pmod 8$ and $-\left(\frac{q}{r}\right)= \left(\frac{q}{s}\right)= \left(\frac{q}{t}\right)=-1$,
			
			\item   $s\equiv 7\pmod 8$, $t\equiv 5$ or $3\pmod 8$ and $-\left(\frac{q}{r}\right)= \left(\frac{q}{s}\right)=-\left(\frac{q}{t}\right)=-1$,
			
			\item   $s\equiv 1\pmod 8$, $t\equiv 5$ or $3\pmod 8$ and $-\left(\frac{q}{r}\right)= \left(\frac{q}{s}\right)=-\left(\frac{q}{t}\right)=-1$,
			
			\item   $s\equiv t\equiv 3 \pmod 8$ and $-\left(\frac{q}{r}\right)=-\left(\frac{q}{s}\right)=-\left(\frac{q}{t}\right)=-1$,
			
			\item   $s\equiv t\equiv 5 \pmod 8$ and $-\left(\frac{q}{r}\right)=-\left(\frac{q}{s}\right)=-\left(\frac{q}{t}\right)=-1$.

		\end{enumerate}
	We check that under each of these conditions, we have $\rg(A(K))	\geq 2$ and $\rg(A(K))=2$ if and only if $s\equiv 7\pmod 8$, $t\equiv 5 \pmod 8$, $q\equiv  3\pmod 8$ and $-\left(\frac{q}{r}\right)= \left(\frac{q}{s}\right)=-\left(\frac{q}{t}\right)=-1$.

	\item  Let $r\equiv 5  \pmod 8$. Under this condition, $\C3$ is equivalent to the fact that $s$ and $t$ satisfy one of the following conditions:
	\begin{enumerate}[$a)$]
		\item   $s\equiv 3\pmod 8$, $t\equiv 5\pmod 8$ and $-\left(\frac{q}{r}\right)=-\left(\frac{q}{s}\right)=-\left(\frac{q}{t}\right)=-1$,
		
		\item   $s\equiv 1\pmod 8$, $t\equiv 1\pmod 8$ and $-\left(\frac{q}{r}\right)= \left(\frac{q}{s}\right)= \left(\frac{q}{t}\right)=-1$,
		
		\item   $s\equiv 7\pmod 8$, $t\equiv 1\pmod 8$ and $-\left(\frac{q}{r}\right)= \left(\frac{q}{s}\right)= \left(\frac{q}{t}\right)=-1$,

		\item   $s\equiv 7\pmod 8$, $t\equiv 7\pmod 8$ and $-\left(\frac{q}{r}\right)= \left(\frac{q}{s}\right)= \left(\frac{q}{t}\right)=-1$,
		
		\item   $s\equiv 7\pmod 8$, $t\equiv 5$ or $3\pmod 8$ and $-\left(\frac{q}{r}\right)= \left(\frac{q}{s}\right)=-\left(\frac{q}{t}\right)=-1$,
		
		\item   $s\equiv 1\pmod 8$, $t\equiv 5$ or $3\pmod 8$ and $-\left(\frac{q}{r}\right)= \left(\frac{q}{s}\right)=-\left(\frac{q}{t}\right)=-1$,
		
		\item   $s\equiv t\equiv 3 \pmod 8$ and $-\left(\frac{q}{r}\right)=-\left(\frac{q}{s}\right)=-\left(\frac{q}{t}\right)=-1$,
		
		\item   $s\equiv t\equiv 5 \pmod 8$ and $-\left(\frac{q}{r}\right)=-\left(\frac{q}{s}\right)=-\left(\frac{q}{t}\right)=-1$.

	\end{enumerate}
	We check that under each of these conditions, we have $\rg(A(K))	\geq 2$ and $\rg(A(K))=2$ if and only if    $s\equiv 7\pmod 8$, $t\equiv 3\pmod 8$, $q\equiv 3\pmod 8$ and $-\left(\frac{q}{r}\right)= \left(\frac{q}{s}\right)=-\left(\frac{q}{t}\right)=-1$.
	
		\item Moreover, we check that  $\rg(A(K))	\geq 3$, for the remaining cases, i.e., 
	\begin{enumerate}[$a)$]
		
		\item $r\equiv 7\pmod 8$,   $s\equiv 1\pmod 8$, $t\equiv 1\pmod 8$ and $\left(\frac{q}{r}\right)= \left(\frac{q}{s}\right)= \left(\frac{q}{t}\right)=-1$,
		
		\item   $r\equiv 7\pmod 8$,   $s\equiv 7\pmod 8$, $t\equiv 1\pmod 8$ and $\left(\frac{q}{r}\right)= \left(\frac{q}{s}\right)= \left(\frac{q}{t}\right)=-1$,

		\item   $r\equiv 7\pmod 8$,   $s\equiv 7\pmod 8$, $t\equiv 7\pmod 8$ and $\left(\frac{q}{r}\right)= \left(\frac{q}{s}\right)= \left(\frac{q}{t}\right)=-1$,
		
		\item   $r\equiv 7\pmod 8$,   $s\equiv 7\pmod 8$, $t\equiv 5$ or $3\pmod 8$ and $\left(\frac{q}{r}\right)= \left(\frac{q}{s}\right)=-\left(\frac{q}{t}\right)=-1$,
		
	  		\item   $r\equiv 1\pmod 8$,   $s\equiv 1\pmod 8$, $t\equiv 1\pmod 8$ and $\left(\frac{q}{r}\right)= \left(\frac{q}{s}\right)=\left(\frac{q}{t}\right)=-1$.

	\end{enumerate}

\end{enumerate}

Hence, $d$ satisfies the condition $\C3$ of  Lemma \ref{lemma2qB}-$(3)$ and 	 $\rg(A(K))=2$ if and only if, after a suitable permutation of $r$, $s$, $t$, we have   the conditions in the third item.

\end{enumerate}

\end{proof}

}

 \begin{remark} Let $K$ be a real biquadratic field   of the form $B)$, then by  combining  Lemma  \ref{realTruquad2-rankq}, Lemma \ref{lemarst2q}, Lemma \ref{lemma2qB} and Lemma \ref{lm fukuda}, we check  that
 	 $ \rg(A(K_\infty))\leq 2$ and $ \rg(A(K_\infty))=\rg( A(K))$ if and only if $K$ is taking one of the forms in  the items  $10)$,  $11)$,  $12)$, $13)$, $14)$, $15)$, $16)$, ..., $21)$, $22)$ of the main theorem.
\end{remark}

 \begin{remark} \label{remq=r=7mod8}  
 Let	$q\equiv r\equiv  7\pmod 8$	 be two different prime numbers. We can assume that  $\left(\frac{r}{q}\right)=1$. Proceeding as in the proof of \cite[Lemmas 5 and 7]{CAZbol}, we check the following.
 \begin{enumerate}[\rm $1)$]
  	\item  Let    $a$ and $b$ be two integers such that
 	$ \varepsilon_{rq}=a+b\sqrt{rq}$. 
  	Then  $2r(a+1)$ is a square in $\NN$ and  
 	   $$\sqrt{\varepsilon_{rq}}=b_1\sqrt{p}+b_2\sqrt{q}  \text{ and } 1= rb_1^2-qb_2^2,$$ 
  
 	where $b_1$ and $b_2$ are two integers.
  
 \item Let $x$ and $y$ be two integers  such that
 $ \varepsilon_{2rq}=x+y\sqrt{2rq}$. Then
 out of the three numbers   
 $(x+1)$, $r(x+1)$ or $2r(x+1)$, exactly one of them is   a square in $\NN $. Moreover, we have:
 \begin{enumerate}[\rm $\bullet$]
 	 	\item If $(x+1)$ is a square in $\NN$, then $\sqrt{2\varepsilon_{2rq}}=y_1+y_2\sqrt{2rq}$ and 	$2=  y_1^2-2rqy_2^2$,
 	\item If $r(x+1)$ is a square in $\NN$, then $\sqrt{2\varepsilon_{2rq}}=y_1\sqrt{r}+y_2\sqrt{2q}$ and $2= ry_1^2-2qy_2^2$,
 	\item If $2r(x+1)$ is a square in $\NN$, then $\sqrt{2\varepsilon_{2rq}}=y_1\sqrt{2r}+y_2\sqrt{q}$ and $2=  2ry_1^2-qy_2^2$, 
 \end{enumerate}
 where $y_1$ and $y_2$ are two integers. 
 Notice that the fact that exactly one of the values $(x+1)$, $r(x+1)$ and $2r(x+1)$  is   a square in $\NN $  comes from  
the equality  $x^2-1=(x+1)(x-1)=2rqy^2$ and the unique factorization in $\mathbb{Z}$ (other possibilities are eliminated by means of Legendre symbol as in the	 proof of \cite[Lemmas 5 and 7]{CAZbol}.
 
  \end{enumerate}
 \end{remark}

\bigskip

 Let us recall that $L$ is named to be   any   real biquadratic field   of the form  
 \begin{center}
 	{$ \displaystyle L:=\QQ(\sqrt{\delta_0},\sqrt{r})\text{ or }\QQ(\sqrt{\delta_0},\sqrt{rs}) \text{ where }   r\equiv 1\pmod8   \text{ and }  \left(\frac{\delta_0}{r}\right)=1,$}
 \end{center}
 and $\delta_0\in\{q, q_1q_2\}	$. We have:

 \begin{lemma}\label{realTruquad2-rankq1q1} Let $K= \QQ(\sqrt{q_1q_2},\sqrt{d})\not=L$ and  $K_1= \QQ(\sqrt{q_1q_2},\sqrt{d}, \sqrt{2})$  where $q_1\equiv3 \pmod 4$,  $q_2\equiv3 \pmod 8$ are two prime numbers and $d \equiv 1\pmod  4$  is a positive square free integer that is not divisible by  $q_1q_2$.
 	
 	\noindent $\mathbf{1)}$	Assume that $K_1= \QQ(\sqrt{q_1q_2},\sqrt{d}, \sqrt{2})$,  where   $d=\delta r\equiv 1\pmod 4 $ with $r$ is a prime  number  and    $\delta\in \{1,q_1,q_2\}$.  
 	We have  $$\rg(A(K_1))\in\{0,1\}.$$ More precisely:
 	$	\rg(A(K_1)) =0$ if and only if  we have      one of the following conditions :  
 	\begin{enumerate}[\indent$\bullet$]
 		\item $r\equiv 3\pmod 8$,
 		 \item $r\equiv 5\pmod 8$, $q_1\equiv 7\pmod 8$ and  $\left(\frac{q_1}{ r}\right)=-1$,
 	 \item $r\equiv 5\pmod 8$, $q_1\equiv 3\pmod 8$     and  $\left(\frac{q_1q_2}{ r}\right)=-1$,

 		\item  $r\equiv  7\pmod 8$ and   $q_1\equiv 3\pmod 8$. 
 		
 	\end{enumerate}

 	\noindent $\mathbf{2)}$ 	Assume that $K_1= \QQ(\sqrt{q_1q_2},\sqrt{d}, \sqrt{2})$ with $d=\delta rs\equiv 1\pmod 4$, where $r$ and $s$ are two prime numbers and    $\delta\in \{1,q_1,q_2\}$.
 	We have:  $$\rg(A(K_1))\in\{1,2,3,4\}.$$
 	More precisely,  we have:

 	\noindent$\rg(A(K_1))=1$ if and only if, after a suitable permutation of $r$ and $s$, we have one of the following conditions:
 	
 	\begin{enumerate}[$\C1:$]
 		\item $ \left(\frac{q_1q_2}{s}\right)=\left(\frac{q_1q_2}{r}\right)=-1$ and we have one of the following congruence conditions:
 		\begin{enumerate}[$\bullet$]
 			\item  $r\equiv 5\pmod 8$ and $s\equiv 3\pmod 8$ with $q_1\equiv 3\pmod 8$ or $[q_1\equiv 7\pmod 8$ and $\left(\frac{q_1}{r}\right)=-1]$,
 			
 			\item $r\equiv 3\pmod 8$ and $s\equiv 3\pmod 8$ with  $q_1\equiv 7\pmod 8$ and $\left(\frac{q_1}{r}\right)\not=\left(\frac{q_1}{s}\right)$, 
 			
 			\item $r\equiv 3\pmod 8$ and $s\equiv 7\pmod 8$ with  $q_1\equiv 3\pmod 8$.

 		\end{enumerate}

 		\item	$ \left(\frac{q_1q_2}{r}\right)=-\left(\frac{q_1q_2}{s}\right)=-1$ and we have one of the following congruence conditions:
 		\begin{enumerate}[$\bullet$]
 			\item  $ r\equiv         s\equiv 3\pmod 8$ with $q_1\equiv 7\pmod 8$ and $\left(\frac{q_1}{r}\right) \not=\left(\frac{q_1}{s}\right)$,
 			\item $ r\equiv     3\pmod 8$ and  $    s\equiv 5\pmod 8$  with $q_1\equiv 7\pmod 8$ and $\left(\frac{q_1}{s}\right) =-1$,  
 			
 			\item   $ r\equiv     5\pmod 8$ and  $    s\equiv 3\pmod 8$ with   $q_1\equiv 3\pmod 8$ or    $[q_1\equiv 7\pmod 8$ and $\left(\frac{q_1}{r}\right) =-1]$.
 		\end{enumerate}

 		\item	$ \left(\frac{q_1q_2}{r}\right)=\left(\frac{q_1q_2}{s}\right)=1$ and we have one of the following congruence conditions:
 		\begin{enumerate}[$\bullet$]
 			\item  $r\equiv 3\pmod 8$ and $s\equiv 3\pmod 8$  with       $ q_1\equiv 7\pmod 8$ and $\left(\frac{q_1}{r}\right) \not=\left(\frac{q_1}{s}\right) $,
 			
 			\item   $r\equiv 5\pmod 8$ and $s\equiv 3\pmod 8$  with       $ q_1\equiv 7\pmod 8$ and $\left(\frac{q_1}{r}\right) =-1$.
 		\end{enumerate}

 	\end{enumerate}

 	\noindent $\rg(A(K_1))=2$ if and only if, after a suitable permutation of $r$ and $s$, we have one of the following conditions:
 	\begin{enumerate}[$\C1:$]
 		\item $ \left(\frac{q_1q_2}{s}\right)=\left(\frac{q_1q_2}{r}\right)=-1$ and we have one of the following congruence conditions:
 		\begin{enumerate}[$\bullet$]
 			\item    $r\equiv 5\pmod 8$ $s\equiv 3\pmod 8$ with      $ q_1\equiv 7\pmod 8$ and $\left(\frac{q_1}{r}\right) =1$.

 			\item  $r\equiv 5\pmod 8$ and $s\equiv 5\pmod 8$  with    $q_1\equiv 3\pmod 8$ or $[q_1\equiv 7\pmod 8$ and $\left(\frac{q_1}{r}\right) =-1$ or $\left(\frac{q_1}{s}\right) =-1]$.
 			
 			\item $r\equiv 3\pmod 8$ and $s\equiv 3\pmod 8$ with    $q_1\equiv 3\pmod 8$ or  $[q_1\equiv 7\pmod 8$ and $\left(\frac{q_1}{r}\right) =\left(\frac{q_1}{s}\right)  ]$.

 			\item  $r\equiv 3\pmod 8$ and $s\equiv 7\pmod 8$ with    $ q_1\equiv 7\pmod 8$.
 			
 			\item  $r\equiv 5\pmod 8$ and $s\equiv 7\pmod 8$ with  $q_1\equiv 3\pmod 8$ or $[q_1\equiv 7\pmod 8$ and $\left(\frac{q_1}{r}\right) =-1]$.

 			\item $r\equiv 5\pmod 8$ and $s\equiv 1\pmod 8$ with   $q_1\equiv 3\pmod 8$ or $[q_1\equiv 7\pmod 8$ and $\left(\frac{q_1}{r}\right) =-1]$.

 			\item $r\equiv 3\pmod 8$ and $s\equiv 1\pmod 8$,
 			
 			\item $r\equiv 7\pmod 8$ and $s\equiv 7$ or $1\pmod 8$ with    $q_1\equiv 3\pmod 8$.
 			
 		\end{enumerate}

 		\item	$ \left(\frac{q_1q_2}{r}\right)=-\left(\frac{q_1q_2}{s}\right)=-1$ and we have one of the following congruence conditions:
 		\begin{enumerate}[$\bullet$]
 			\item   $    r\equiv 7$ or $1\pmod 8$ and $s\equiv     7\pmod 8$ with  $q_1\equiv 3\pmod 8$.
 			
 			\item  $r\equiv   s\equiv 5\pmod 8$ with   $q_1\equiv 3\pmod 8$ or   $[q_1\equiv 7\pmod 8$ and $\left(\frac{q_1}{r}\right) =-1$ or $ \left(\frac{q_1}{s}\right)=-1  ]$.
 			
 			\item $r\equiv   s\equiv 3\pmod 8$ with  $q_1\equiv 3\pmod 8$ or  $[q_1\equiv 7\pmod 8$ and  $\left(\frac{q_1}{r}\right) =\left(\frac{q_1}{s}\right)]$.
 			
 			\item $r\equiv     3\pmod 8$ and  $   s\equiv 5\pmod 8$ with $q_1\equiv 3\pmod 8$ or  $[q_1\equiv 7\pmod 8$ and    $\left(\frac{q_1}{s}\right)=1  ]$.

 			\item $ r\equiv     5\pmod 8$ and  $    s\equiv 3\pmod 8$ with     $ q_1\equiv 7\pmod 8$ and $\left(\frac{q_1}{r}\right) =1$,

 				\item       $ r\equiv     1\pmod 8$,  $    s\equiv 5\pmod 8$        $q_1\equiv 7\pmod 8$ and $\left(\frac{q_1}{s}\right) =-1$,
 			
 			\item      $ r\equiv     1\pmod 8$ and  $    s\equiv 3\pmod 8$.

 		\end{enumerate}

 		\item	$ \left(\frac{q_1q_2}{r}\right)=\left(\frac{q_1q_2}{s}\right)=1$ and we have one of the following congruence conditions:
 		\begin{enumerate}[$\bullet$]
 			\item   $r\equiv 7\pmod 8$ and $s\equiv 5 \pmod 8$ with   $q_1\equiv 3\pmod 8$ or    $[q_1\equiv 7\pmod 8$ and $\left(\frac{q_1}{s}\right) =-1]$.

 			\item $r\equiv 7\pmod 8$ and $s\equiv 3 \pmod 8$ with   $q_1\equiv 3\pmod 8$ or    $ [q_1\equiv 7\pmod 8$ and $\left(\frac{q_1}{r}\right) \not=\left(\frac{q_1}{s}\right)]$,
 			
 			\item $r\equiv 3\pmod 8$ and $s\equiv 3\pmod 8$ with   $q_1\equiv 3\pmod 8$ or     $[ q_1\equiv 7\pmod 8$ and $\left(\frac{q_1}{r}\right) =\left(\frac{q_1}{s}\right)] $,

 			\item $r\equiv 5\pmod 8$ and $s\equiv 5\pmod 8$ with       $ q_1\equiv 7\pmod 8$ and $[\left(\frac{q_1}{r}\right)=-1$ or $  \left(\frac{q_1}{s}\right)=-1 ]$,
 			
 			\item $r\equiv 5\pmod 8$ and $s\equiv 3\pmod 8$ with    $q_1\equiv 3\pmod 8$ or   $[ q_1\equiv 7\pmod 8$ and $\left(\frac{q_1}{r}\right)=1]$.
 		\end{enumerate}

 	\end{enumerate}

 		\noindent $\mathbf{3)}$ 	Assume that $K_1= \QQ(\sqrt{q_1q_2},\sqrt{d}, \sqrt{2})$ with $d=\delta rst\equiv 1\pmod 4$,  where  $r$, $s$ and $t$ are three prime numbers such that
 		  $ \left(\frac{q_1q_2}{r}\right)=\left(\frac{q_1q_2}{s}\right)=\left(\frac{q_1q_2}{t}\right)=-1$ or $ \left(\frac{q_1q_2}{r}\right)=-\left(\frac{q_1q_2}{s}\right)=\left(\frac{q_1q_2}{t}\right)=-1$, and    $\delta\in \{1,q_1,q_2\}$. Then, we have: 
 		 $$\rg(A(K_1))\in\{3,4,5\}.$$

 \end{lemma}
 \begin{proof} 	As the class number of $K_1'=\QQ(\sqrt{q_1q_2},\sqrt{2}) $ is odd,  we get $\rg(A(K_1))=t_{K_1/K_1'}-1-e_{K_1/K_1'}$.
Let us define the following parameter:
 $$
 \gamma=
 \begin{cases}
 	0,\text{ if }\left(\frac{q_2}{q_1}\right)=-1  \text{ and }   q_1\equiv 7\pmod 8 \\
 	1, \text{ else}.
 \end{cases}
 $$
 
	We note that if 
 	$q_1\equiv3 \pmod 8$, then $E_{K_1'}=\langle-1, \vep_{2},\vep_{q_1q_2}, \sqrt{\vep_{2q_1q_2}}  \rangle $ and $\sqrt{\vep_{2q_1q_2}} = \frac{\sqrt{2}}{2}(x+y\sqrt{{2q_1q_2}})  $ where $x$ and $y$  are two integers such that $2=-x^2+2q_1q_2y^2$ (cf. \cite[Lemma 4]{AzchemsTrends}).   
 	
 	We also note  that if  $q_1\equiv7 \pmod 8$, $E_{K_1'}=\langle-1, \vep_{2},\vep_{q_1q_2}, \sqrt{\vep_{2q_1q_2}\vep_{q_1q_2}}  \rangle $ and $\sqrt{\vep_{2q_1q_2}} = \frac{\sqrt{2}}{2}(a\sqrt{2q_1}+b\sqrt{q_2})  $ (resp. $\sqrt{\vep_{q_1q_2}} = c\sqrt{q_1}+d\sqrt{q_2})$ where $a$ and $b$ (resp. $c$ and $d$)  are two integers such that $(-1)^{\gamma}2=2q_1a^2-b^2q_2$ (resp. $(-1)^{\gamma}=q_1c^2-d^2q_2$)\label{sqrtepsiq1q2} 
 	(cf. \cite[Lemmas 5 and 7]{CAZbol}).

 	 Thus, for $q_1\equiv3 \pmod 4$, we have $E_{K_1'}=\langle-1, \vep_{2},\vep_{q_1q_2}, \eta  \rangle $, where $\eta$ is defined as follows
 	  $$
 	 \eta =
 	 \begin{cases}
 	 	\sqrt{\vep_{2q_1q_2}\vep_{q_1q_2}},\text{ if }   q_1\equiv 7\pmod 8 \\
 	 	\sqrt{ \vep_{2q_1q_2}}, \text{ else}.
 	 \end{cases}
 	 $$
 	We check that  $N_{K_1'/\QQ(\sqrt{2q_1q_2})}(\eta)=(-1)^\gamma \vep_{ 2q_1q_2}$.
 	Let    $r$ be an odd prime number and    $\mathcal R$ be a prime ideal of $K_1'$   lying above $r$.
 	\begin{enumerate}[$1)$]
 		\item By Remark \ref{remkresiduesymbolsVYfroms}, if $r$ is not totally decomposed in $K_1'$, then  $\left(\dfrac{-1,\,\delta r}{ \mathcal R}\right)=1 $,   else we have $\left(\dfrac{-1,\,\delta r}{ \mathcal R}\right)=\left(\dfrac{-1 }{ r}\right) $. Furthermore, we have:	
 		\begin{enumerate}[$\bigstar$]
 			\item Assume that $r\equiv 3$ or $5\pmod 8$. Then the decomposition of $r$ in $K_1'$ takes the $\Y$ form for $\QQ(\sqrt{2})$ and so
 			\begin{eqnarray}
 				\left(\frac{\vep_2,\,\delta r}{ \mathcal R}\right)&=&\left(\frac{\vep_2,\,\delta  }{ \mathcal R}\right)\left(\frac{\vep_2,\, r}{ \mathcal R}\right)=\left(\frac{-1}{  r}\right)=(-1)^{\frac{r-1}{2}} .   
 			\end{eqnarray}
 			For the other symbols we distinguish the following two subcases:\\
 			\noindent$\bullet$ If $ \left(\frac{q_1q_2}{r}\right)=-1$, then the decomposition of $r$ in $K_1'$ takes the $\V$ form for $\QQ(\sqrt{2q_1q_2})$. Note that 
 			$N_{K_1'/\QQ(\sqrt{2q_1q_2})}({\vep_{q_1q_2}})= 1\in \QQ$ and  $N_{K_1'/\QQ(\sqrt{2q_1q_2})}(\eta)=(-1)^\gamma \vep_{ 2q_1q_2}$. Then, by   Remark \ref{remkresiduesymbolsVYfroms}, we have:
 			\begin{eqnarray}
 				\left(\frac{\vep_{q_1q_2},\, \delta r}{ \mathcal R}\right)&=& 	\left(\frac{\vep_{q_1q_2},\,   r}{ \mathcal R}\right)=1   ,\\
 				& &\nonumber\\
 				\left(\frac{\eta,\,\delta r}{ \mathcal R}\right)&=&\left(\frac{\eta,\,  r}{ \mathcal R}\right)=\left(\frac{(-1)^\gamma \vep_{ 2q_1q_2},\,r}{ \mathcal R_{\QQ(\sqrt{2q_1q_2})}}\right)=\left(\frac{(-1)^\gamma}{ r}\right)\left(\frac{\omega}{ r}\right)  .
 			\end{eqnarray}
 			where 
 			$$
 			\omega =
 			\begin{cases}
 				{q_1}           ,\text{ if }   q_1\equiv 7\pmod 8 \\
 				2, \text{ else}.
 			\end{cases}
 			$$
 				The last equality follows from the fact that $\vep_{2q_1q_2}=\frac{1}{\omega}  u^2$ for some $u  \in \QQ(\sqrt{2q_1q_2})$.

 			\noindent $\bullet$ Now if  $ \left(\frac{q_1q_2}{r}\right)=1$,	then the decomposition of $r$ in $K_1'$ takes the $\V$ form for $\QQ(\sqrt{q_1q_2})$. Note that 
 			$N_{K_1'/\QQ(\sqrt{q_1q_2})}(\vep_{q_1q_2})=  \vep_{q_1q_2}^2 $ and we check that 
 			$N_{K_1'/\QQ(\sqrt{q_1q_2})}({\eta})=  (-1)^{\gamma}\vep_{q_1q_2} $ or $1$ according to whether $q_1\equiv 7\pmod 8$  or not. Thus, by Remark \ref{remkresiduesymbolsVYfroms}, we have
 			\begin{eqnarray}
 				\left(\frac{\vep_{q_1q_2},\, \delta r}{ \mathcal R}\right)&=&  1   ,\\
 				& &\nonumber\\
 				\left(\frac{\eta,\,\delta r}{ \mathcal R}\right)&=& \left(\frac{(-1)^{\gamma}\vep_{q_1q_2}\text{ or }1,\,  r}{ \mathcal R_{\QQ(\sqrt{q_1q_2})}}\right)=\left(\frac{(-1)^\gamma}{ r}\right)\left(\frac{q_1}{ r}\right) \text{ or } 1,  
 			\end{eqnarray}
 			according to whether $q_1\equiv 7\pmod 8$  or not. The last equality follows from the fact that $\vep_{q_1q_2}=\frac{1}{q_1}v^2$ for some $v\in \QQ(\sqrt{q_1q_2})$.

 		\quad	Notice that $\rg(A(K_1))=2-1-e_{K_1/K_1'}=1-e_{K_1/K_1'}\in\{0,1\}$. By the above values of norm residue symbols we have, 
 				$e_{K_1/K_1'}=0$ if and only if $r\equiv 5\pmod 8$ and [$(q_1\equiv 3\pmod 8$ and $\left(\frac{q_1q_2}{ r}\right)=1)$ or $(q_1\equiv 7\pmod 8$ and $\left(\frac{q_1}{ r}\right)=1)]$.
 				 Therefore,
 		 	$\rg(A(K_1))=1$ if and only if $r\equiv 5\pmod 8$ and [$(q_1\equiv 3\pmod 8$ and $\left(\frac{q_1q_2}{ r}\right)=1)$ or $(q_1\equiv 7\pmod 8$ and $\left(\frac{q_1}{ r}\right)=1)]$.

 			\item Assume that $r\equiv 7$ or $1\pmod 8$ with $ \left(\frac{q_1q_2}{r}\right)=-1$. Then the decomposition of $r$ in $K_1'$ takes the $\V$ form for $\QQ(\sqrt{2})$. As $N_{K_1'/\QQ(\sqrt{2})}(\vep_{q_1q_2})= 1$ and 
 			$N_{K_1'/\QQ(\sqrt{2})}({\eta})= 1$ or $-1$ according to whether $q_1\equiv 7\pmod 8$ or not. Thus,  by Remark \ref{remkresiduesymbolsVYfroms}, we have: 
 			\begin{eqnarray}
 				\left(\frac{\vep_2,\,\delta r}{ \mathcal R}\right) = 1=	\left(\frac{\vep_{q_1q_2},\, \delta r}{ \mathcal R}\right),
 			\end{eqnarray}
 			and 
 			\begin{eqnarray}
 				\left(\frac{\eta,\,\delta r}{ \mathcal R}\right)&=&1 \text{ or }\left(\frac{-1}{ r}\right),
 			\end{eqnarray}
 			according to whether $q_1\equiv 7\pmod 8$ or not.
 			Therefore, in this subcase, 
 		 $\rg(A(K_1))=2-1-e_{K_1/K_1'}=1-e_{K_1/K_1'}\in\{0,1\}$, and $e_{K_1/K_1'}=0$
 			   if and only if $r\equiv 1\pmod 8$ or $(r\equiv q_1\equiv 7\pmod 8 )$. Thus, 
 			  $\rg(A(K_1))=1$ if and only if $r\equiv 1\pmod 8$ or $(r\equiv q_1\equiv 7\pmod 8 )$.

 		\item	Assume that $r\equiv 7\pmod 8$ with $ \left(\frac{q_1q_2}{r}\right)=1$. Then $r$ is totally decomposed in $K$. 
 	 	Therefore, in this subcase, we have  $\rg(A(K_1))=3-e_{K_1/K_1'}$.
 				  Inspired by \cite[p. 160]{lemmermeyer2013reciprocity}, we check easily that there are 	two prime ideals      $\mathcal R_{\QQ(\sqrt{2})}$  and  $\mathcal R'_{\QQ(\sqrt{2})}$ of $\QQ(\sqrt{2})$ above $r$ such that : 
 				  $$\left(\frac{\vep_2,\,\delta r}{ \mathcal R_{\QQ(\sqrt{2})}}\right) = 1  \text{ and }  \left(\frac{\vep_2,\,\delta r}{ \mathcal R'_{\QQ(\sqrt{2})}}\right) = -1.$$
 			 	Let $\mathcal R$ (resp. $\mathcal R'$) be an ideal  of $K_1'$ above the ideal    $\mathcal R_{\QQ(\sqrt{2})}$ (resp. $\mathcal R'_{\QQ(\sqrt{2})}$). Thus,  by Remark \ref{remkresiduesymbolsVYfroms}, we have:   
 				\begin{eqnarray}
 					\left(\dfrac{-1,\,\delta r}{ \mathcal R}\right)=\left(\dfrac{-1,\,\delta r}{ \mathcal R'}\right)=-1, \quad	 \left(\frac{\vep_2,\,\delta r}{ \mathcal R}\right) = 1  \text{ and }  \left(\frac{\vep_2,\,\delta r}{ \mathcal R'}\right) = -1.
 				\end{eqnarray}
 				Furthermore, we have 	$\left(\dfrac{\varepsilon_{q_1q_2},\,\delta r}{ \mathcal R}\right)=\left(\dfrac{\varepsilon_{q_1q_2},\,\delta r}{ \mathcal R'}\right)=\left(\frac{q_1}{ r}\right)$,
 				and notice that if $\varepsilon_{q_1q_2}$ is not norm, then $\overline{\varepsilon_{q_1q_2}}=\overline{-1}$ in the quotient  $E_{K_1'}/(E_{K_1'}\cap N_{K_1/K_1'}(K_1))$.
 				Thus   $\rg(A(K_1))=3-e_{K_1/K_1'}\leq 1$. Put $M=\QQ(\sqrt{2q_1r},\sqrt{q_1q_2})$. Remark that $K_1/M$ is an unramified quadratic extension and $$h_2(M)=\frac14q(M)h_2( q_1q_2)h_2(2q_1r)h_2(2q_2r)=\frac12q(M) h_2(2q_1r) .$$
 				
 			Note that $q(M)=2$ (we check this by using Remark \ref{remq=r=7mod8} for the case $q_1\equiv r\equiv7\pmod 8$, and the expressions of square roots of units given in   Page \pageref{sqrtepsiq1q2} for the other cases), and that $h_2(2q_1r)=2$ if and only if $q_1\equiv 3\pmod 8$ (cf. \cite[Corollaries 19.7 and 18.4]{connor88}). So by class field theory $\rg(A(K_1))=1$ if and only if  $q_1\equiv 7\pmod 8$. 
 			It follows that $e_{K_1/K_1'}\in \{2,3\}$, more precisely, $e_{K_1/K_1'}=2$ if and only if  $q_1\equiv 7\pmod 8$. 
 		 In the same way as we obtained  Remark	\ref{remqym7mod8}, we deduce the following.
 				\begin{remark}  In the case when  $r\equiv 7\pmod 8$   with $ \left(\frac{q}{r}\right)=1$, it seems that the direct computation of 
 					$\left(\frac{\eta,\,r}{ \mathcal R}\right)$ and $\left(\frac{\eta,\,r}{ \mathcal R'}\right)$ is   difficult. But from the above paragraph 
 					we deduce that 
 					\begin{enumerate}[$a)$]
 						\item If   $q_1\equiv 3\pmod 8$, then   	$\left(\frac{\eta,\,r}{ \mathcal R}\right)=-1$ or $\left(\frac{\eta,\,r}{ \mathcal R'}\right)=-1$, and $\overline{\eta}\not\in \{\overline{-1}, \overline{\vep_{2}}\}$ in the quotient  $E_{K_1'}/(E_{K_1'}\cap N_{K_1/K_1'}(K_1))$. This means that $\left(\frac{\sqrt{\vep_{2q_1q_2}},\,r}{ \mathcal R}\right)=-1$ and $\left(\frac{\sqrt{\vep_{2q_1q_2}},\,r}{ \mathcal R'}\right)=1$.
 						\item 	If   $q_1\equiv 7\pmod 8$,  we have $\overline{\eta} \in \{\overline{1},\overline{-1}, \overline{\vep_{2}}\}$ in the quotient  $E_{K_1'}/(E_{K_1'}\cap N_{K_1/K_1'}(K_1))$. 
 					\end{enumerate}
 				\end{remark}
 		 	\end{enumerate}

 		\item We shall use the above signs of norm  residue symbols.
 		\begin{enumerate}[\rm$\star$]
 			\item Assume that $ \left(\frac{q_1q_2}{s}\right)=\left(\frac{q_1q_2}{r}\right)=-1$. We have $\rg(A(K_1))= 4-1-e_{K_1/K_1'}=3-e_{K_1/K_1'} $.  Moreover,
 		 	\begin{enumerate}[$a)$]
 				
 				\item    If $r\equiv 5\pmod 8$ and $s\equiv 3\pmod 8$, then $\rg(A(K_1))\in \{1,2\}$ and $\rg(A(K_1))=1$ if and only if   $q_1\equiv 3\pmod 8$ or $[q_1\equiv 7\pmod 8$ and $\left(\frac{q_1}{r}\right) =-1  ]$.
 				
 				\item    If $r\equiv 5\pmod 8$ and $s\equiv 5\pmod 8$, then  $\rg(A(K_1))\in \{2,3\}$ and $\rg(A(K_1))=2$ if and only if   $q_1\equiv 3\pmod 8$ or $[q_1\equiv 7\pmod 8$ and $\left(\frac{q_1}{r}\right) =-1$ or $\left(\frac{q_1}{s}\right) =-1]$.

 				\item    If $r\equiv 3\pmod 8$ and $s\equiv 3\pmod 8$, then $\rg(A(K_1))\in \{1,2\}$ and $\rg(A(K_1))=1$ if and only if     $[q_1\equiv 7\pmod 8$ and $\left(\frac{q_1}{r}\right) \not=\left(\frac{q_1}{s}\right)  ]$.
 				
 				\item   If $r\equiv 3\pmod 8$ and $s\equiv 7\pmod 8$, then $\rg(A(K_1))\in \{1,2\}$ and $\rg(A(K_1))=1$ if and only if     $ q_1\equiv 3\pmod 8$. 
 				
 				\item    If $r\equiv 5\pmod 8$ and $s\equiv 7\pmod 8$, then we have  then  $\rg(A(K_1))\in \{2,3\}$ and $\rg(A(K_1))=2$ if and only if   $q_1\equiv 3\pmod 8$ or $[q_1\equiv 7\pmod 8$ and $\left(\frac{q_1}{r}\right) =-1]$.
 				
 				\item   If $r\equiv 5\pmod 8$ and $s\equiv 1\pmod 8$, then we have $\rg(A(K_1))\in \{2,3\}$ and $\rg(A(K_1))=2$ if and only if   $q_1\equiv 3\pmod 8$ or $[q_1\equiv 7\pmod 8$ and $\left(\frac{q_1}{r}\right) =-1]$.
 				
 				\item    If $r\equiv 3\pmod 8$ and $s\equiv 1\pmod 8$, then we have    $\rg(A(K_1))=2$.
 				
 				\item   If $r\equiv 7\pmod 8$ and $s\equiv 7$ or $1\pmod 8$, then we have $\rg(A(K_1))\in \{2,3\}$ and $\rg(A(K_1))=2$ if and only if   $q_1\equiv 3\pmod 8$.

 				\item   If $r\equiv 1\pmod 8$ and $s\equiv 1\pmod 8$, then   $\rg(A(K_1))=3$.
 			\end{enumerate}

 			\item   Assume that  $ \left(\frac{q_1q_2}{r}\right)=-\left(\frac{q_1q_2}{s}\right)=-1$  and $s\not \equiv 1\pmod 8$. Notice that $\rg(A(K_1))=5-e_{K_1/K_1'}$ or $3-e_{K_1/K_1'}$ according to whether $s\equiv 7\pmod 8$ or not. We have:
 			\begin{enumerate}[$a)$]
 				\item  If $    r\equiv 7$ or $1\pmod 8$ and $s\equiv     7\pmod 8$, then  $\rg(A(K_1))\in\{2,3\}$ and $\rg(A(K_1))=2$ if and only if $q_1\equiv 3\pmod 8$.
 				
 				\item If  $r\equiv   s\equiv 5\pmod 8$, then $\rg(A(K_1))\in \{2,3\}$ and $\rg(A(K_1))=2$ if and only if   $q_1\equiv 3\pmod 8$ or   $[q_1\equiv 7\pmod 8$ and $\left(\frac{q_1}{r}\right) =-1$ or $ \left(\frac{q_1}{s}\right)=-1  ]$.
 				\item If  $r\equiv   s\equiv 3\pmod 8$, then $\rg(A(K_1))\in \{1,2\}$ and $\rg(A(K_1))=1$ if and only if   $q_1\equiv 7\pmod 8$ and    $[  \left(\frac{q_1}{r}\right) =-\left(\frac{q_1}{s}\right)=-\left(\frac{q_2}{q_1}\right)=1   $, $   -\left(\frac{q_1}{r}\right) =\left(\frac{q_1}{s}\right)=\left(\frac{q_2}{q_1}\right)=1$,   
 				$-\left(\frac{q_1}{r}\right) =\left(\frac{q_1}{s}\right)=-\left(\frac{q_2}{q_1}\right)=1$ or 
 				$\left(\frac{q_1}{r}\right) =-\left(\frac{q_1}{s}\right)=\left(\frac{q_2}{q_1}\right)=1$
 				$]$.  {This is equivalent to}, $q_1\equiv 7\pmod 8$ and  $\left(\frac{q_1}{r}\right) \not=\left(\frac{q_1}{s}\right)$.
 				
 				\item   If  $r\equiv     3\pmod 8$ and  $   s\equiv 5\pmod 8$, then  $\rg(A(K_1))\in \{1,2\}$ and $\rg(A(K_1))=1$ if and only if   $q_1\equiv 7\pmod 8$ and    $\left(\frac{q_1}{s}\right)=-1  $.

 				\item   If    $ r\equiv     5\pmod 8$ and  $    s\equiv 3\pmod 8$, then  $\rg(A(K_1))\in \{1,2\}$ and $\rg(A(K_1))=1$ if and only if   $q_1\equiv 3\pmod 8$ or    $[q_1\equiv 7\pmod 8$ and $\left(\frac{q_1}{r}\right) =-1]$.

 				\item    If    $ r\equiv     1\pmod 8$ and  $    s\equiv 5\pmod 8$, then $\rg(A(K_1))\in \{2,3\}$ and $\rg(A(K_1))=2$ if and only if       $q_1\equiv 7\pmod 8$ and $\left(\frac{q_1}{s}\right) =-1$.
 				
 				\item    If    $ r\equiv     1\pmod 8$ and  $    s\equiv 3\pmod 8$, then   $\rg(A(K_1))=2$.
 				
 			\end{enumerate}

 			\item   Assume that  $ \left(\frac{q_1q_2}{r}\right)=\left(\frac{q_1q_2}{s}\right)=1$. 	
 			\begin{enumerate}[$a)$]
 				\item   If $r\equiv 7\pmod 8$ and $s\equiv 5 \pmod 8$, then  $\rg(A(K_1))\in \{2,3\}$ and $\rg(A(K_1))=2$ if and only if   $q_1\equiv 3\pmod 8$ or    $[q_1\equiv 7\pmod 8$ and $\left(\frac{q_1}{s}\right) =-1]$.

 				\item  If $r\equiv 7\pmod 8$ and $s\equiv 3 \pmod 8$, then  $\rg(A(K_1))\in \{2,3\}$ and $\rg(A(K_1))=2$ if and only if   $q_1\equiv 3\pmod 8$ or    $ [q_1\equiv 7\pmod 8$ and $\left(\frac{q_1}{r}\right) \not=\left(\frac{q_1}{s}\right)]$.

 				\item    If $r\equiv 3\pmod 8$ and $s\equiv 3\pmod 8$, then  $\rg(A(K_1))\in \{1,2\}$ and $\rg(A(K_1))=1$ if and only if       $ q_1\equiv 7\pmod 8$ and $\left(\frac{q_1}{r}\right) \not=\left(\frac{q_1}{s}\right) $.

 				\item   If $r\equiv 5\pmod 8$ and $s\equiv 5\pmod 8$, then     $\rg(A(K_1))\in \{2,3\}$ and $\rg(A(K_1))=2$ if and only if       $ q_1\equiv 7\pmod 8$ and $[\left(\frac{q_1}{r}\right)=-1$ or $  \left(\frac{q_1}{s}\right)=-1 ]$. 
 				
 				\item    If $r\equiv 5\pmod 8$ and $s\equiv 3\pmod 8$, then   $\rg(A(K_1))\in \{1,2\}$ and $\rg(A(K_1))=1$ if and only if       $ q_1\equiv 7\pmod 8$ and $\left(\frac{q_1}{r}\right)=-1$. 
 				\item    If $r\equiv 7\pmod 8$ and $s\equiv 7\pmod 8$, then it is clear that   $\rg(A(K_1))=7-e_{K_1/K_1'}\geq 3$.
 			\end{enumerate}
 			
 		\end{enumerate}

 		\item Now assume that we are in the conditions of the third item. Thus, $\rg(A(K_1))=6-1-e_{K_1/K_1'}=5-e_{K_1/K_1'}\leq 5$. 
 		By means of the signs of norms residue symbols that we have in the proof of the first item, we check that $e_{K_1/K_1'}\leq  2$. So the result.

 	\end{enumerate}	
  \end{proof}

 \begin{remark}Let $K$ be a real biquadratic field of the form $C)$ with $K\not= L$.
 	By combining  Lemma  \ref{realTruquad2-rankq1q1}, Lemma \ref{lemmaq1q2Bd=1mod4} and Lemma \ref{lm fukuda}, we get that  $ \rg(A(K_\infty))=\rg( A(K))$ if and only if $K$ is taking one of the forms in  the items $23)$, $24)$,  $25)$,  $26)$,  $27)$, $28)$ and   $29)$ of the main theorem. 
 \end{remark}
 This completes the proof of the main theorem.

 	\section{\bf Investigation of the structure of the $2$-Iwasawa module}\label{sec4}

 The computation of the $2$-class number of real biquadratic fields   $F=\mathbb{Q}(\sqrt{d_1},\sqrt{d_2})$   is reachable, especially when the 
 discriminant of $F$ is divisible by two, three or even four   prime numbers and the $2$-class number of the quadratic subfields is small $($i.e. $\leq 16$$)$. In fact, one can 
    use the Kuroda's formula 	$h_2(F)=\frac{1}{4}q(F)h_2(d_1)h_2(d_2)h_2(d_1d_2)$. For   values of $h_2(d_1)$, $h_2(d_2)$ and $h_2(d_1d_2)$ one can refer to \cite{kaplan76,connor88,kaplan1973divisibilitepar8} and the references listed therein. Moreover, for 
     the   computation of unit index $q(F)$ one can use the Wada's method (see page \pageref{wada's f.} above) and proceed as in   \cite{AzmouhcapitulationActaArith} or \cite{AzRezzouguiinternationalj}.
     Therefore, the following lemma highlights  an effective procedure for computing the size of the Iwasawa module of biquadratic fields.

 \begin{lemma}
 	Let $K$ be a real biquadratic field such that   $K_\infty/K$ is totally ramified  and  $K_1$ has an abelian $2$-tower $($or in particular $\rg(A(K_1)) =1$$)$. 
 	Let $F$ be a real biquadratic field such that      $K_1/F$ is a quadratic unramified extension.   
 	Then $|A(K_\infty)|=\frac12h_2(F)$ if and only if $ F^{(1)}=F^{(2)}$ and $h_2(K)=\frac12h_2(F)$. 
 \end{lemma}
  \begin{proof}
 	Notice that $K_1/F$ is  unramified and 	$ A(K_1) $ is cyclic. Then by class field theory, we have the following diagram: 
 	\[\xymatrix@R=1cm@C=0.2cm{
 		&F &  &K_1 \ar@{-}[ll]&  &F^{(1)} \ar@{-}[ll]\ar@/_1pc/@{-}[llll]_{A(F)}& & &\mathcal{L}(F)=K_1^{(1)}\ar@{-}[lll]\ar@/^1.5pc/@{-}[lllll]^{A(K_1)}\\
 	}\]
 	Thus,  the Hilbert $2$-class field tower of $F$ is abelian if and only if  	$h_2(K_1)=\frac12h_2(F)$. If furthermore, 	$h_2(K)=\frac12h_2(F)$, then we have the result by Fukuda's theorem.
 	\end{proof}
 \bigskip	
  In the light of the main theorem (i.e. Theorem \ref{maintheorem}) we have the following results.
   \bigskip

 \begin{theorem}
 	  Let $K=\mathbb{Q}(\sqrt{q},\sqrt{rs})$ where $r\equiv s \equiv   q\equiv3\pmod 4$ are three different prime numbers. Then $A(K_\infty)\simeq \ZZ/2 \ZZ $ if and only if, after a permutation of $r$ and $s$,  we have one of the following conditions:
 		\begin{enumerate}[\indent$\C1:$]
 			\item $r\equiv    q\equiv 3\pmod 8$,  $s\equiv    7\pmod 8$   and   $\left(\frac{q}{s}\right)=\left(\frac{q}{r}\right)=-1= \left(\frac{s}{r}\right)$,
 			
 			\item $r\equiv    s\equiv 3\pmod 8$,  $q\equiv    7\pmod 8$ and   $\left(\frac{q}{r}\right)=-\left(\frac{q}{s}\right)=-1$,
 			
 				\item	$s\equiv    q\equiv 3\pmod 8$,  $r\equiv    7\pmod 8$   and   $\left(\frac{q}{r}\right)=-\left(\frac{q}{s}\right)=-1=\left(\frac{r}{s}\right)$,
 		 \item	$r\equiv    q\equiv 3\pmod 8$,  $s\equiv    7\pmod 8$   and   $\left(\frac{q}{r}\right)=-\left(\frac{q}{s}\right)=-1=\left(\frac{r}{s}\right)$,
 				
 			\item	$s\equiv    q\equiv 3\pmod 8$,  $r\equiv    7\pmod 8$   and   $\left(\frac{q}{s}\right)=\left(\frac{q}{r}\right)=1=\left(\frac{r}{s}\right)$.
 	 	\end{enumerate}
  	Moreover, under each of these conditions, we have $\mu_K=\lambda_K=0$ and $\nu_K=1$.
  \end{theorem}
 \begin{proof}
    It suffices to focus on the forms appearing in the items $2)$, $3)$, $4)$ and $5)$ of the main theorem.
 		Assume that 	$r\equiv     3\pmod 8$ and  $s\equiv    7\pmod 8$   are two prime numbers such that $\left(\frac{q}{s}\right)=\left(\frac{q}{r}\right)=-1$
 		and $q\equiv 3\pmod 8$.  We have $K_1/F$ is unramified, where $F=\mathbb{Q}(\sqrt{2},\sqrt{qrs})$. From 
 		\cite[Theorem 11]{AzmouhcapitulationActaArith} and \cite[Proposition 8 and   its proof]{AzmouhcapitulationActaArith}, we deduce that  $h_2(F)=4$ and  $ F^{(1)}=F^{(2)}$ if and only if $\left(\frac{s}{r}\right)=-1$. In this case,   
 		$h_2(K)=2$. So  $A(K_\infty)\simeq \ZZ/2 \ZZ \text{ if and only if } \left(\frac{s}{r}\right)=-1$. We similarly prove the other cases.
 \end{proof}

 \begin{remark}\label{remh2F}
 	 Let     $r\equiv q\equiv     3\pmod 8$ and  $s\equiv    7\pmod 8$   be three prime numbers such that $\left(\frac{q}{s}\right)=\left(\frac{q}{r}\right)=-1=-\left(\frac{s}{r}\right)$.
 	 	 Put $F=\mathbb{Q}(\sqrt{2},\sqrt{qrs})$. Then 	$h_2(F)=\frac{1}{4}q(F)h_2(qrs)h_2({2qrs})$.
 	 	 Under our assumptions, we have $h_2(2qrs)=4$ and $h_2(qrs)$ is divisible by $8$ (cf. \cite[p. 40]{AzmouhcapitulationActaArith}). Therefore,
 	 	 $h_2(F)=2^m$ if and only if $q(F)=1$ and $h_2(qrs)=2^m$.
 \end{remark}

 \begin{theorem}\label{THMAinfy=2m}
   Let $K=\mathbb{Q}(\sqrt{q},\sqrt{rs})$,   where     $r\equiv q\equiv     3\pmod 8$ and  $s\equiv    7\pmod 8$   be three prime numbers such that $\left(\frac{q}{s}\right)=\left(\frac{q}{r}\right)=-1=-\left(\frac{s}{r}\right)$. Let $a$   and $b$  be the integers such that
 	$ \varepsilon_{qrs}=a+b\sqrt{qrs}$.
 	Assume that   $s(a-1)$  is not a square in $\NN $.    Then  
 	$$A(K_\infty)\simeq A(K)\simeq\ZZ/2^{m-1} \ZZ ,$$
 	where $m$ is the positive integer such that  $h_2(qrs)=2^m$.   	Therefore,  we have $\mu_K=\lambda_K=0$ and $\nu_K=m-1$.
  \end{theorem}

To prove this theorem, we need the following preparatory  lemmas.

\begin{lemma}[\cite{Az-00}, Lemme 5]\label{lem2}
	Let $d>1$ be a square free integer and $\varepsilon_d=x+y\sqrt d$, where $x$, $y$ are  integers or semi-integers. If $N(\varepsilon_d)=1$, then $2(x+1)$, $2(x-1)$, $2d(x+1)$ and $2d(x-1)$ are not squares in $\QQ$.
\end{lemma}

\begin{lemma}\label{lem0}
Let $r\equiv q\equiv     3\pmod 8$ and  $s\equiv    7\pmod 8$   be   prime numbers such that $\left(\frac{q}{r}\right)=\left(\frac{q}{s}\right)=-1=-\left(\frac{s}{r}\right)$.
 
	\begin{enumerate}[\rm1)]
		\item Let $a$   and $b$  be the integers such that
		$ \varepsilon_{qrs}=a+b\sqrt{qrs}$. Then
		out of the three numbers   
		 $2q(a-1)$, $r(a-1)$ or $s(a-1)$, exactly one of them is   a square in $\NN $. Furthermore, we have:
		\begin{enumerate}[$ a)$]
			\item  If $2q(a-1)$ is a square in $\NN $, then $\sqrt{\varepsilon_{qrs}}= b_1\sqrt{q} +b_2\sqrt{rs}$ and $1=-qb_1^2+rsb_2^2$, where $b_1$ and $b_2$ are two integers such that $b=2b_1b_2$. 
			\item  If $r(a-1)$ is a square in $\NN $, then $\sqrt{2\varepsilon_{qrs}}= b_1\sqrt{r} +b_2\sqrt{qs}$ and $2=-rb_1^2+qsb_2^2$, where $b_1$ and $b_2$ are two integers such that $b=b_1b_2$.

			\item If $s(a-1)$ is a square in $\NN $, then $\sqrt{2\varepsilon_{qrs}}= b_1\sqrt{s} +b_2\sqrt{qr}$ and $2=-sb_1^2+qrb_2^2$, where $b_1$ and $b_2$ are two integers such that $b=b_1b_2$. 
		\end{enumerate}
 
		\item Let $x$ and $y$   be the integers such that
		  $\varepsilon_{2qrs}=x+y\sqrt{2qrs}$. Then $s(x-1)$  is a square in $\NN $. Moreover, we have:
		 $$\sqrt{2\varepsilon_{2qrs}}= y_1\sqrt{s} +y_2\sqrt{2qr}  \text{ and } 2=-sy_1^2+2qry_2^2.$$
	 \end{enumerate}
\end{lemma}	 

\begin{proof}
		\begin{enumerate}[\rm1)]
		\item 
	As $ N(\varepsilon_{qrs})=1 $, then by the unique factorization  of $ a^{2}-1=qrsb^{2} $ in $ \mathbb{Z} $, and Lemma \ref{lem2}, there exist $b_1$ and $b_2$ in $\mathbb{Z}$ such that  we have exactly one of the following systems:
	$$(1):\ \left\{ \begin{array}{ll}
		a\pm1=2qb_1^2\\
		a\mp1=2rsb_2^2,
	\end{array}\right.  \quad
	(2):\ \left\{ \begin{array}{ll}
		a\pm1=2rb_1^2\\
		a\mp1=2qsb_2^2,
	\end{array}\right. \quad
	(3):\ \left\{ \begin{array}{ll}
		a\pm1=2sb_1^2\\
		a\mp1=2qrb_2^2,
	\end{array}\right. 
	$$

	$$   
	(4):\ \left\{ \begin{array}{ll}
		a\pm1=qb_1^2\\
		a\mp1=rsb_2^2,
	\end{array}\right. \quad (5):\ \left\{ \begin{array}{ll}
	a\pm1=rb_1^2\\
	a\mp1=qsb_2^2
\end{array}\right.\text{ or }(6):\ \left\{ \begin{array}{ll}
a\pm1=sb_1^2\\
a\mp1=qrb_2^2.
\end{array}\right.
	$$

	\begin{enumerate}[\rm$\bullet$]
	\item  Assume that the system $\left\{ \begin{array}{ll}
		a+1=2qb_1^2\\
		a-1=2rsb_2^2,
	\end{array}\right.$ holds. We have:	
	\[1=\left(\dfrac{2qb_1^2}{r}\right)=\left(\dfrac{a+1}{r}\right)=\left(\dfrac{a-1+2}{r}\right)=\left(\dfrac{2rsb_2^2+2}{r}\right)=\left(\dfrac{2}{r}\right)=-1,
	\]
	which is absurd. So this system can not hold.
	
	\item Assume that the system $\left\{ \begin{array}{ll}
		a+1=qb_1^2\\
		a-1=rsb_2^2 
	\end{array}\right.$ holds. We have:	
	\[-1=\left(\dfrac{qb_1^2}{s}\right)=\left(\dfrac{a+1}{s}\right)=\left(\dfrac{a-1+2}{s}\right) =\left(\dfrac{2}{s}\right)=1,
	\]
	which is absurd. So this system also can not hold.
 \end{enumerate}
We similarly eliminate all systems except the following:
$$(1):\ \left\{ \begin{array}{ll}
	a-1=2qb_1^2\\
	a+1=2rsb_2^2,
\end{array}\right.  \quad \left\{ \begin{array}{ll}
a-1=rb_1^2\\
a+1=qsb_2^2
\end{array}\right. \quad \text{ and } \left\{ \begin{array}{ll}
a-1=sb_1^2\\
a+1=qrb_2^2.
\end{array}\right.$$
Suppose that the left system holds, i.e. $2q(a-1)$ is a square in $\NN$. So by summing and   subtracting the two equalities, 
we get $a=qb_1^2+ rsb_2^2$ and    $1=-qb_1^2+rsb_2^2$. Thus, 
$\varepsilon_{qrs}=a+b\sqrt{qrs}=qb_1^2+ rsb_2^2+2b_1b_2\sqrt{q}\sqrt{rs}$ and so $\varepsilon_{qrs}=( b_1\sqrt{q} +b_2\sqrt{rs})^2$.
Therefore, we have $1)$-$a)$. We similarly, we get the rest of our first item. 
\item For the second item, we proceed similarly and eliminate all systems except $\left\{ \begin{array}{ll}
	x-1=sy_1^2\\
	x+1=2qry_2^2,
\end{array}\right. $ where $y_1$ and $y_2$ are two integers such that $y=2y_1y_2$. This  gives the second item.
\end{enumerate}
 \end{proof}

\begin{corollary}
	The unit group of  $F=\mathbb{Q}(\sqrt{2},\sqrt{qrs})$ is $\{\varepsilon_2, \varepsilon_{qrs}, \sqrt{ \varepsilon_{qrs} \varepsilon_{2qrs} }\}$ or $\{\varepsilon_2, \varepsilon_{qrs}, \varepsilon_{2qrs}\}$ according to whether  $s(a-1)$  is a square in $\NN $ or not.
\end{corollary}

\begin{remark}\label{remhF}To link this corollary with    Remark \ref{remh2F}, we have $q(F)=1$ if and only if  $s(a-1)$  is not a square in $\NN $. Hence, 
	$h_2(F)=2^m$ if and only if $s(a-1)$  is not a square in $\NN $ and $h_2(qrs)=2^m$.
	\end{remark}

\begin{lemma}\label{lemmunitstriquad}
	Let $r\equiv q\equiv     3\pmod 8$ and  $s\equiv    7\pmod 8$   be   prime numbers such that $\left(\frac{q}{s}\right)=\left(\frac{q}{r}\right)=-1=-\left(\frac{s}{r}\right)$.
 Let $a$   and $b$  be the integers such that
$ \varepsilon_{qrs}=a+b\sqrt{qrs}$.
Assume that   $s(a-1)$  is not a square in $\NN $. Then the unit group of $\FF'=\mathbb{Q}(\sqrt{2},\sqrt{qr}, \sqrt{s})$ is
 $$E_{\FF'}=\langle-1,   \varepsilon_2,  \varepsilon_{qr}     ,\sqrt{\varepsilon_{s}}, \sqrt{\varepsilon_{2s}}, \sqrt{\varepsilon_{2qr}},\sqrt{\varepsilon_{qr}\varepsilon_{qrs}}, \sqrt{ \varepsilon_{2qrs}} \rangle.$$
 Thus, $q(\FF')=2^5$ and $h_2(\FF')=\frac{1}{2}h_2(qrs)$.
\end{lemma}	 
 \begin{proof}
Put $k_1=\mathbb{Q}(\sqrt{2},  \sqrt{s})$, $k_2=\mathbb{Q}(\sqrt{2},  \sqrt{qr})$ and $k_3=\mathbb{Q}(\sqrt{2},  \sqrt{qrs})$. The fundamental system   of units of 
$k_1$, $k_2$ and $k_3$ are $\{\varepsilon_2,\sqrt{\varepsilon_{s}}, \sqrt{\varepsilon_{2s}}\}$, $\{\varepsilon_2, \varepsilon_{qr}, \sqrt{\varepsilon_{2qr}}\}$ and $\{\varepsilon_2, \varepsilon_{qrs}, \varepsilon_{2qrs}\}$ respectively. 
Let $\tau_1$, $\tau_2$ and $\tau_3$ be the elements of  $ \mathrm{Gal}(\FF'/\QQ)$ defined by
\begin{center}	\begin{tabular}{l l l }
		$\tau_1(\sqrt{2})=-\sqrt{2}$, \qquad & $\tau_1(\sqrt{qr})=\sqrt{qr}$, \qquad & $\tau_1(\sqrt{s})=\sqrt{s},$\\
		$\tau_2(\sqrt{2})=\sqrt{2}$, \qquad & $\tau_2(\sqrt{qr})=-\sqrt{qr}$, \qquad &  $\tau_2(\sqrt{s})=\sqrt{s},$\\
		$\tau_3(\sqrt{2})=\sqrt{2}$, \qquad &$\tau_3(\sqrt{qr})=\sqrt{qr}$, \qquad & $\tau_3(\sqrt{s})=-\sqrt{s}.$
	\end{tabular}
\end{center}
Note that  $\mathrm{Gal}(\FF'/\QQ)=\langle \tau_1, \tau_2, \tau_3\rangle$
and the subfields  $k_1$, $k_2$ and $k_3$ are
fixed by  $\langle \tau_2\rangle$, $\langle\tau_3\rangle$ and $\langle\tau_2\tau_3\rangle$ respectively. Therefore,\label{fsu preparations} a fundamental system of units  of $\FF'$ consists  of seven  units chosen from those of $k_1$, $k_2$ and $k_3$, and  from the square roots of the elements of $E_{k_1}E_{k_2}E_{k_3}$ which are squares in $\FF'$ (see Page \pageref{algo wada}). With these notations, we have:
$$E_{k_1}E_{k_2}E_{k_3}=\langle-1,   \varepsilon_2,  \varepsilon_{qr},   \varepsilon_{qrs}, \varepsilon_{2qrs} ,\sqrt{\varepsilon_{s}}, \sqrt{\varepsilon_{2s}}, \sqrt{\varepsilon_{2qr}} \rangle.$$	
Thus we shall determine elements of $E_{k_1}E_{k_2}E_{k_3}$ which are squares in $\FF'$. Let  $\xi$ be an element of $\FF'$ which is a  square root of an element of $E_{k_1}E_{k_2}E_{k_3}$. We can assume that
$$\xi^2=\varepsilon_2^a  \varepsilon_{qr}^b   \varepsilon_{qrs}^c \varepsilon_{2qrs}^d\sqrt{\varepsilon_{s}}^e \sqrt{\varepsilon_{2s}}^f \sqrt{\varepsilon_{2qr}}^g,$$
where $a, b, c, d, e, f$ and $g$ are in $\{0, 1\}$.
By \cite[Lemma 3.2]{CEH2024} and \cite[Lemma 4]{AzchemsTrends}, we have: 

$$ \sqrt{\varepsilon_{s}}=\frac{1}{\sqrt{2}}(\beta_1+\beta_2\sqrt{s} ) \text{ and } 2=\beta_1^2- s\beta_2^2,    $$
$$\sqrt{\varepsilon_{2s}}=\frac{1}{\sqrt{2}}(d_1+d_2\sqrt{2s} ) \text{ and } 2=d_1^2-2sd_2^2   $$
and
$$ \sqrt{\varepsilon_{2qr}}=\frac{1}{\sqrt{2}}(y_1+y_2\sqrt{2qr} )  \text{ and } 2=-y_1^2+ 2qry_2^2. $$
for some integers $\beta_1$, $\beta_2$, $d_1$, $d_2$, $y_1$ and $y_2$.

\noindent\ding{224} By the above equalities, we have  
  $\sqrt{\varepsilon_{s}}^{1+ \tau_2}=\varepsilon_{s}$,  $\sqrt{\varepsilon_{2s}}^{1+\tau_2}=\varepsilon_{2s}$ and 
$\sqrt{\varepsilon_{2qr}}^{1+ \tau_2}=-1$.  Thus, by applying the norm $N_{\FF'/k_1}=1+ \tau_2$, we get:
\begin{eqnarray*}
	N_{\FF'/k_1}(\xi^2)=N_{\FF'/k_1}(\xi)^2&=&\varepsilon_{2}^{2a}\cdot 1\cdot 1\cdot1\cdot\varepsilon_{s}^{e}\cdot \varepsilon_{2s}^{f}\cdot (-1)^g.
\end{eqnarray*}
Since the left   side of the equality is positive, this implies that $g=0$. Thus, 
$$\xi^2=\varepsilon_2^a  \varepsilon_{qr}^b   \varepsilon_{qrs}^c \varepsilon_{2qrs}^d\sqrt{\varepsilon_{s}}^e \sqrt{\varepsilon_{2s}}^f,$$

\noindent\ding{224} Let $k_4=\mathbb{Q}(\sqrt{s},\sqrt{qr})$. By applying the norm $N_{\mathbb F'/k_4}=1+ \tau_1$, we get
\begin{eqnarray*}
	N_{\FF'/k_4}(\xi^2) &=&(-1)^{a}\cdot \varepsilon_{qr}^{2b} \cdot   \varepsilon_{qrs}^{2c}\cdot 1\cdot1\cdot(-1)^{e}\cdot\varepsilon_{s}^{e}\cdot (-1)^{f},\\
	&=& \varepsilon_{qr}^{2b}\varepsilon_{qrs}^{2c} (-1)^{a+e+f}\varepsilon_{s}^{e}.
\end{eqnarray*}
Thus  $a+e+f\equiv 0\pmod2$. Since $\varepsilon_{s}$ is not a square in $k_4$, we have $e=0$ and so $a=f$. 
$$\xi^2=\varepsilon_2^a  \varepsilon_{qr}^b   \varepsilon_{qrs}^c \varepsilon_{2qrs}^d  \sqrt{\varepsilon_{2s}}^a.$$

\noindent\ding{224} Let $k_5=\mathbb{Q}(\sqrt{2s},\sqrt{qr})$. By applying the norm $N_{\FF'/k_5}=1+ \tau_1\tau_3$, we get 
\begin{eqnarray*}
	N_{\FF'/k_5}(\xi^2) &=&(-1)^{a}\cdot \varepsilon_{qr}^{2b} \cdot 1\cdot \varepsilon_{2qrs}^{2d}\cdot (-1)^{a}\cdot\varepsilon_{2s}^{a},\\
	&=& \varepsilon_{qr}^{2b}\varepsilon_{2qrs}^{2d}  \varepsilon_{2s}^{a}.
\end{eqnarray*}
As $\varepsilon_{2s}$ is not a square in $k_5$,   $a=0$. Therefore,
 $$\xi^2=  \varepsilon_{qr}^b   \varepsilon_{qrs}^c \varepsilon_{2qrs}^d.$$
Notice that by  \cite[Lemma 4]{AzchemsTrends}, we have
 $$\sqrt{\varepsilon_{qr}}=\alpha_1\sqrt{q}+\alpha_2\sqrt{r},$$ 
 and by Lemma \ref{lem0}, we have 
\begin{enumerate}[$\star$]
	\item If $2q(a-1)$ is a square in $\NN $, then $\sqrt{\varepsilon_{qrs}}= b_1\sqrt{q} +b_2\sqrt{rs}$, 
	\item If $r(a-1)$ is a square in $\NN $, then $\sqrt{ \varepsilon_{qrs}}= \frac{1}{\sqrt{2}}(b_1\sqrt{r} +b_2\sqrt{qs})$,
\end{enumerate}    
   and  $\sqrt{ \varepsilon_{2qrs}}= \frac{1}{\sqrt{2}}(y_1\sqrt{s} +y_2\sqrt{2qr})$, 
   where $b_1$, $b_2$, $y_1$, $y_2$, $ \alpha_1$ and $\alpha_2$ are  integers.
Thus $\sqrt{\varepsilon_{qr}}$, $\sqrt{\varepsilon_{qrs}}\not\in \FF'$ and $\sqrt{\varepsilon_{qr}\varepsilon_{qrs}}$, $\sqrt{ \varepsilon_{2qrs}} \in \FF'$. It follows that 
 $$E_{\FF'}=\langle-1,   \varepsilon_2,  \varepsilon_{qr}     ,\sqrt{\varepsilon_{s}}, \sqrt{\varepsilon_{2s}}, \sqrt{\varepsilon_{2qr}},\sqrt{\varepsilon_{qr}\varepsilon_{qrs}}, \sqrt{ \varepsilon_{2qrs}} \rangle.$$

Thus, by class number formula and  \cite[Corollaries 18.4 and 19.7]{connor88}, we have $h_2(\FF')=\frac{1}{2}h_2(qrs)$.

\end{proof}

\begin{proof}[Proof of Theorem \ref{THMAinfy=2m}]
 As the rank of the $2$-class group of $F=\mathbb{Q}(\sqrt{2},\sqrt{qrs})$ equals $2$,  under our assumptions, we have $A(F)\simeq \ZZ/2^{\delta_1} \ZZ\times\ZZ/2^{\delta_2} \ZZ$, where $\delta_1+\delta_2=m$. Notice that the quadratic unramified extensions of $F$ are $\FF'=\mathbb{Q}(\sqrt{2},\sqrt{qr}, \sqrt{s})$, $\FF''=K_1=\mathbb{Q}(\sqrt{2},\sqrt{q}, \sqrt{rs})$ and $\FF'''=\mathbb{Q}(\sqrt{2},\sqrt{r}, \sqrt{qs})$. By Lemma \ref{realTruquad2-rankq}, the $2$-class groups of $\FF'$ and $\FF'''$ are cyclic and that of $\FF'$ is of rank $2$. As by Lemma \ref{lemmunitstriquad}, $h_2(\FF')=\frac{1}{2}h_2(qrs)=2^{m-1}$, by Proposition \ref{LemBenjShn}, we have $ F^{(1)}=F^{(2)}$. On the other hand, by means of the class number formula, \cite[Corollary 18.4]{connor88},  and Lemma \ref{lem0}, we check that 
$h_2(K)=\frac{1}{4}q(K)h_2(q)h_2(rs)h_2({qrs})=\frac{1}{4}\cdot2\cdot 1\cdot 1 \cdot2^m=2^{m-1}=h_2(F)$. So the result.	
 \end{proof}
\begin{corollary}
	Keep the same hypothesis of Theorem \ref{THMAinfy=2m}. According to  the above proof and Proposition \ref{LemBenjShn}, we have  
	 $$h_2(F'')=h_2(F''')=2^{m-1}.$$
\end{corollary}

 \bigskip
	
	The previous investigations are also useful  to get  the following result.
	
	\begin{theorem}\label{THMAinfy=22}
		Let $K=\mathbb{Q}(\sqrt{q},\sqrt{rs})$,   where     $r\equiv s\equiv     3\pmod 8$ and  $q\equiv    7\pmod 8$   are three prime numbers such that $\left(\frac{q}{s}\right)=\left(\frac{q}{r}\right)=1$. Let $a$   and $b$  be the integers such that
		$ \varepsilon_{qrs}=a+b\sqrt{qrs}$. 	Assume that   $q(a-1)$  is not a square in $\NN $.   Then, we have:
			$$A(K_\infty)\simeq A(K)\simeq\ZZ/2 \ZZ\times\ZZ/2^{m-2} \ZZ .$$
		where $m$ is the positive integer such that  $h_2(qrs)=2^m$.   	Therefore,  we have $\mu_K=\lambda_K=0$ and $\nu_K=m-1$.
	\end{theorem}
 	\begin{proof}
	Consider the real quadratic field    $L=\QQ( \sqrt{qrs})$.
	 	Since $\rg(A(L))=	\rg(A(L_1))=2$ (cf. \cite[Theorem 3.1]{Azmouh2-rank}), Fukuda's theorem implies that   for all $n\geq 0$, $ \rg(A(L_n))=2$. Put 
	 $K= \mathbb{Q}( \sqrt{q}, \sqrt{rs})$,	$K'=\mathbb{Q}( \sqrt{qr}, \sqrt{s})$  and $K''=\mathbb{Q}( \sqrt{r}, \sqrt{qs})$. Notice that for all  $n\geq 0$, 
	$K_n$, $K'_n$ and $K''_n$ (i.e. the $n$th layer of the cyclotomic $\ZZ_2$-extension of $K$, $K'$ and $K''$ respectively) are the  three  unramified quadratic extensions of $L_n$.

	Using Lemma \ref{lem0} and Page \ref{sqrtepsiq1q2}, we check that $q(K)=2$ (in fact,   $\sqrt{\varepsilon_{rs}\varepsilon_{qrs}}\in K$ or $\sqrt{\varepsilon_{q}\varepsilon_{qrs}}\in K)$. Thus, by the class number formula and \cite[Corollary 18.4]{connor88}, we have 
	$h_2(K)=\frac{1}{4}q(K)h_2(q)h_2(rs)h_2({qrs})=\frac{1}{4}\cdot2\cdot 1\cdot 1 \cdot h_2({qrs})=\frac{1}{2}h_2({qrs})$. 
	It follows by Lemma \ref{lemmunitstriquad}, $h_2(K)=h_2(K_1)$ (here it is taken into account the change of places of       $q$, $r$ and $s$).
	So, from Remark \ref{remhF},
	we deduce that $h_2(L_1)= h_2(qrs)=2^m=\frac12h_2(K_1)=\frac12h_2(K)$. Thus, by Fukuda's theorem, we have for all $n\geq 0$,  $h_2(L_n)=\frac12h_2(K_n)$.
 It follows that, for all $n\geq 0$,     the Hilbert $2$-class field tower of $L_n$ stops at the first layer  
			(cf. Proposition \ref{LemBenjShn}), i.e. $G_{L_n}$ is abelian.
		Let $n$ be   such that $n\geq0$.	Assume that $4$-$\mathrm{rank}(A(L_n))=2$. 
			Thus, there exist  (cf. \cite[Lemma 1]{Ben17}) $a$ and $b$   such that  $G_{L_n} =\langle a,b \rangle$ and such that the three subgroups of $G_{L_n}$ of index $2$ are 
				\begin{center}
				$H_{1}=\langle a, b^2, G_{L_n}'\rangle=\langle a, b^2 \rangle$, $H_{2}=\langle ab, b^2, G_{L_n}'\rangle=\langle ab, b^2 \rangle$ and
				$H_{3}=\langle a^2, b, G_{L_n}'\rangle=\langle a^2, b  \rangle.$
			\end{center}
		Thus,       $\rg(A(K'_n))=	\rg(A(K''_n))=2$, which is a contradiction. In fact,   according to the main theorem, we have $A(K'_n)	$ and $A(K''_n)$ are cyclic. 
			So    $4$-$\mathrm{rank}(A(L_n))=1$. Therefore, for all $n\geq0$,  $A(L_n)$ is isomorphic to $\ZZ/2 \ZZ\times\ZZ/2^{m-1} \ZZ$.
			Hence, for all $n\geq0$,  $A(K_n)$ is isomorphic to $\ZZ/2 \ZZ\times\ZZ/2^{m-2} \ZZ$   (cf. Theorem \ref{AabounePrzekhini}), which completes the proof.
	\end{proof}

 	We close the paper with the following corollary and remark that follow  from the above theorem and its proof.
 
 \begin{corollary}\label{fincorol} Keep the same hypothesis as in the above theorem and 
 	consider the real quadratic field    $L=\QQ( \sqrt{qrs})$. Then, we have:
 $$A(L_\infty)\simeq A(L)\simeq\ZZ/2 \ZZ\times\ZZ/2^{m-1} \ZZ .$$
 		where $m$ is the positive integer such that  $h_2(qrs)=2^m$.   	Therefore,  we have $\mu_L=\lambda_L=0$ and $\nu_L=m$.
 \end{corollary}
	
	The previous results give the following constructions.
	\begin{remark}
		There exists an infinite family of  real biquadratic number fields (resp.   real quadratic number fields)  $k$ whose   
		 $2$-Iwasawa module  is of rank equals $2$ and of $4$-rank  equals $1$, and $\mathrm{Gal}(\mathcal{L}(k_\infty)/k_\infty)$, the Galois group of the maximal unramified pro-$2$-extension $\mathcal{L}(k_\infty)$ of $k_\infty$, is abelian.
		\end{remark}

	\section*{Acknowledgment}

I sincerely appreciate the reviewer's insightful comments and questions, which helped correct some errors and add missing details, greatly enhancing the quality of this paper.

I would also like to thank Katharina M\"uller for   her useful comments on the preliminary version of this paper.

\end{document}